\documentclass{article}

\usepackage[margin=1in]{geometry}

\usepackage{amssymb}
\usepackage{amsmath}
\usepackage{amsthm}
\usepackage{enumitem}
\setlist{nolistsep}
\usepackage{bm}
\usepackage{tikz}
\usetikzlibrary{arrows.meta}
\usetikzlibrary{decorations} 
\usetikzlibrary{snakes}
\usetikzlibrary{shapes}
\usepackage[margin=2cm]{caption}
\usepackage{subcaption}
\usepackage{titlesec}
\usepackage[backref]{hyperref}
\usepackage{graphicx}
\usepackage{xcolor}
\usepackage{thmtools}
\usepackage{thm-restate}
\usepackage{cleveref}

 \setlist{  
  listparindent=\parindent,
  parsep=0pt,
}

\newtheorem{theorem}{Theorem}[section]

\newtheorem{lemma}[theorem]{Lemma}
\newtheorem*{claim*}{Claim}

\newtheorem{cor}[theorem]{Corollary}
\newtheorem{conjecture}[theorem]{Conjecture}
\newtheorem{prop}[theorem]{Proposition}

\theoremstyle{definition}

\newenvironment{subproof}[1][Proof]{\begin{proof}[#1]}{\end{proof}}

\newcommand{\hd}{{\textup{\texttt{head}}}}
\newcommand{\tl}{{\textup{\texttt{tail}}}}

\newcommand{\PP}{\mathrm{PP}}
\newcommand{\CDG}{\mathrm{CDG}}
\newcommand{\supp}{\mathrm{supp}}
\newcommand{\vodd}{v_{\mathrm{odd}}}
\newcommand{\ennd}{\mathrm{end}}

\tikzstyle{knode}=[circle,draw=black,thick,inner sep=2pt]
\tikzstyle{klet}=[circle,draw=black,thick,inner sep=2pt]
\tikzstyle{kfill}=[circle,draw=black,thick,inner sep=2pt,fill=black]
\tikzstyle{klabel}=[rectangle,draw=black,thick,inner sep=2pt,align=center]

\title{Improved Decomposition Bounds for Partition Polytopes and Odd-Covers}

\author{Steffen Borgwardt\thanks{\texttt{steffen.borgwardt@ucdenver.edu}; University of Colorado Denver, Department of Mathematical and Statistical Sciences} \and Zden\v{e}k Dvo\v{r}\'{a}k\thanks{\texttt{rakdver@iuuk.mff.cuni.cz}; Charles University, Prague} \and Bryce Frederickson\thanks{\texttt{bfrede4@emory.edu}; Emory University, Department of Mathematics} \and Abigail Nix\thanks{\texttt{abigail.nix@ucdenver.edu};  University of Colorado Denver, Department of Mathematical and Statistical Sciences} \and Youngho Yoo\thanks{\texttt{yyoo.math@gmail.com}; Texas A\&M University, Department of Mathematics}}

\date{}

\begin{document}

\maketitle

\begin{abstract}
    The assignments of a set of $m$ items into $n$ clusters of prescribed sizes $\kappa_1,\dots,\kappa_n$ can be encoded as the vertices of the \emph{partition polytope} $\mathrm{PP}(\kappa_1,\dots,\kappa_n)$. We prove that, if $K = \max\{\kappa_1,\dots,\kappa_n\}$, then the combinatorial diameter of $\mathrm{PP}(\kappa_1,\dots,\kappa_n)$ is at most $\lceil 3K/2\rceil$. This improves the previously known upper bound of $2K$. 

    A \emph{cycle} (resp.~\emph{path}) \emph{odd-cover} of a graph $G$ is a set $\mathcal{H}$ of cycles (resp.~paths) such that $G$ is the symmetric difference of the graphs in $\mathcal{H}$.
    We prove that every Eulerian graph $G$ with maximum degree $\Delta$ admits a cycle odd-cover and a path odd-cover, each of size at most $\lceil 3\Delta/4\rceil$. This improves the previously known upper bound of $\Delta$. 

    The two proofs share many similarities and are both based on the proof of Akiyama, Exoo, and Harary that every graph with maximum degree 4 has linear arboricity at most 3.
    
\end{abstract}

\section{Introduction}
    Decomposition problems in graph theory typically seek to express the edge set of a graph or digraph $G$ as a disjoint union of edge sets of subgraphs of $G$ such that each subgraph belongs to a fixed family of graphs, and to minimize the number of such subgraphs required to decompose $G$.
    For example, Erd\H{o}s and Gallai~\cite{erdHos1983some} (resp.~Gallai~\cite{lovasz1968covering}) conjectured in the 1960's that every Eulerian graph (resp.~connected graph) on $n$ vertices can be decomposed into at most $O(n)$ cycles (resp.~$\lceil\frac n2\rceil$ paths). 
    In a similar vein, Akiyama, Exoo, and Harary~\cite{AHE81} conjectured in 1981 that every graph with maximum degree $\Delta$ can be decomposed into at most $\lceil\frac{\Delta+1}2\rceil$ \emph{linear forests} (forests whose connected components are all paths). These are classical and long-standing open problems in graph theory that have drawn widespread attention. Directed analogues of these problems have also been studied~\cite{bollobas1996proof, alspach1974path, nakayama1987linear}.

    In this paper, we study two independently motivated variants of the above problems where we seek to express a given (di)graph $G$ as a certain ``product'' or ``sum'' of cycles or paths, rather than as a disjoint union. The first variant comes from the study of the combinatorial diameters of \emph{partition polytopes}, where we seek to express the ``difference'' between two partitions of a set with the same cluster sizes, represented as an Eulerian digraph, as a minimum length sequence of cyclic exchanges, represented by directed cycles (see \Cref{sec:intropp}). The second variant comes from the study of \emph{odd-covers} where we seek to express a graph as a symmetric difference of a minimum set of cycles or paths in the underlying complete graph (see \Cref{sec:introodd}).

    Our main results are improved bounds for these two problems. Both proofs are based on a proof of Akiyama, Exoo, and Harary~\cite{AHE81} that every graph with maximum degree 4 has linear arboricity at most 3. 
    
\vspace*{0.4cm}
     \noindent
    \textbf{Notation.} A \emph{graph} refers to a simple, undirected graph with no loops or multi-edges. We generally assume that graphs have no isolated vertices, and they may be identified with their edge sets. Thus, for any set~$E$ of edges in the complete graph~$K_n$, we may write~$V(E)$ for the set of vertices incident to some edge in~$E$. 
   We denote the \emph{symmetric difference} of two sets $A$ and $B$ by $A \oplus B := (A \setminus B) \cup (B \setminus A)$. In particular, the \emph{symmetric difference} of two graphs $H_1$ and $H_2$ is the graph $H_1 \oplus H_2$ with edge set $E(H_1) \oplus E(H_2)$. We write~$\Delta(G)$ for the maximum degree of a graph~$G$ and~$\Delta_e(G) := 2\lceil \Delta(G)/2\rceil$ for the smallest even integer satisfying~$\Delta(G) \leq \Delta_e(G)$. We write~$\vodd(G)$ for the number of vertices in~$G$ with odd degree. If every vertex of a graph~$G$ has even degree (i.e.,~$\vodd(G) = 0$), then~$G$ is called \emph{Eulerian} (we do not require~$G$ to be connected). A \emph{polycycle} is a graph in which every component is a cycle. The \emph{distance} between two vertices in a connected graph $G$ is the length of the shortest path between them. The \emph{diameter} of $G$ is the maximum distance between any two vertices in $G$. 
   
   A \emph{digraph} refers to a directed multigraph, possibly with loops. Formally, we represent a digraph as a tuple~$G = (V, E, \tl, \hd)$, where~$\tl$ and~$\hd$ are functions~$E \to V$. We think of each edge~$e \in E$ as an arc from the vertex~$\tl(e) \in V$ to the vertex~$\hd(e) \in V$. The digraph~$G$ is called \emph{Eulerian} if, for every vertex~$v \in V$, the \emph{out-degree}~$\deg^+(v) := |\tl^{-1}(v)|$ is equal to the \emph{in-degree}~$\deg^-(v) := |\hd^{-1}(v)|$. We define a \emph{directed polycycle} as a digraph in which every component is a directed cycle. 

    The \emph{support} of a permutation~$\pi \in S_m$, denoted~$\supp(\pi)$, is the set of~$j \in [m]$ with~$\pi(j) \neq j$. A permutation $\sigma \in S_m$ is called a \emph{cycle} if there is an ordering $x_1, \ldots, x_t$ of its support such that $\sigma(x_i) = x_{i+1}$ for all $1 \leq i \leq t-1$, and $\sigma(x_t) = x_1$. In this case, we write $\sigma = (x_1 \: \cdots \: x_t)$. As a convention, we consider the identity permutation to be a cycle, which we call \emph{trivial}, and we may denote it as $(x)$ for any $x \in [m]$. Given a function~$p : [m] \to [n]$, we say that~$\pi \in S_m$ is \emph{$p$-balanced} if~$p$ is injective on~$\supp(\pi)$. If~$\sigma \in S_m$ is a cycle that is~$p$-balanced, then~$\sigma$ is called a \emph{$p$-cycle}.

    \subsection{Diameters of partition polytopes}\label{sec:intropp}
    
    The first problem arises naturally in the context of linear optimization for the balanced partitioning of a data set, and the diameters of the corresponding polytopes. 
    The study of the \emph{combinatorial diameters} of polyhedra, i.e.,~the diameters of their edge graphs, is a classic topic, in particular due to its connection to the open question of the possibility of a polynomial pivot rule for the famous Simplex method. 

    Let $m, n, \kappa_1, \ldots, \kappa_n$ be nonnegative integers with 
    $\kappa_1 + \cdots + \kappa_n = m$, and consider the task of partitioning a set of $m$ items into $n$ clusters of prescribed sizes $\kappa_1,\dots,\kappa_n$. The feasible solutions to this problem correspond to the vertices of the \emph{partition polytope} in $\mathbb R^{nm}$ with parameters $\kappa_1, \ldots, \kappa_n$, denoted $\PP(\kappa_1, \ldots, \kappa_n)$, which is described by the following linear constraints: 
  \begin{align*}
    \sum_{j=1}^{m} y_{ij}  &= \kappa_i \;\, \ \text{ for all } 1\leq i \leq n,\\
    \sum_{i=1}^{n} y_{ij}  &= 1 \;\;\;\ \text{ for all }  1\leq j \leq m,\\
    y_{ij} &\geq 0 \;\;\ \ \text{ for all }  1\leq i \leq n, \text{ }   1\leq j \leq m.
\end{align*}
   The first set of constraints guarantees that each cluster receives the prescribed number of items, and the second set of constraints ensures that each item is assigned to one cluster.
   Partition polytopes lie between two well-known classes of polytopes. They generalize the \emph{Birkhoff polytope} (or \emph{assignment polytope}), which represents permutations, i.e.,~one-to-one assignments between two sets of $m$ items ($n=m$ and $\kappa_i=1$ for $i\in[n]$); they also specialize transportation polytopes through the specification of one set of margins as $1$. 

The second set of constraints, together with the nonnegativity constraints, implies an upper bound of $y_{ij}  \leq 1$ for all $i,j$. It is well-known that the underlying constraint matrix is totally unimodular, hence the vertices of the partition polytope $\PP(\kappa_1, \ldots, \kappa_n)$ are integral and in one-to-one correspondence with the feasible assignments of items to clusters. More precisely, an \emph{$(m,n)$-partition} is a function $p:[m]\to[n]$, and its \emph{shape} is the $n$-tuple $(|p^{-1}(1)|,\dots,|p^{-1}(n)|)$ of cluster sizes. The \emph{incidence matrix} of an $(m,n)$-partition $p$ is the matrix $M_p=(y_{ij})\in\{0,1\}^{nm}$ with $y_{ij}=1$ if and only if $p(j)=i$. The vertices of $\PP(\kappa_1,\dots,\kappa_n)$ are exactly the incidence matrices of $(m,n)$-partitions with shape $(\kappa_1,\dots,\kappa_n)$.

To describe the edges of $\mathrm{PP}(\kappa_1, \ldots, \kappa_n)$, we use the terminology of clustering difference graphs as in~\cite{borgwardt2013diameter}, which we now define. Given two $(m,n)$-partitions $p, p' : [m] \to [n]$, the  \emph{clustering difference graph} (or \emph{CDG}, for short) from $p$ to $p'$ is the digraph $\mathrm{CDG}(p, p'):=([n],[m], p, p')$. Figure \ref{fig:CDGexample} depicts an example of a clustering difference graph. The vertices of $\mathrm{CDG}(p, p')$ are the clusters, and for each item $j$ there is an edge from $p(j)$ to $p'(j)$ representing the assignment of $j$ to the cluster $p(j)$ that must be moved to the cluster $p'(j)$.  If $p(j)=p'(j)$, then the edge corresponding to $j$ is a loop, and we omit loops in our figures. Note that if $p$ has shape $(\kappa_1,\dots,\kappa_n)$ and $p'$ has shape $(\kappa_1',\dots,\kappa_n')$, then for each $i\in[n]$, the node $i$ in $\mathrm{CDG}(p, p')$ has out-degree $\kappa_i$ and in-degree $\kappa_i'$. In particular, if $p$ and $p'$ have the same shape, then $\mathrm{CDG}(p,p')$ is Eulerian.

\begin{figure}[t]
     \centering
     \begin{subfigure}[b]{\textwidth}
         \centering
\subcaptionbox{Two $(9,4)$-partitions $p$ (left) and $p'$ (right). The items $1,\dots,9$ are assigned to $4$ clusters, as indicated by the circles.}{\begin{tikzpicture}[vertices/.style={draw, fill=black, circle, inner sep=0pt, minimum size = 4pt, outer sep=0pt}]

\path[use as bounding box] (-1,-1.5) rectangle (11,4.6);

\node[knode,minimum height=0.75cm] (c_1) at (0, 3.4) {\small $1,2$};
\node[knode,minimum height=0.75cm] (c_2) at (0, 0) {\small $3,4$};
\node[knode,minimum height=0.75cm] (c_3) at (2.5, 3.4) {\small $5,6$};
\node[knode,minimum height=0.75cm] (c_4) at (2.5, 0) {\small $7,8,9$};

\node[knode,minimum height=0.75cm] (c_5) at (8, 3.4) {\small $3,4$};
\node[knode,minimum height=0.75cm] (c_6) at (8, 0) {\small $1,7$};
\node[knode,minimum height=0.75cm] (c_7) at (10.5, 3.4) {\small $2,8$};
\node[knode,minimum height=0.75cm] (c_8) at (10.5, 0) {\small $5,6,9$};

\node[above=14pt] at (c_1) {1}; 
\node[below=14pt] at (c_2) {2}; 
\node[above=14pt] at (c_3) {3}; 
\node[below=17pt] at (c_4) {4}; 
\node[above=14pt] at (c_5) {1}; 
\node[below=14pt] at (c_6) {2}; 
\node[above=14pt] at (c_7) {3}; 
\node[below=17pt] at (c_8) {4}; 

\draw[dashed, ->] (4,1.7) -- (6.5,1.7);

\end{tikzpicture}}
        \end{subfigure}
        \hfill
        \begin{subfigure}[b]{\textwidth}
         \centering
\subcaptionbox{The clustering difference graph $\mathrm{CDG}(p,p')$. Each item $j\in\{1,\dots,9\}$ corresponds to an edge from $p(j)$ to $p'(j)$. Item 9 corresponds to a loop on cluster 4, which is omitted from the figure.}{
\begin{tikzpicture}[vertices/.style={draw, fill=black, circle, inner sep=0pt, minimum size = 4pt, outer sep=0pt}]

\path[use as bounding box] (-4.75,-1) rectangle (6.75,4.6);

\node[knode,minimum height=0.25cm] (c_1) at (0, 3.4) {};
\node[knode,minimum height=0.25cm] (c_2) at (0, 0) {};
\node[knode,minimum height=0.25cm] (c_3) at (2.5, 3.4) {};
\node[knode,minimum height=0.25cm] (c_4) at (2.5, 0) {};

\node[above=6pt] at (c_1) {1}; 
\node[below=6pt] at (c_2) {2}; 
\node[above=6pt] at (c_3) {3}; 
\node[below=6pt] at (c_4) {4}; 

\path[draw=black, thick,  ->, >=latex]  (c_3) edge[bend left=13] node[right]{$5$} (c_4);
\path[draw=black, thick,  ->, >=latex]  (c_4) edge[bend left=13] node[left]{$8$} (c_3);
\path[draw=black, thick,  ->, >=latex]  (c_1) edge[bend left=13] node[right]{$1$} (c_2);
\path[draw=black, thick,  ->, >=latex]  (c_2) edge[bend left=13] node[left]{$3$} (c_1);
\path[draw=black, thick,  ->, >=latex]  (c_2) edge[bend left=48] node[left]{$4$} (c_1);
\path[draw=black, thick,  ->, >=latex]  (c_3) edge[bend left=48] node[right]{$6$} (c_4);
\path[draw=black, thick,  ->, >=latex]  (c_1) edge node[below]{$2$} (c_3);
\path[draw=black, thick,  ->, >=latex]  (c_4) edge node[above]{$7$} (c_2);

\end{tikzpicture}}
        \end{subfigure}
     \caption{Illustration of two $(9,4)$-partitions $p$ and $p'$ and the corresponding clustering difference graph $\mathrm{CDG}(p,p')$.} 
        \label{fig:CDGexample}
\end{figure}

     It follows readily from the characterization of edges of general transportation polytopes in~\cite{kw-68} that two vertices $M_p$ and $M_{p'}$ of $\PP(\kappa_1, \ldots, \kappa_n)$ are joined by an edge if and only if $\mathrm{CDG}(p, p')$ (suppressing loops) is a simple directed cycle (see~\cite{borgwardt2013diameter}). Hence, an edge walk in $\PP(\kappa_1, \ldots, \kappa_n)$ takes the form of a sequence of `cyclic exchanges' of items between clusters as indicated by such cycles. In other words, two vertices $M_p$ and $M_{p'}$ of $\PP(\kappa_1, \ldots, \kappa_n)$ are joined by an edge if and only if there exists a non-trivial $p$-cycle $\tau\in S_m$ such that $p'=p\tau$. More generally, an edge walk in $\PP(\kappa_1, \ldots, \kappa_n)$ from $M_p$ to $M_{p'}$ corresponds to a sequence $(p_0, \ldots, p_t)$ of $(m,n)$-partitions such that $p_0 = p$, $p_t = p'$, and for each $i \in [t]$, we have $p_i = p_{i-1}\tau_i$ for some $p_{i-1}$-cycle $\tau_i$. We call such a sequence $(p_0,\dots,p_t)$ a \emph{resolution} of $(p,p')$ of length $t$, and we say that $(\tau_1, \ldots, \tau_t)$ \emph{encodes} $(p_0, \ldots, p_t)$. Hence the diameter of $\PP(\kappa_1,\dots,\kappa_n)$ is the maximum, over all pairs of $(m,n)$-partitions $p$ and $p'$ with shape $(\kappa_1,\dots,\kappa_n)$, of the shortest length of a resolution of $(p,p')$.

     There is yet another way to describe an edge walk from $M_p$ to $M_{p'}$ in $\PP(\kappa_1, \ldots, \kappa_k)$ which will be useful for some of our explanations. We say that a sequence $(\sigma_1, \ldots, \sigma_t)$ of $p$-cycles is a \emph{$p$-cycle decomposition} of $\mathrm{CDG}(p,p')$ of length $t$ if $p' = p\sigma_t \cdots \sigma_1$. We will later see that decompositions of $\mathrm{CDG}(p,p')$ of length $t$ are in one-to-one correspondence with resolutions of $(p,p')$ of length $t$ (see \Cref{prop:diameter reformulations}), hence determining the diameter of a partition polytope can be seen as a $p$-cycle decomposition problem for Eulerian digraphs. \Cref{fig:lemma 3.2} shows an example of a $p$-cycle decomposition. 

\begin{figure}[h!]
     \centering
     \begin{subfigure}[b]{\textwidth}
         \centering
\subcaptionbox{An Eulerian digraph $G$ representing moves in the permutation $\pi = (x^1_1 \: \cdots \: x^1_4)(x^2_1 \: \cdots \: x^2_5)(x^3_1 \: \cdots \: x^3_3)$. Here, $p$ is the identity map on $\{x^1_1, \ldots, x^1_4, x^2_1, \ldots, x^2_5, x^3_1, \ldots, x^3_3\}$, $p' = p\pi$, and $G = \mathrm{CDG}(p,p')$. \label{fig:pi}}{
\begin{tikzpicture}

\path[use as bounding box] (-0.5,-0.5) rectangle (11,3.5);
            \node[knode] (x1^1) at (0,0) {\small $x_1^1$};
            \node[knode] (x2^1) at (0,2) {\small $x_2^1$};
            \node[knode] (x3^1) at (2,2) {\small $x_3^1$};
            \node[knode] (x4^1) at (2,0) {\small $x_4^1$};
            
            \node[knode] (x1^2) at (4.5,0) {\small $x_1^2$};
    	\node[knode] (x2^2) at (4,2) {\small $x_2^2$};
            \node[knode] (x3^2) at (5.5,3) {\small $x_3^2$};
            \node[knode] (x4^2) at (7,2) {\small $x_4^2$};
            \node[knode] (x5^2) at (6.5,0) {\small $x_5^2$};
            
            \node[knode] (x1^3) at (8.5,0) {\small $x_1^3$};
			\node[knode] (x2^3) at (9.5,2) {\small $x_2^3$};
            \node[knode] (x3^3) at (10.5,0) {\small $x_3^3$};
            
			\draw[->, thick] (x1^1) to (x2^1);
            \draw[->, thick] (x2^1) to (x3^1);
            \draw[->, thick] (x3^1) to (x4^1);
			\draw[->, thick] (x4^1) to (x1^1);

            \draw[->, thick] (x1^2) to (x2^2);
            \draw[->, thick] (x2^2) to (x3^2);
            \draw[->, thick] (x3^2) to (x4^2);
            \draw[->, thick] (x4^2) to (x5^2);
            \draw[->, thick] (x5^2) to (x1^2);

            \draw[->, thick] (x1^3) to (x2^3);
            \draw[->, thick] (x2^3) to (x3^3);
            \draw[->, thick] (x3^3) to (x1^3);
			
		\end{tikzpicture}}
        \end{subfigure}
        \hfill
        \begin{subfigure}[b]{\textwidth}
         \centering
\subcaptionbox{A $p$-cycle decomposition $(\sigma_1, \sigma_2)$ of the directed polycycle $G$ above. The black, solid edges correspond to $\sigma_1 = (x^1_1 \: \cdots \: x^1_4 \: x^2_1 \: \cdots \: x^2_5 \: x^3_1 \: \cdots \: x^3_3)$, and the orange, wavy edges correspond to $\sigma_2 = (x^3_1 \: x^2_1 \: x^1_1)$. Each object is moved to its target location by following a black, then an orange edge (if available), so $p' = p\sigma_2\sigma_1$ is satisfied. \label{fig:3.2sigma1}}{
\begin{tikzpicture}
\path[use as bounding box] (-0.5,-2) rectangle (11,3.5);
            \node[knode] (x1^1) at (0,0) {\small $x_1^1$};
            \node[knode] (x2^1) at (0,2) {\small $x_2^1$};
            \node[knode] (x3^1) at (2,2) {\small $x_3^1$};
            \node[knode] (x4^1) at (2,0) {\small $x_4^1$};
            
            \node[knode] (x1^2) at (4.5,0) {\small $x_1^2$};
    	\node[knode] (x2^2) at (4,2) {\small $x_2^2$};
            \node[knode] (x3^2) at (5.5,3) {\small $x_3^2$};
            \node[knode] (x4^2) at (7,2) {\small $x_4^2$};
            \node[knode] (x5^2) at (6.5,0) {\small $x_5^2$};
            
            \node[knode] (x1^3) at (8.5,0) {\small $x_1^3$};
			\node[knode] (x2^3) at (9.5,2) {\small $x_2^3$};
            \node[knode] (x3^3) at (10.5,0) {\small $x_3^3$};
            
			\draw[->, thick] (x1^1) to (x2^1);
            \draw[->, thick] (x2^1) to (x3^1);
            \draw[->, thick] (x3^1) to (x4^1);
			\draw[->, thick] (x4^1) to (x1^2);

            \draw[->, thick] (x1^2) to (x2^2);
            \draw[->, thick] (x2^2) to (x3^2);
            \draw[->, thick] (x3^2) to (x4^2);
            \draw[->, thick] (x4^2) to (x5^2);
            \draw[->, thick] (x5^2) to (x1^3);

            \draw[->, thick] (x1^3) to (x2^3);
            \draw[->, thick] (x2^3) to (x3^3);
            \draw[->, thick] (x3^3) to [out = -120, in = -80, looseness=0.5] (x1^1);

 \tikzset{decoration={snake,amplitude=.4mm,segment length=2mm, post length=1mm, pre length=0mm}}

\draw[->, orange, decorate, thick] (x1^1) to [out = -55, in = -125, looseness=0.5] (x1^3);
\draw[->, orange, decorate, thick] (x1^3) to [out = -150, in = -30, looseness=0.7] (x1^2);
\draw[->, orange, decorate, thick] (x1^2) to [out = -150, in = -30, looseness=0.7] (x1^1);
            
		\end{tikzpicture}}
        \end{subfigure}
     \caption{A $p$-cycle decomposition of a CDG. } 
        \label{fig:lemma 3.2}
\end{figure}
 
Unlike for Birkhoff polytopes, which have diameter $2$~\cite{br-74}, and transportation polytopes, which meet the Hirsch bound~\cite{bdf-17}, the diameters of partition polytopes are not fully understood. If $\kappa_1 \geq \cdots \geq \kappa_n \geq 0$, then it is known that the diameter of $\PP(\kappa_1, \ldots, \kappa_n)$ is at most $\kappa_1 + \kappa_2$~\cite{borgwardt2013diameter}, i.e., the sum of the sizes of the two largest clusters. Note that the property $\kappa_1 \geq \cdots \geq \kappa_n$ is not a restriction and can be achieved through relabeling. This bound $\kappa_1 + \kappa_2$ recovers the bound $1+1=2$ for Birkhoff polytopes, and it is dramatically lower than the Hirsch bound known to be tight for general transportation polytopes, which for partition polytopes takes the form $m+n-1$ (for $\kappa_n>0$). In this work, we improve this upper bound as follows.

    \begin{theorem}\label{thm:improved upper bound}
	Let $\kappa_1 \geq \cdots \geq \kappa_n$ be nonnegative integers. Then the partition polytope $\PP(\kappa_1, \ldots, \kappa_n)$ has diameter at most $\kappa_1 + \lceil \kappa_2/2 \rceil$.
    \end{theorem}
In terms of the largest cluster size $\kappa_1$, Theorem \ref{thm:improved upper bound} gives an upper bound of $\lceil 3\kappa_1/2\rceil$ on the diameter of $\PP(\kappa_1, \ldots, \kappa_n)$, which improves on the previous bound $2\kappa_1$ of~\cite{borgwardt2013diameter}.
We present a proof of Theorem \ref{thm:improved upper bound} in Section \ref{sec:upperboundpp}. The key argument is the following special case of the theorem when $\kappa_1 = 2$. 

    \begin{theorem}\label{thm:out-degree 2}
        Let $\kappa_1 \geq \cdots \geq \kappa_n$ be nonnegative integers with $\kappa_1=2$. Then the partition polytope $\PP(\kappa_1, \ldots, \kappa_n)$ has diameter at most 3.
    \end{theorem}
    This latter result is best possible. For example, if $n=4$ and $\kappa_1=\kappa_2=\kappa_3=\kappa_4=2$, then there are pairs $p,p'$ of $(8,4)$-partitions such that $\CDG(p,p')$ is the vertex-disjoint union of two cycles each of length $2$ and each with edge-multiplicity $2$; in this case, there does not exist a resolution of $(p,p')$ of length less than 3. 
    It was noted in~\cite{borgwardt2013diameter} that this construction can be generalized to show that $\PP(\kappa_1, \ldots, \kappa_n)$ can have diameter as large as $\lceil 4\kappa_1/3 \rceil$. We present a more general version of this lower bound based on all values of $\kappa_i$ in Section \ref{sec:lowerboundpp}.
    However, in terms of $\kappa_1$, there is still a gap between the known lower bound $\lceil 4\kappa_1/3\rceil$ and our new upper bound $\lceil 3\kappa_1/2\rceil$ on the diameter of $\PP(\kappa_1,\dots,\kappa_n)$. We suspect that the true diameter of $\PP(\kappa_1,\cdots,\kappa_n)$ is closer to the lower bound.
    
    \begin{conjecture}\label{conj:pp diameter}
        There exists a constant $c$ such that for any positive integers $n$ and $\kappa_1 \geq \cdots \geq \kappa_n$, the partition polytope $\PP(\kappa_1, \ldots, \kappa_n)$ has diameter at most $\lceil 4\kappa_1/3 \rceil + c$.
    \end{conjecture}
    One natural approach to proving this conjecture would be to prove a statement analogous to \Cref{thm:out-degree 2}, namely that if $\kappa_1\geq\cdots\geq\kappa_n$ and $\kappa_1=3$, then $\PP(\kappa_1,\cdots,\kappa_n)$ has diameter at most 4. Unfortunately, this statement is not true, as we will demonstrate with an explicit example in Section \ref{sec:lowerboundpp}. Thus, new techniques will be needed to resolve \Cref{conj:pp diameter}.

    \subsection{Odd-covers of graphs} \label{sec:introodd}

    Our second setting arises from the study of odd-covers of graphs. For a graph $G \subseteq K_n$ and a collection $\mathcal{H}=\{H_1,\dots,H_k\}$ of subgraphs of $K_n$, we say that $\mathcal{H}$ is an \emph{odd-cover} of $G$ if $G=H_1\oplus\dots\oplus H_k$ (that is, an edge is in $G$ if and only if it is in an odd number of graphs in $\mathcal{H}$). As a special case, $\mathcal H$ is called a \emph{decomposition} of $G$ if $H_1, \ldots, H_k$ are pairwise edge-disjoint (and so $G = H_1 \cup \cdots \cup H_k$, in particular).
    The notion of an odd-cover was introduced by Babai and Frankl~\cite{babai1988linear} who posed the problem of determining the minimum cardinality of an odd-cover $\mathcal{H}$ of $K_n$ such that $\mathcal{H}$ consists solely of complete bipartite graphs. 
    This problem is an $\mathbb{F}_2$-analogue of the celebrated Graham-Pollak theorem~\cite{graham1971, graham1972}; see~\cite{buchanan2024odd, buchanan2022odd, leader2024odd} for recent progress on this odd-cover problem. 

    In this paper, we are interested in the problem of determining the minimum cardinality of an odd-cover of a given graph consisting solely of cycles or of paths. To be precise, given a subgraph $G$ of a larger host graph $F$, a \emph{cycle odd-cover} (resp.~\emph{path odd-cover}) of $G$ in $F$ is an odd-cover $\mathcal H$ of $G$ in which every $H \in \mathcal H$ is a cycle (resp.~path) in $F$. If the host graph $F$ is not specified, we take $F$ to be the complete graph on $V(G)$. Note that we allow our cycles (resp.~paths) to include edges that are not in the original graph $G$. These problems were introduced in~\cite{BBCFRY23} as $\mathbb{F}_2$-analogues of the conjectures of Erd\H{o}s and Gallai~\cite{erdHos1983some} and of Gallai~\cite{lovasz1968covering} on decompositions of graphs into cycles and into paths. 
    The cycle odd-cover problem was also studied implicitly in~\cite{fan2003covers} where it is shown that every Eulerian graph $G$ on $n$ vertices admits a cycle odd-cover of size at most $\lfloor \frac{n-1}2\rfloor$ consisting of cycles in $G$; this was used to confirm a conjecture of Chung \cite{chung1980coverings} that every Eulerian graph $G$ admits a cycle \emph{cover} (i.e., a collection of cycles whose union is $G$) of size at most $\lfloor\frac{n}2\rfloor$.
    
    Note that a graph $G$ admits a cycle odd-cover if and only if $G$ is Eulerian, but every graph admits a path odd-cover. It was also observed in~\cite{BBCFRY23} that path and cycle odd-cover problems are in some sense undirected analogues of diameter problems on partition polytopes, and by adapting 
    techniques for proving diameter bounds on partition polytopes~\cite{borgwardt2013diameter, bv-19a} to undirected graphs, it was proved that every Eulerian graph $G$ has a cycle odd-cover of cardinality at most $\Delta(G)$, the maximum degree of $G$, and that every graph $G$ has a path odd-cover of cardinality at most $\max\left\{\frac{v_{\text{odd}}(G)}2, \Delta_e(G)\right\}$, where $v_{\text{odd}}(G)$ denotes the number of vertices of odd degree in $G$, and $\Delta_e(G) := 2\left \lceil \Delta(G)/2 \right \rceil$. 

    Another interesting observation in~\cite{BBCFRY23} is that path odd-covers are strengthenings of linear forest decompositions (recall that a \emph{linear forest} is a graph whose connected components are all paths). The \emph{linear arboricity} of a graph $G$, denoted $\mathrm{la}(G)$, is the minimum cardinality of a linear forest decomposition of $G$. If $\mathcal H = \{P_1, ,\ldots, P_k\}$ is a path odd-cover of $G$, then $G$ also admits a linear forest decomposition of the same cardinality $k$, obtained by replacing each path $P_i$ in $\mathcal{H}$ by the linear forest $(P_i\cap G) \setminus(P_1 \cup \cdots \cup P_{i-1})$. Thus the minimum size of a path odd-cover of $G$ is at least $\mathrm{la}(G)$. The authors of~\cite{BBCFRY23} proved a partial converse of this statement for Eulerian graphs: for every Eulerian graph $G \subseteq K_n$, there exists $n' \geq n$ such that $G$ admits a path odd-cover in $K_{n'}$ of size $\mathrm{la}(G)$ (and also a cycle odd-cover in $K_{n'}$ of size at most $\mathrm{la}(G)$). These observations connect linear arboricity to path and cycle odd-covers, which are in turn related to diameters of partition polytopes. In this work, we strengthen these connections by refining a proof of Akiyama, Exoo, and Harary~\cite{AHE81} that every graph $G$ with $\Delta(G)=4$ has linear arboricity at most 3 to obtain the following result on path and cycle odd-covers; our proof of Theorem \ref{thm:out-degree 2} is also based on this proof of~\cite{AHE81}.
    \begin{theorem}\label{thm:oddcover3}
        Let $G$ be an Eulerian graph with $\Delta(G)=4$. Then $G$ admits a path odd-cover of size at most~$3$ and a cycle odd-cover of size at most~$3$. 
    \end{theorem}

    The two conclusions in \Cref{thm:oddcover3} are proven in Sections \ref{sec:path odd-covers} and \ref{sec:oddcycle}. We then show, 
    following the framework of~\cite{BBCFRY23}, that Theorem \ref{thm:oddcover3}
    leads to the following improvement on minimum path and cycle odd-covers of general graphs.
    \begin{theorem}\label{thm:odd cover general}
        Every graph $G$ admits a path odd-cover of size at most $\max \left\{\frac{\vodd(G)}{2}, \left \lceil\frac{\vodd(G)/2 + 3\Delta_e(G)}{4} \right \rceil\right\}$.
        Moreover, if $G$ is Eulerian, then $G$ admits a path odd-cover of size at most $\lceil\frac34\Delta(G)\rceil$ and a cycle odd-cover of size at most $\lceil\frac34\Delta(G)\rceil$.
    \end{theorem}

    Note that there is a trivial lower bound of $\max\{\frac{\vodd(G)}{2},\frac{\Delta(G)}{2}\}$ (resp.~$\frac{\Delta(G)}{2}$) on the minimum size of a path (resp.~cycle) odd-cover of a graph (resp.~Eulerian graph) $G$; see \cite{BBCFRY23}. Hence, for graphs with many vertices with odd degree (more precisely, for graphs $G$ with $\frac{\vodd(G)}{2} \geq \Delta_e(G)$), the minimum size of a path odd-cover of $G$ is equal to the trivial lower bound $\frac{\vodd(G)}{2}$. On the opposite end, for Eulerian graphs $G$ (i.e.,~$\vodd(G)=0$), Theorem \ref{thm:odd cover general} improves the upper bound $\Delta(G)$ of \cite{BBCFRY23} to $\lceil\frac{3\Delta(G)}{4}\rceil$, cutting the gap to the trivial lower bound $\frac{\Delta(G)}{2}$ roughly in half, for both path and cycle odd-covers.

    The lower bound $\frac{\Delta(G)}{2}$ for Eulerian graphs can be improved slightly to $\frac12\Delta(G)+1$ for certain graphs: let $\Delta\geq 2$ be an even integer, let $G$ be a non-Hamiltonian $\Delta$-regular graph (for example, the disjoint union of two copies of $K_{\Delta+1}$; see Figure \ref{fig:pathcoverexample}) and let $n=|V(G)|$.
    Then $G$ is an Eulerian graph with $\Delta(G)=\Delta$, but $G$ admits neither a path odd-cover of size $\frac12\Delta$ nor a cycle odd-cover of size $\frac12\Delta$. 
    Indeed, since every path on at most $n$ vertices has at most $n-1$ edges, the symmetric difference of $\frac12\Delta$ such paths has at most $\frac12\Delta(n-1)$ edges. Since $G$ has $\frac12\Delta n$ edges, it follows that $G$ does not admit a path odd-cover of size $\frac12\Delta$. Moreover, since a cycle on at most $n$ vertices has at most $n$ edges, $G$ admits a cycle odd-cover of size $\frac12\Delta$ if and only if $G$ can be decomposed into $\frac12\Delta$ Hamiltonian cycles, which would contradict our assumption that $G$ is non-Hamiltonian. Hence $G$ does not admit a cycle odd-cover of size $\frac12\Delta$.

    In particular, there are infinitely many Eulerian graphs $G$ with $\Delta(G)=4$ that do not admit a path nor cycle odd-cover of size $\frac12\Delta(G)=2$, and Theorem \ref{thm:oddcover3} is best possible. Compared to our new upper bound $\lceil\frac34\Delta(G)\rceil$, we again suspect that this lower bound is closer to the truth.

    \begin{conjecture}
        There exists a constant $c$ such that every Eulerian graph $G$ admits a path odd-cover of size $\frac12\Delta(G)+c$ and a cycle odd-cover of size at most $\frac12\Delta(G)+c$.
    \end{conjecture}

    \begin{figure}[ht]
\centering        
\begin{tikzpicture}[vertices/.style={draw, fill=black, circle, inner sep=0pt, minimum size = 4pt, outer sep=0pt}, r_edge/.style={draw=red,line width= 1,>=latex,red}, b_edge/.style={draw=blue,line width= 1,>=latex,blue}, k_edge/.style={draw=black,line width= 1,>=latex,black}, scale=1]
\path[use as bounding box] (-0.5,-2) rectangle (10.5,2.5);

\node[vertices] (c_1) at (0.0,0.6) {};
\node[vertices] (c_2) at (2.0, 2) {};
\node[vertices] (c_3) at (4.0, 0.6) {};
\node[vertices] (c_4) at (3.2, -1.4) {};
\node[vertices] (c_5) at (0.8, -1.4) {};

\node[vertices] (d_1) at (6.0,0.6) {};
\node[vertices] (d_2) at (8.0, 2) {};
\node[vertices] (d_3) at (10.0, 0.6) {};
\node[vertices] (d_4) at (9.2, -1.4) {};
\node[vertices] (d_5) at (6.8, -1.4) {};

\foreach \to/\from in {c_1/c_2, c_2/c_3, c_3/c_4, c_4/c_5, c_5/c_1, c_1/c_3, c_1/c_4, c_2/c_4, c_2/c_5, c_3/c_5,
    d_1/d_2, d_2/d_3, d_3/d_4, d_4/d_5, d_5/d_1, d_1/d_3, d_1/d_4, d_2/d_4, d_2/d_5, d_3/d_5}
\draw[k_edge]  (\to)--(\from);
\end{tikzpicture}

\begin{tikzpicture}[vertices/.style={draw, fill=black, circle, inner sep=0pt, minimum size = 4pt, outer sep=0pt}, r_edge/.style={draw=red, line width= 1,>=latex, dashed}, b_edge/.style={draw=blue,line width= 1,>=latex}, k_edge/.style={draw=black,line width= 2,>=latex,black}, scale=1]
\tikzset{decoration={snake,amplitude=.4mm,segment length=2mm}}
\path[use as bounding box] (-0.5,-2) rectangle (10.5,2.5);

\node[vertices] (c_1) at (0.0,0.6) {};
\node[vertices] (c_2) at (2.0, 2) {};
\node[vertices] (c_3) at (4.0, 0.6) {};
\node[vertices] (c_4) at (3.2, -1.4) {};
\node[vertices] (c_5) at (0.8, -1.4) {};

\node[vertices] (d_1) at (6.0,0.6) {};
\node[vertices] (d_2) at (8.0, 2) {};
\node[vertices] (d_3) at (10.0, 0.6) {};
\node[vertices] (d_4) at (9.2, -1.4) {};
\node[vertices] (d_5) at (6.8, -1.4) {};

\foreach \to/\from in {c_1/c_5,c_4/c_5,c_1/c_2,c_2/c_3,d_1/d_2,d_2/d_3,d_3/d_4,d_4/d_5}
\draw[k_edge]  (\to)--(\from);

\foreach \to/\from in {c_5/c_2,c_2/c_4,c_4/c_1,c_1/c_3,d_1/d_3,d_3/d_5,d_5/d_2,d_2/d_4}
\draw[r_edge]  (\to)--(\from);

\foreach \to/\from in {c_5/c_3,c_3/c_4,d_5/d_1,d_1/d_4}
\draw[b_edge]  (\to)--(\from);

\path [k_edge] (c_3) edge[bend left=15, looseness=1] (d_1);
\path [r_edge] (c_3) edge[bend right=15, looseness=1] (d_1);

\path [k_edge] (c_4) edge[bend right=10, looseness=1] (d_5);
\path [b_edge] (c_4) edge[bend left=10, looseness=1] (d_5);

\path [r_edge] (c_5) edge[bend right=15, looseness=1] (d_4);
\path [b_edge] (c_5) edge[bend right=20, looseness=1] (d_4);

\end{tikzpicture}

\caption{An odd-cover of the disjoint union of two $K_5$'s by $3$ cycles.}
\label{fig:pathcoverexample}
\end{figure}
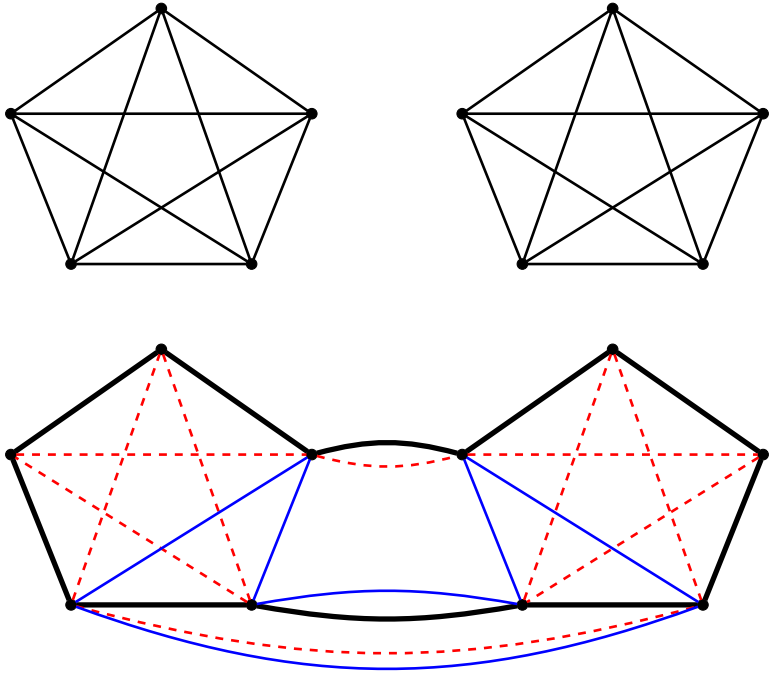
    
\vspace*{0.4cm}
     \noindent
    \textbf{Outline.} We provide some preliminaries for our discussion of the two decomposition problems in Section \ref{sec:prelim}. In Section \ref{sec:pp}, we devise improved upper and lower bounds on the diameters of partition polytopes. Section \ref{sec:odd} is dedicated to improved bounds for path and cycle odd-covers. We conclude with a final remark in Section \ref{sec:conclusion}.

\section{Preliminaries}\label{sec:prelim}
In this section, we collect the elementary observations and external results that we will need for our proofs.
\subsection{Edge walks in the partition polytope}

In \Cref{sec:intropp}, we discussed two equivalent ways of describing edge walks in $\PP(\kappa_1, \ldots, \kappa_k)$. We now discuss these perspectives in more detail and formally prove their equivalence. 

We think of a $p$-cycle decomposition $(\sigma_1, \ldots, \sigma_t)$ of $\mathrm{CDG}(p,p')$ as a type of factorization of $p' = p\sigma_t \cdots \sigma_1$ into irreducible parts, namely $p$-cycles. On the other hand, we think of a resolution $(p_0, \ldots, p_t)$ of $(p,p')$ encoded by $(\tau_1, \ldots, \tau_t)$ as describing a process of transitioning from $p = p_0$ to $p' = p_t$ in $t$ stages, where at the $i$-th stage ($1 \leq i \leq t$), the $p_{i-1}$-cycle $\tau_i$ dictates the objects that are to be interchanged. Of course, the $p$-cycle $\sigma_i$ also encodes the $i$-th cyclic exchange, but in a different way: namely, by dictating the \textit{absolute positions} of the objects that are interchanged (thinking of each object $j \in [m]$ as occupying position $j$ at the beginning of the process). Our study of diameters of partition polytopes will make use of both of these perspectives.
    
\begin{prop}\label{prop:diameter reformulations}
    Let $\kappa_1,\ldots,\kappa_n\geq 0$ be integers summing to $m$, and let $p, p' : [m] \to [n]$ be $(m,n)$-partitions with shape $(\kappa_1, \ldots, \kappa_n)$. Then the following are equivalent for $t \in \mathbb N$.
	\begin{enumerate}[label=(\roman*)]
		\item The incidence matrices $M_p$ and $M_{p'}$ are at distance at most $t$ in the graph of $\PP(\kappa_1, \ldots, \kappa_n)$.
		\item There exists a resolution of $(p,p')$ of length $t$.
		\item There exists a $p$-cycle decomposition of $\mathrm{CDG}(p,p')$ of length $t$.
	\end{enumerate}
\end{prop}
\begin{proof}
    The equivalence of $(i)$ and $(ii)$ follows from the characterization of the edges of $\PP(\kappa_1,\dots,\kappa_n)$. To see the equivalence of $(ii)$ and $(iii)$, we first observe the following.
    \begin{claim*}
        For any $\sigma, \pi \in S_m$, $\sigma$ is a $p$-cycle if and only if $\tau := \pi^{-1}\sigma\pi$ is a $p\pi$-cycle.
    \end{claim*}
    \begin{subproof}
        Since conjugation preserves cycle structure, we have that $\sigma$ is a cycle if and only if $\tau$ is a cycle. It remains to show that $\sigma$ is $p$-balanced if and only if $\tau$ is $p\pi$-balanced. Note that for $j\in[m]$, we have $j\in \supp(\tau)=\supp(\pi^{-1}\sigma\pi)$ if and only if $\pi(j) \in \supp(\sigma)$. This gives
    \begin{align*}
        \supp(\sigma) &= \pi(\supp(\tau)).
    \end{align*}
    It follows that $p$ is injective on $\supp(\sigma)$ if and only if $p\pi$ is injective on $\supp(\tau)$, as desired.
    \end{subproof}

    Now consider a resolution $(p_0, \ldots, p_t)$ of $(p,p')$ encoded by $(\tau_1, \ldots, \tau_t)$, and for each $1 \leq i \leq t$, define $\sigma_i := (\tau_1 \cdots \tau_{i-1})\tau_i(\tau_1 \cdots \tau_{i-1})^{-1}$. It is easy to see that for each $i \in [t]$ that $\tau_1 \cdots \tau_i = \sigma_i \cdots \sigma_1$. In particular, we have 
    \[p' = p_t = p\tau_1 \cdots \tau_t = p\sigma_t \cdots \sigma_1.\]
    Also, by the definition of resolution, we have for each $i \in [t]$ that $\tau_i$ is a $p_{i-1}$-cycle, where $p_{i-1} = p\tau_1 \cdots \tau_{i-1}$, so by the claim, $\sigma_i$ is a $p$-cycle. Thus $(\sigma_1, \ldots, \sigma_t)$ is a $p$-cycle decomposition of $\mathrm{CDG}(p,p')$. Thus $(ii)$ implies $(iii)$.

    Conversely, consider a $p$-cycle decomposition $(\sigma_1, \ldots, \sigma_t)$ of $\mathrm{CDG}(p,p')$, and for each $1 \leq i \leq t$, define $\tau_i := (\sigma_{i-1} \cdots \sigma_1)^{-1}\sigma_i(\sigma_{i-1} \cdots \sigma_1)$. Again, for each $i \in [t]$, we have that $\tau_1 \cdots \tau_i = \sigma_i \cdots \sigma_1$. Therefore, if we define
    \[p_i := p\sigma_i \cdots \sigma_1 = p\tau_1 \cdots \tau_i,\]
    for $0 \leq i \leq t$, we have that $p_i = p_{i-1}\tau_i$ for $1 \leq i \leq t$, and also $p' = p_t$. Moreover, for each $i \in [t]$, since $\sigma_i$ is a $p$-cycle, we have by the claim that $\tau_i$ is a $p(\sigma_{i-1} \cdots \sigma_1)$-cycle, hence a $p_{i-1}$-cycle, so $(p_0, \ldots, p_t)$ is a resolution of $(p,p')$. Thus $(iii)$ implies $(ii)$, and the proof is complete.
\end{proof}

\subsection{Decompositions into polycycles} \label{sec:polycycles}
A classical theorem of Petersen~\cite{Pet91} states that for every even positive integer $d$, any $d$-regular graph admits a decomposition into $2$-factors (i.e., spanning $2$-regular subgraphs). We will make use of the following generalization, proved in~\cite{borgwardt2013diameter} using a 
network flow argument and also in~\cite{BBCFRY23} using a simpler matching argument.

\begin{lemma}[\cite{borgwardt2013diameter, BBCFRY23}]\label{lem:Petersen generalization directed}
	Let~$G$ be an Eulerian directed multigraph with maximum out-degree~$\Delta$. Then~$G$ can be decomposed into~$\Delta$ edge-disjoint directed polycycles~$H_1, \ldots, H_\Delta$. 
    
    Moreover, for any integer~$0 \leq t \leq \Delta$, if at most one vertex in~$G$ has out-degree greater than~$t$, then~$H_{t+1}, \ldots, H_\Delta$ can be chosen to be directed cycles themselves. 
\end{lemma}

\noindent
\textbf{Remark.} The second statement of \Cref{lem:Petersen generalization directed} does not appear explicitly in~\cite{borgwardt2013diameter,BBCFRY23}, but it follows from the first statement. Indeed, let $v \in V(G)$ be such that every vertex $u \neq v$ has out-degree at most $t$. If $\{C_1, \ldots, C_s\}$ is a cycle decomposition of $G$, with $v \in V(C_1), \ldots, V(C_{\Delta-t})$, then $G' := G \setminus \bigcup_{i=1}^{\Delta-t}E(C_i)$ is Eulerian with maximum out-degree $t$, which decomposes into polycycles $H_1, \ldots, H_t$ by the first statement. Now setting $H_{t+i} := C_i$ for $1 \leq i \leq \Delta - t$ establishes the second statement.
\\

\Cref{lem:Petersen generalization directed} will serve as a starting point for the proofs of \Cref{thm:improved upper bound} and \Cref{thm:odd cover general}. The following two corollaries specialize the lemma to each of our two group-theoretic settings. The first (\Cref{cor:decomposition into balanced permutations}) is merely a restatement of \Cref{lem:Petersen generalization directed} in the language of $p$-balanced permutations. The second (\Cref{cor:Petersen generalization undirected}) is a version of \Cref{lem:Petersen generalization directed} for Eulerian undirected graphs.

\begin{cor}\label{cor:decomposition into balanced permutations}
	Let~$\kappa_1 \geq \cdots \geq \kappa_n \geq 0$ be integers summing to~$m$, and let~$p,p' : [m] \to [n]$ be~$(m,n)$-partitions with shape~$(\kappa_1, \ldots, \kappa_n)$.
    Then there exist~$p$-cycles~$\sigma_1, \ldots, \sigma_{\kappa_1 - \kappa_2} \in S_m$ and~$p$-balanced permutations~$\pi_1, \ldots, \pi_{\kappa_2} \in S_m$, all with disjoint supports, such that~$p' = p\sigma_{\kappa_1 - \kappa_2} \cdots \sigma_1 \pi_{\kappa_2} \cdots \pi_1$.
\end{cor}

\begin{proof}
    Let~$G = \mathrm{CDG}(p, p')$, which is an Eulerian digraph with maximum out-degree~$\kappa_1$ and has at most one vertex of out-degree greater than~$\kappa_2$. Let~$H_1, \ldots, H_{\kappa_1}$ be polycycles as in \Cref{lem:Petersen generalization directed} with~$\Delta := \kappa_1$ and~$t := \kappa_2$. For each~$1 \leq i \leq \kappa_1$, since every vertex in~$V(H_i)$ has in-degree and out-degree exactly~$1$, we have the following property: for each edge~$e \in E(H_i)$, there exists a unique edge~$\pi_i(e) \in E(H_i)$ such that~$p(\pi_i(e)) = p'(e)$.
    Now defining~$\pi_i(e) = e$ for every~$e \in [m] \setminus E(H_i)$, we have that~$\pi_i$ is a~$p$-balanced permutation of~$[m]$. Since~$H_1, \ldots, H_{\kappa_1}$ decompose the edges of~$G$, the permutations~$\pi_1, \ldots, \pi_{\kappa_1}$ have pairwise disjoint supports, and~$p' = p\pi_{\kappa_1} \cdots \pi_{1}$. Finally, for~$1 \leq i \leq \kappa_1 - \kappa_2$, since~$H_{\kappa_2+i}$ is a directed cycle, it follows that~$\pi_{\kappa_2+1}$ is a~$p$-cycle, so we may define~$\sigma_i := \pi_{\kappa_2 + i}$.
\end{proof}

\begin{cor}\label{cor:Petersen generalization undirected}
	Let~$G$ be an Eulerian graph with maximum degree~$\Delta$. Then~$G$ can be decomposed into~$\Delta/2$ edge-disjoint polycycles~$H_1, \ldots, H_{\Delta/2}$. 

    Moreover, for any integer~$0 \leq t \leq \Delta/2$, if at most one vertex in~$G$ has degree greater than~$2t$, then~$H_{t+1}, \ldots, H_{\Delta/2}$ can be chosen to be directed cycles themselves.
\end{cor}
\begin{proof}
    Since~$G$ is Eulerian,~$G$ admits an Eulerian orientation~$\vec G$ with maximum out-degree~$\Delta/2$. Now applying \Cref{lem:Petersen generalization directed} to~$\vec G$ and removing orientations gives the desired result.
\end{proof}

\subsection{Decomposing polycycles}
We recall the well-known fact that every permutation~$\pi \in S_m$ is a product of two cycles. In fact, if~$\pi$ is~$p$-balanced for some~$p : [m] \to [n]$, then the two cycles can be chosen to be~$p$-cycles. This means that every directed polycycle~$G$ admits a $p$-cycle decomposition of length at most $2$.

\begin{lemma}\label{lem:one permutation two cycles}
	Let~$p : [m] \to [n]$. Then every~$p$-balanced permutation~$\pi \in S_m$ is a product of two~$p$-cycles.
\end{lemma}
\begin{proof}
    Let $\pi\in S_m$ be a $p$-balanced permutation and write~$\pi$ as a product of disjoint non-trivial cycles~$C_1 \cdots C_t$. For each~$i \in [t]$, write~$C_i = \left(x^i_1 \: x^i_2 \: \cdots \: x^i_{c_i}\right)$, where~$c_i = |\supp(C_i)|$, so that
    \begin{align*}
        \pi = (x_1^1\cdots x_{c_1}^1)(x_1^2\cdots x_{c_2}^2)\cdots(x_1^t\cdots x_{c_t}^t).
    \end{align*}
    We claim that $\pi=\sigma_2\sigma_1$, where
    \begin{align*}
        \sigma_1 &:= \left(x^1_1 \: \cdots \: x^1_{c_1} \: x^2_1 \: \cdots \: x^2_{c_2} \: \cdots \: x^t_1 \: \cdots \: x^t_{c_t}\right) \text{ and} \\
        \sigma_2 &:= \left(x^t_1 \: x^{t-1}_1 \: \cdots \: x^2_1 \: x^1_1\right).
    \end{align*}
    Indeed, for $i\in[t]$ and $j\in[c_i-1]$, we have $\sigma_2\sigma_1(x_j^i) = \sigma_2(x_{j+1}^i) = x_{j+1}^i=\pi(x_j^i)$, and for $i\in[t]$, we have $\sigma_2\sigma_1(x_{c_i}^i) = \sigma_2(x_1^{i+1})=x_1^i = \pi(x_{c_i}^i)$, where the index $i$ is modulo $t$.
    See Figure \ref{fig:lemma 3.2} for a visual representation of the $p$-cycle decomposition. Moreover, since~$\pi$ is~$p$-balanced, we have that~$p$ is injective on~$\supp(\pi) = \bigcup_{i \in [t]} \supp(C_i)$, so $\sigma_1$ and $\sigma_2$ are~$p$-cycles.
\end{proof}

In the odd-cover setting, we make use of an analogous observation for (undirected) polycycles.

\begin{lemma}\label{lem:polycycle two paths or cycles}
    Every polycycle $G$ admits a path odd-cover of size at most $2$ and a cycle odd-cover of size at most $2$.
\end{lemma}

\begin{proof}
    We may assume that $G$ is nonempty. Let $C_1, \ldots, C_t$ be the components of $G$. For each $i\in[t]$, select an edge $u_iv_i \in E(C_i)$. Define 
    \begin{align*}
        P_1 &:= \left(\bigcup_{i=1}^t C_i \setminus \{u_iv_i\}\right) \cup \bigcup_{i=1}^{t-1} \{v_iu_{i+1}\}; \\
        P_2 &= \left(\bigcup_{i=1}^t \{u_iv_i\}\right) \cup \bigcup_{i=1}^{t-1} \{v_iu_{i+1}\}.
    \end{align*}
    Then $\{P_1, P_2\}$ is a path odd-cover of $G$. 
    Moreover, if $t=1$, then $\{C_1\}$ is a cycle odd-cover of $G$, and if $t \geq 2$, then $\{P_1 \cup \{v_tu_1\}, P_2 \cup \{v_tu_1\}\}$ is a cycle odd-cover of $G$. See Figure \ref{fig:lemma 2.5}.
\end{proof}

\begin{figure}[h!]
     \centering
     \begin{subfigure}[b]{\textwidth}
         \centering
\subcaptionbox{A polycycle. \label{fig:pi}}{
\begin{tikzpicture}[vertices/.style={draw, fill=black, circle, inner sep=0pt, minimum size = 4pt, outer sep=0pt}, r_edge/.style={draw=red,line width= 1,>=latex,red}, b_edge/.style={draw=blue,line width= 1,>=latex,blue}, k_edge/.style={draw=black,line width= 1,>=latex,black}, scale=1]
\path[use as bounding box] (-0.5,-0.5) rectangle (11,3.5);
            
            \node[vertices] (x1^1) at (0,2) {};
            \node[vertices] (x2^1) at (0,0) {};
            \node[vertices] (x3^1) at (2,0) {};
            \node[vertices] (x4^1) at (2,2) {};
            
            \node[vertices] (x1^2) at (4,2) {};
			\node[vertices] (x2^2) at (4.5,0) {};
            \node[vertices] (x3^2) at (6.5,0) {};
            \node[vertices] (x4^2) at (7,2) {};
            \node[vertices] (x5^2) at (5.5,3) {};
            
            \node[vertices] (x1^3) at (9.5,2) {};
			\node[vertices] (x2^3) at (8.5,0) {};
            \node[vertices] (x3^3) at (10.5,0) {};
            
			\draw[ thick] (x1^1) to [out = -90, in = 90] (x2^1);
            \draw[ thick] (x2^1) to [out = 0, in = -180] (x3^1);
            \draw[ thick] (x3^1) to [out = 90, in = -90] (x4^1);
			\draw[ thick] (x4^1) to [out = -180, in = 0] (x1^1);

            \draw[ thick] (x1^2) to (x2^2);
            \draw[ thick] (x2^2) to (x3^2);
            \draw[ thick] (x3^2) to (x4^2);
            \draw[ thick] (x4^2) to (x5^2);
            \draw[ thick] (x5^2) to (x1^2);

            \draw[ thick] (x1^3) to (x2^3);
            \draw[ thick] (x2^3) to (x3^3);
            \draw[ thick] (x3^3) to (x1^3);
			
		\end{tikzpicture}}
        \end{subfigure}
        \hfill
        \begin{subfigure}[b]{\textwidth}
         \centering
\subcaptionbox{A path odd-cover of size $2$.\label{fig:3.2sigma1}}{
\begin{tikzpicture}[vertices/.style={draw, fill=black, circle, inner sep=0pt, minimum size = 4pt, outer sep=0pt}, r_edge/.style={draw=red,line width= 1,>=latex,red, dashed}, b_edge/.style={draw=blue,line width= 1,>=latex,blue}, k_edge/.style={draw=black,line width= 1,>=latex,black}, scale=1]
\path[use as bounding box] (-0.5,-0.5) rectangle (11,3.5);
            
            \node[vertices] (x1^1) at (0,2) {};
            \node[vertices] (x2^1) at (0,0) {};
            \node[vertices] (x3^1) at (2,0) {};
            \node[vertices] (x4^1) at (2,2) {};
            
            \node[vertices] (x1^2) at (4,2) {};
			\node[vertices] (x2^2) at (4.5,0) {};
            \node[vertices] (x3^2) at (6.5,0) {};
            \node[vertices] (x4^2) at (7,2) {};
            \node[vertices] (x5^2) at (5.5,3) {};
            
            \node[vertices] (x1^3) at (9.5,2) {};
			\node[vertices] (x2^3) at (8.5,0) {};
            \node[vertices] (x3^3) at (10.5,0) {};

\foreach \to/\from in {x2^1/x1^1,x1^1/x4^1,x4^1/x3^1,x2^2/x1^2,x1^2/x5^2,x5^2/x4^2,x4^2/x3^2,x2^3/x1^3,x1^3/x3^3}
\draw[b_edge]  (\to)--(\from);

\foreach \to/\from in {x3^1/x2^1,x3^2/x2^2,x3^3/x2^3}
\draw[r_edge]  (\to)--(\from);

\path [r_edge] (x3^1) edge[bend right=15, looseness=1] (x2^2);
\path [r_edge] (x3^2) edge[bend right=15, looseness=1] (x2^3);
\path [b_edge] (x3^1) edge[bend left=15, looseness=1] (x2^2);
\path [b_edge] (x3^2) edge[bend left=15, looseness=1] (x2^3);
            
		\end{tikzpicture}}
        \end{subfigure}
        \hfill
        \begin{subfigure}[b]{\textwidth}
         \centering
\subcaptionbox{A cycle odd-cover of size $2$.\label{fig:3.2sigma1}}{
\begin{tikzpicture}[vertices/.style={draw, fill=black, circle, inner sep=0pt, minimum size = 4pt, outer sep=0pt}, r_edge/.style={draw=red,line width= 1,>=latex,red, dashed}, b_edge/.style={draw=blue,line width= 1,>=latex,blue}, k_edge/.style={draw=black,line width= 1,>=latex,black}, scale=1]
\path[use as bounding box] (-0.5,-2) rectangle (11,3.5);
            
            \node[vertices] (x1^1) at (0,2) {};
            \node[vertices] (x2^1) at (0,0) {};
            \node[vertices] (x3^1) at (2,0) {};
            \node[vertices] (x4^1) at (2,2) {};
            
            \node[vertices] (x1^2) at (4,2) {};
			\node[vertices] (x2^2) at (4.5,0) {};
            \node[vertices] (x3^2) at (6.5,0) {};
            \node[vertices] (x4^2) at (7,2) {};
            \node[vertices] (x5^2) at (5.5,3) {};
            
            \node[vertices] (x1^3) at (9.5,2) {};
			\node[vertices] (x2^3) at (8.5,0) {};
            \node[vertices] (x3^3) at (10.5,0) {};

\foreach \to/\from in {x2^1/x1^1,x1^1/x4^1,x4^1/x3^1,x2^2/x1^2,x1^2/x5^2,x5^2/x4^2,x4^2/x3^2,x2^3/x1^3,x1^3/x3^3}
\draw[b_edge]  (\to)--(\from);

\foreach \to/\from in {x3^1/x2^1,x3^2/x2^2,x3^3/x2^3}
\draw[r_edge]  (\to)--(\from);

\path [r_edge] (x3^1) edge[bend right=15, looseness=1] (x2^2);
\path [r_edge] (x3^2) edge[bend right=15, looseness=1] (x2^3);
\path [b_edge] (x3^1) edge[bend left=15, looseness=1] (x2^2);
\path [b_edge] (x3^2) edge[bend left=15, looseness=1] (x2^3);

\path [r_edge] (x2^1) edge[bend right=20, looseness=1] (x3^3);
\path [b_edge] (x2^1) edge[bend right=30, looseness=0.9] (x3^3);

        \end{tikzpicture}}
        \end{subfigure}
     \caption{A path odd-cover and a cycle odd-cover of size 2 for a polycycle.} 
        \label{fig:lemma 2.5}
\end{figure}

\subsection{Linear arboricity}

In initiating the study of linear arboricity, Akiyama, Exoo, and Harary~\cite{AHE81} proved the following.
\begin{theorem}[\cite{AHE81}]\label{thm:three linear forests}
	Every graph $G$ of maximum degree $4$ admits a decomposition into three linear forests.
\end{theorem}
We present a version of their argument here, as our proof of \Cref{thm:oddcover3} will build directly from it. For our purposes, we assume that $G$ is Eulerian, so that $G$ admits a decomposition into two polycycles by \Cref{cor:Petersen generalization undirected}. 
For linear forest decompositions, there is no loss of generality in this assumption because any graph $G$ of maximum degree $4$ is a subgraph of some $4$-regular graph $G'$, and restricting a linear forest decomposition of $G'$ to the edges of $G$ yields a linear forest decomposition of $G$ of the same size. 

We call a matching $M$ \emph{transversal} to a graph $H$ if $M$ consists of precisely one edge from each component of $H$. We call a pair $(M_1, M_2)$ of matchings a \emph{transversal pair} for a pair of graphs $(H_1, H_2)$ if each $M_i$ ($1 \leq i \leq 2$) is transversal to $H_i$. The argument of~\cite{AHE81} now proceeds as follows.

\begin{lemma}\label{lem:3 linear forests from transversal pair}
    Let~$H_1$ and~$H_2$ be edge-disjoint polycycles, and let~$(M_1, M_2)$ be any transversal pair of matchings for~$(H_1, H_2)$. Then there exists a linear forest decomposition~$\mathcal F = \{F_1, F_2, F_3\}$ of~$H_1 \cup H_2$ such that~$H_1 \setminus M_1 \subseteq F_1 \subseteq (H_1\setminus M_1)\cup M_2$,~$M_1 \subseteq F_2 \subseteq M_1\cup M_2$, and~$H_2 \setminus M_2 = F_3$.
\end{lemma}

\begin{proof}
    Since~$M_1$ and~$M_2$ are matchings, every component of~$M_1 \cup M_2$ is a path or an even cycle (of length at least~$4$ since~$H_1$ and~$H_2$ are edge-disjoint). Let~$H'$ be the union of those components that are cycles, and let~$M' \subseteq M_2$ be a matching that is transversal to~$H'$. Define
    \begin{align*}
        F_1 &:= (H_1 \setminus M_1) \cup M'; \\
        F_2 &:= M_1 \cup (M_2 \setminus M'); \\
        F_3 &:= H_2 \setminus M_2.
    \end{align*}
    Clearly,~$F_1, F_2,$ and~$F_3$ are edge-disjoint and decompose~$H_1 \cup H_2$. It now remains to show that each~$F_i$ ($1 \leq i \leq 3$) is a linear forest.

    This is immediate for~$F_2$ and~$F_3$, as they were both obtained by deleting an edge from each cycle in a vertex-disjoint union of paths and cycles. We now consider~$F_1$. To begin, since~$V(M') \subseteq V(M_1)$ by construction, every endpoint of every edge in~$M'$ has degree~$1$ in~$H_1 \setminus M_1$. This means that every vertex in~$F_1$ has degree at most~$2$, and every component~$C$ of~$F_1$ contains some component~$P$ of~$H_1 \setminus M_1$, where for some edge~$uv \in M_1$,~$P$ is a path from~$u$ to~$v$. Since every component of~$M_1 \cup M_2$ contains at most one edge from~$M'$, and~$H_1$ and~$H_2$ are edge-disjoint, at most one of~$u$ or~$v$ can be incident to an edge from~$M'$. That is,~$C$ contains a vertex of degree~$1$ and is not a cycle. Thus~$F_1$ is a linear forest as well.
\end{proof}

\section{Improved bounds on diameters of partition polytopes}\label{sec:pp}

In this section, we devise improved upper and lower bounds on the diameters of partition polytopes. 

\subsection{Upper bound}\label{sec:upperboundpp}

First, we prove \Cref{thm:improved upper bound}. The strategy is simple: given nonnegative integers~$\kappa_1 \geq \cdots \geq \kappa_n$ summing to~$m$ and two $(m,n)$-partitions~$p,p' : [m] \to [n]$ with shape $(\kappa_1, \ldots, \kappa_n)$, we wish to express~$p'$ in the form~$p\sigma_d \cdots \sigma_1$ for~$p$-cycles~$\sigma_1, \ldots, \sigma_d$, with~$d \leq \kappa_1 + \left \lceil \kappa_2/2 \right \rceil$. First, we write~$p'$ as~$p\sigma_{\kappa_1 - \kappa_2} \cdots \sigma_1\pi_{\kappa_2} \cdots \pi_1$ as in \Cref{cor:decomposition into balanced permutations}, with each~$\sigma_i$ a~$p$-cycle and each~$\pi_i$ a~$p$-balanced permutation, all with disjoint supports. Next, we replace~$\pi_{\kappa_2} \cdots \pi_1$ with a product of~$p$-cycles. Initially, we see that replacing each~$\pi_i$ individually with two~$p$-cycles by \Cref{lem:one permutation two cycles} yields a decomposition of~$\mathrm{CDG}(p,p')$ of length~$(\kappa_1 - \kappa_2) + 2\kappa_2 = \kappa_1 + \kappa_2$, as in~\cite{borgwardt2013diameter}. To improve this bound to~$\kappa_1 + \lceil \kappa_2/2 \rceil$, we instead take the~$\pi_i$ in pairs, replacing each pair with three~$p$-cycles. This is made possible by the following lemma, which is equivalent to \Cref{thm:out-degree 2}.

\begin{lemma}\label{lem:two permutations three cycles}
	Let~$p : [m] \to [n]$, and let~$\pi_1, \pi_2 \in S_m$ be~$p$-balanced permutations with disjoint supports. Then there exist~$p$-cycles~$\sigma_1, \sigma_2, \sigma_3 \in S_m$ such that~$\supp(\sigma_3\sigma_2\sigma_1) = \supp(\pi_2\pi_1)$ and~$p\sigma_3\sigma_2\sigma_1 = p\pi_2\pi_1$.
\end{lemma}

To prove this claim, we need the following simple observation about matchings in graphs. We denote by $K_n$ the complete undirected graph with vertex set $[n]$.

\begin{lemma}\label{lem:disjoint representative vertices}
    Let $M_1, M_2 \subseteq K_n$ be two matchings. Then there exists a partition $\{S_1, S_2\}$ of $V(K_n)$ such that for each $i \in \{1,2\}$, every edge $e \in M_i$ intersects $S_1$ and $S_2$ in exactly one vertex.
\end{lemma}

\begin{proof}
    Let~$G \subseteq K_n$ be the union of~$M_1$ and~$M_2$. Since~$M_1$ and~$M_2$ are matchings,~$G$ does not contain odd cycles, so~$\chi(G) \leq 2$. Thus, we can take~$S_1$ and~$S_2$ to be the two color classes in any proper~$2$-coloring of~$G$.
\end{proof}

\begin{proof}[Proof of \Cref{lem:two permutations three cycles}]
	If~$\pi_2\pi_1$ is~$p$-balanced, then~$\pi_2\pi_1$ is a product of two~$p$-cycles and the trivial~$p$-cycle by \Cref{lem:one permutation two cycles}. Thus we may assume that there exist some~$x \in \supp(\pi_1)$ and~$y \in \supp(\pi_2)$ with~$p(x) = p(y)$.
    
    We first write~$\pi_1$ and~$\pi_2$ as products of disjoint non-trivial cycles~$C_1 \cdots C_t$ and~$D_1 \cdots D_s$, respectively, so that~$x \in \supp(C_1)$ and~$y \in \supp(D_s)$. Since each of the cycles~$C_i$ is non-trivial and~$p$-balanced, there exists a pair~$e_i \subseteq p(\supp(C_i))$ for all~$i \in [t]$. Similarly, there exists a pair~$f_i \subseteq p(\supp(D_i))$ for all~$i \in [s]$, and we can choose~$f_s = \{p(y), p(\pi_2(y))\}$. Let~$M_1$ and~$M_2$ be the matchings in~$K_n$ given by~$M_1 := \{e_i : i \in [t]\}$ and~$M_2 := \{f_i : i \in [s]\}$, and let~$\{S_1, S_2\}$ be a partition of~$[n]$ as in \Cref{lem:disjoint representative vertices}, so that every~$e_i$~$(i \in [t])$ and every~$f_i$ ($i \in [s]$) intersects~$S_1$ and~$S_2$ in exactly one vertex. Assume without loss of generality that~$p(y) = p(x) \in S_1$, and therefore~$p(\pi_2(y)) \in S_2$.
    
    Note that~$S_1$ intersects every~$p(C_i)$ ($i \in [t]$), and~$S_2$ intersects every~$p(D_i)$ ($i \in [s]$). Therefore, for each~$i \in [t]$, we may write~$C_i = \left(x^i_1 \: x^i_2 \: \cdots \: x^i_{c_i}\right)$ with~$p\left(x_1^i\right) \in S_1$, and for each~$i \in [s]$, we may write~$D_i = \left(y^i_1 \: y^i_2 \: \cdots \: y^i_{d_i}\right)$, with~$p\left(y_1^i\right) \in S_2$. Moreover, we can ensure that~$x_1^1 = x$ and~$y^s_1 = \pi_2(y)$, which means that~$p(x^1_1) = p(x) = p(y) = p(y^s_{d_s})$.
    
    Now we have
    \[\pi_2\pi_1 = (y_1^1\cdots y_{d_1}^1)\cdots(y_1^s\dots y_{d_s}^s)(x_1^1\cdots x_{c_1}^1)\cdots(x_1^t\dots x_{c_t}^t)\]
    and setting
    \begin{align*}
		\sigma_3 &:= \left(y^s_1 \: \cdots y^2_1 \: y^1_1 \: x^t_1 \: \cdots x^2_1 \: y^s_{d_s} \right); \\
		\sigma_2 &:= \left(y^1_1 \: \cdots \: y^1_{d_1} \: y^2_1 \: \cdots \: y^2_{d_2} \: \cdots \: y^s_1 \: \cdots \: y^s_{d_s-1} \: x^1_1\right); \\
		\sigma_1 &:= \left(x^1_1 \: \cdots \: x^1_{c_1} \: x^2_1 \: \cdots \: x^2_{c_2} \: \cdots \: x^t_1 \: \cdots \: x^t_{c_t}\right),
	\end{align*}
    we have
	\begin{equation}\label{eq:three-step}
	    \pi_2 \pi_1 = \left(x^1_1 \: y^s_{d_s}\right) \sigma_3 \sigma_2 \sigma_1.
	\end{equation}
	See Figure \ref{fig:two permutations three cycles} for an illustration of these permutations.
    
    Applying~$p$ to both sides of \eqref{eq:three-step} gives~$p\pi_2\pi_1 = p\sigma_3\sigma_2\sigma_1$ since~$p\left(x^1_1 \: y^s_{d_s}\right) = p$. Note that~$\sigma_1, \sigma_2$, and~$\sigma_3$ are all~$p$-cycles since~$\pi_1$ and~$\pi_2$ are~$p$-balanced,~$p\left(x^1_1\right) = p\left(y^s_{d_s}\right)$, and~$\left\{p\left(x^i_1\right) : i \in [t]\right\} \subseteq S_1$ is disjoint from~$\left\{p\left(y^i_1\right) : i \in [s]\right\} \subseteq S_2$. Also, it is easily seen that 
	\[\supp(\sigma_3\sigma_2\sigma_1) = \bigcup_{i=1}^t \left\{x^i_1, \ldots, x^i_{c_i}\right\} \cup \bigcup_{i=1}^s \left\{y^i_1, \ldots, y^i_{d_i}\right\} = \supp(\pi_2 \pi_1).\]
	This completes the proof.
\end{proof}

\begin{figure}[p]
     \centering
     \begin{subfigure}[b]{\textwidth}
         \centering
\begin{tikzpicture}

            \node[knode] (x1^1) at (0,0) {\small $\textcolor{blue}{x_1^1}, \textcolor{blue}{y_3^3}$};
            \node[knode] (x2^1) at (0,2) {\small $x_2^1, y_4^1$};
            \node[knode] (x3^1) at (2,2) {\small $x_3^1, \textcolor{red}{y_1^1}$};
            \node[knode] (x4^1) at (2,0) {\small $x_4^1, \textcolor{red}{y_1^2}$};
            
            \node[knode] (x1^2) at (4.5,0) {\small $\textcolor{blue}{x_1^2}$};
		\node[knode] (x2^2) at (4,2) {\small $x_2^2$};
            \node[knode] (x3^2) at (5.5,3) {\small $x_3^2, y_2^2$};
            \node[knode] (x4^2) at (7,2) {\small $x_4^2, y_2^3$};
            \node[knode] (x5^2) at (6.5,0) {\small $x_5^2, y_3^2$};

            \node[knode] (x1^3) at (8.5,0) {\small $\textcolor{blue}{x_1^3}$};
		\node[knode] (x2^3) at (9.5,2) {\small $x_2^3, y_2^1$};
            \node[knode] (x3^3) at (10.5,0) {\small $x_3^3, \textcolor{red}{y_1^3}$};
            
            \node[knode] (y4^1) at (5,4.5) {\small $y_3^1$};
            
		\draw[->, thick] (x1^1) to  (x2^1);
            \draw[->, thick] (x2^1) to  (x3^1);
            \draw[->, thick] (x3^1) to  (x4^1);
		\draw[->, thick] (x4^1) to  (x1^1);

            \draw[->, thick] (x1^2) to (x2^2);
            \draw[->, thick] (x2^2) to (x3^2);
            \draw[->, thick] (x3^2) to (x4^2);
            \draw[->, thick] (x4^2) to (x5^2);
            \draw[->, thick] (x5^2) to (x1^2);

            \draw[->, thick] (x1^3) to (x2^3);
            \draw[->, thick] (x2^3) to (x3^3);
            \draw[->, thick] (x3^3) to (x1^3);
            
            \tikzset{decoration={snake,amplitude=.4mm,segment length=2mm, post length=1mm, pre length=0mm}}

            \draw[->, green, thick, decorate] (x2^1) to [out=30, in=150] (x3^1);
            \draw[->, green, thick, decorate] (x3^1) to [out=60, in=120, looseness=0.8] (x2^3);
            \draw[->, green, thick, decorate] (x2^3) to [out=90, in=0] (y4^1);
		\draw[->, green, thick, decorate] (y4^1) to [out=180, in=60] (x2^1);

            \draw[->, green, thick, decorate] (x4^1) to [bend right=20] (x3^2);
            \draw[->, green, thick, decorate] (x3^2) to (x5^2);
            \draw[->, green, thick, decorate] (x5^2) to [bend left=30] (x4^1);
            
		\draw[->, green, thick, decorate] (x3^3) to (x4^2);
            \draw[->, green, thick, decorate] (x4^2) to [out=-60, in=-30, looseness=1.4] (x1^1);
            \draw[->, green, thick, decorate] (x1^1) to [out=-50, in=230, looseness=0.6] (x3^3);
			
		\end{tikzpicture}
         \caption{An example of a CDG representing two $p$-balanced permutations $\pi_1$ and $\pi_2$. The black solid edges represent $\pi_1=(x_1^1\:x_2^1\:x_3^1\:x_4^1)(x_1^2\:x_2^2\:x_3^2\:x_4^2\:x_5^2)(x_1^3\:x_2^3\:x_3^3)$ 
         and the green wavy edges represent $\pi_2=(y_1^1\:y_2^1\:y_3^1\:y_4^1)(y_1^2\:y_2^2\:y_3^2)(y_1^3\:y_2^3\:y_3^3)$. Blue labels denote items in clusters known to be in $S_1$, and red labels denote items in clusters known to be in $S_2$.
         }\label{fig:3.3CDG}
     \end{subfigure}

     \hfill
     \begin{subfigure}[b]{\textwidth}
\centering
\begin{tikzpicture}

            \node[circle, inner sep=0] (x1^1) at (0,0) {\small $\textcolor{blue}{x_1^1}$};
            \node[circle, inner sep=0] (y3^3) at (0.5,0) {\small $\textcolor{blue}{y_3^3}$};
            \node[circle, inner sep=0] (x2^1) at (0,2) {\small $x_2^1$};
            \node[circle, inner sep=0] (y1^1) at (2.5,2) {$\textcolor{red}{y_1^1}$};
            \node[circle, inner sep=0] (x3^1) at (2,2) {\small $x_3^1$};
            \node[circle, inner sep=0] (y2^1) at (10,2) {\small $y_2^1$};
            \node[circle, inner sep=0] (x4^1) at (2,0) {\small $x_4^1$};
            \node[circle, inner sep=0] (y1^2) at (2.5,0) {\small $\textcolor{red}{y_1^2}$};
            
            \node[circle, inner sep=0] (x1^2) at (4.5,0) {\small $\textcolor{blue}{x_1^2}$};
		\node[circle, inner sep=0] (x2^2) at (4,2) {\small $x_2^2$};
            \node[circle, inner sep=0] (x3^2) at (5.5,3) {\small $x_3^2$};
            \node[circle, inner sep=0] (y3^2) at (7,0) {\small $y_3^2$};
            \node[circle, inner sep=0] (x4^2) at (7,2) {\small $x_4^2$};
            \node[circle, inner sep=0] (y2^3) at (7.5,2) {\small $y_2^3$};
            \node[circle, inner sep=0] (x5^2) at (6.5,0) {\small $x_5^2$};
            \node[circle, inner sep=0] (y2^2) at (6,3) {\small $y_2^2$};
            
            \node[circle, inner sep=0] (x1^3) at (8.5,0) {\small $\textcolor{blue}{x_1^3}$};
		\node[circle, inner sep=0] (x2^3) at (9.5,2) {\small $x_2^3$};
            \node[circle, inner sep=0] (y3^1) at (5,4.5) {\small $y_3^1$};
            \node[circle, inner sep=0] (x3^3) at (10.5,0) {\small $x_3^3$};
            \node[circle, inner sep=0] (y1^3) at (11,0) {\small $\textcolor{red}{y_1^3}$};

            \node (y4^1) at (0.5,2) {\small $y_4^1$};
            
		\draw[->, thick] (x1^1) -- (x2^1);
            \draw[->, thick] (x2^1) to [out = -70, in = -120, looseness=0.6] (x3^1);
            \draw[->, thick] (x3^1) -- (x4^1);
			
            \draw[->, thick] (x4^1) to [out = 60, in = 130, looseness=0.7] (x1^2);

            \draw[->, thick] (x1^2) -- (x2^2);
            \draw[->, thick] (x2^2) -- (x3^2);
            \draw[->, thick] (x3^2) to [out=-90, in=160, looseness=0.1] (x4^2);
            \draw[->, thick] (x4^2) to [out=-120, in=90, looseness=0.1] (x5^2);
            
            \draw[->, thick] (x5^2) to [out = 60, in = 120, looseness=0.8] (x1^3);

            \draw[->, thick] (x1^3) -- (x2^3);
            \draw[->, thick] (x2^3) -- (x3^3);
            
            \draw[->, thick] (x3^3) to [out=-130, in=-90, looseness=0.6] (x1^1);

            \tikzset{decoration={snake,amplitude=.4mm,segment length=2mm, post length=1mm, pre length=0mm}}
                       
            \draw[->, green, decorate, thick] (x1^1) to (y1^1);
            \draw[->, green, decorate, thick] (y1^1) to [out = 60, in = 120, looseness=0.8] (y2^1);
            \draw[->, green, decorate, thick] (y2^1) to [out = 90, in = 0, looseness=1] (y3^1);
            \draw[->, green, decorate, thick] (y3^1) to [out = 180, in = 60] (y4^1);
            \draw[->, green, decorate, thick] (y4^1) -- (y1^2);
            \draw[->, green, decorate, thick] (y1^2) to [bend right=20] (y2^2);
            \draw[->, green, decorate, thick] (y2^2) to  (y3^2);
            \draw[->, green, decorate, thick] (y3^2) to [out = -60, in = -140, looseness=0.6] (y1^3);
            \draw[->, green, decorate, thick] (y1^3) to [out=120, in=-20, looseness=0.1] (y2^3);
            \draw[->, green, decorate, thick] (y2^3) to [out = -60, in = -50, looseness=1.2] (x1^1);

            \draw[->, magenta, dashed, thick] (y3^3) to [out = -90, in =-90, looseness=1] (y1^3);
            \draw[->, magenta, dashed, thick] (y1^3) to [out = -115, in = -65, looseness=1] (y1^2);
            \draw[->, magenta, dashed, thick] (y1^2) to (y1^1);
            \draw[->, magenta, dashed, thick] (y1^1) to (x1^3);
            \draw[->, magenta, dashed, thick] (x1^3) to [out = -120, in = -60] (x1^2);
            \draw[->, magenta, dashed, thick] (x1^2) to [out = -120, in = -60] (y3^3);

		\end{tikzpicture}
        \caption{Items moved in $\sigma_1, \sigma_2, \sigma_3$. The black solid edges represent $\sigma_1=(x_1^1\:x_2^1\:x_3^1\:x_4^1\:x_1^2\:x_2^2\:x_3^2\:x_4^2\:x_5^2\:x_1^3\:x_2^3\:x_3^3)$, the green wavy edges represent $\sigma_2=(y_1^1\:y_2^1\:y_3^1\:y_4^1\: y_1^2\:y_2^2\:y_3^2 \: y_1^3\:y_2^3\:x_1^1)$, and the magenta dashed edges represent $\sigma_3 = (y_1^3\:y_1^2 \:y_1^1\:x_1^3\:x_2^2\:y_3^3)$.} \label{fig:sigma123}
     \end{subfigure}
     \hfill
     \caption{An example of a CDG consisting of edges from $\pi_1$ and $\pi_2$ and depiction of the three item moves $\sigma_1, \sigma_2, \sigma_3$ in the proof of \Cref{lem:two permutations three cycles}.} \label{fig:two permutations three cycles}
     \end{figure}
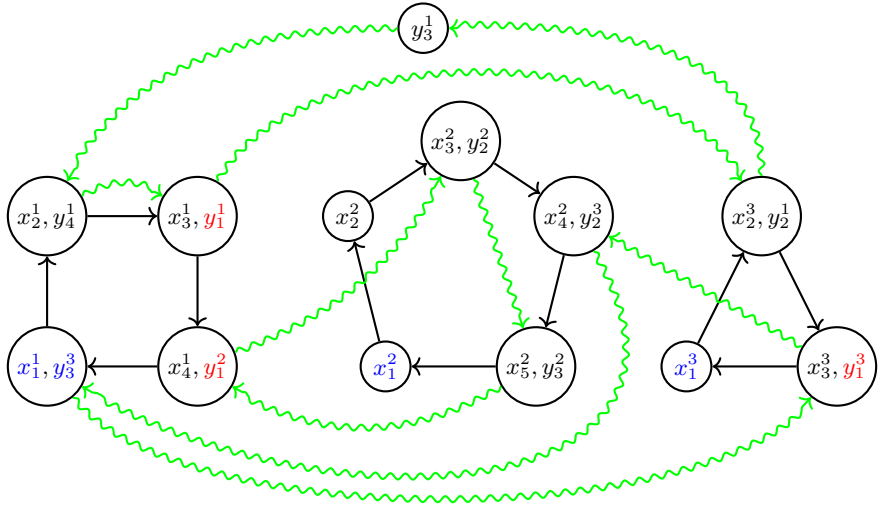
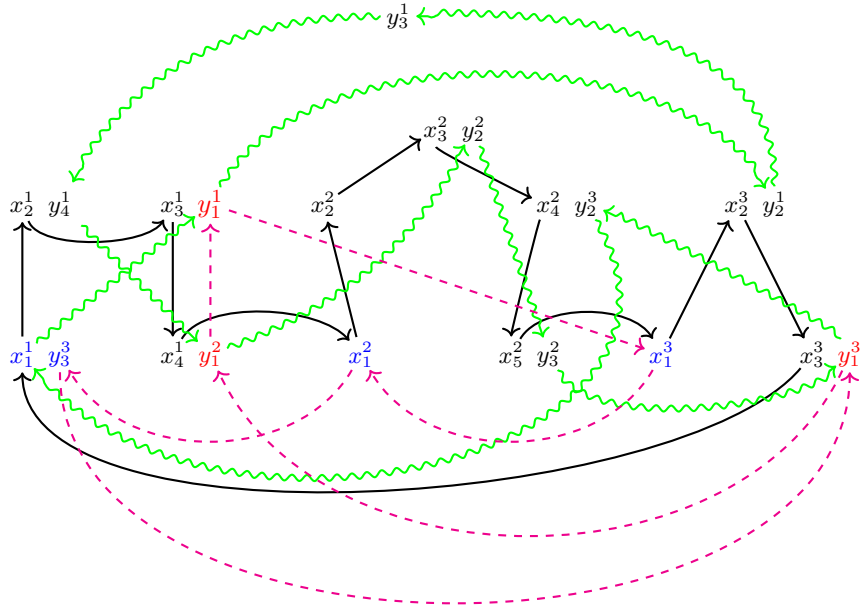

We now deduce the improved upper bound of $\kappa_1 + \left\lceil \kappa_2/2 \right\rceil$ for the diameter of general partition polytopes.

\begin{proof}[Proof of \Cref{thm:improved upper bound}]
    Let~$\kappa_1 \geq \cdots \geq \kappa_n \geq 0$ be integers summing to~$n$, and let~$p, p' : [m] \to [n]$ be two $(m,n)$-partitions with shape $(\kappa_1, \ldots, \kappa_n)$. By \Cref{prop:diameter reformulations}, it suffices to exhibit an expression of~$p'$ in the form~$p' = p\sigma_d \cdots \sigma_1$ for some sequence~$(\sigma_1, \ldots, \sigma_d)$ of~$p$-cycles of length~$d := \kappa_1 + \left \lceil \kappa_2 /2 \right \rceil$. Let~$\sigma_1, \ldots, \sigma_{\kappa_1 - \kappa_2}$ and~$\pi_1, \ldots, \pi_{\kappa_2}$ be disjoint~$p$-balanced permutations in~$S_m$ as in \Cref{cor:decomposition into balanced permutations}, so that~$\sigma_1, \ldots, \sigma_{\kappa_1 - \kappa_2}$ are~$p$-cycles, and
    \begin{equation}\label{eq:initial decomposition}
        p' = p\sigma_{\kappa_2-\kappa_1}\cdots \sigma_1\pi_{\kappa_2}\cdots \pi_1.
    \end{equation}
    By \Cref{lem:two permutations three cycles}, for each~$1 \leq i \leq \left\lfloor \kappa_2/2 \right\rfloor$, we can find~$p$-cycles~$\sigma_{\kappa_1-\kappa_2 + 3i-2}, \sigma_{\kappa_1-\kappa_2 + 3i-1}, \sigma_{\kappa_1-\kappa_2 + 3i}$ such that
    \begin{align*}
        p\sigma_{\kappa_1-\kappa_2 + 3i}\sigma_{\kappa_1-\kappa_2 + 3i-1}\sigma_{\kappa_1-\kappa_2 + 3i-2} &= p\pi_{2i}\pi_{2i-1}; \\
        \supp(\sigma_{\kappa_1-\kappa_2 + 3i}\sigma_{\kappa_1-\kappa_2 + 3i-1}\sigma_{\kappa_1-\kappa_2 + 3i-2}) &= \supp(\pi_{2i}\pi_{2i-1}).
    \end{align*}
    So far we have defined~$p$-cycles~$\sigma_1, \ldots, \sigma_t$ for
    \[t := \kappa_1 - \kappa_2 + 3\left \lfloor \kappa_2/2 \right \rfloor = \left \{\begin{array}{l l} 
    d & \text{if~$\kappa_2$ is even}, \\
    d - 2 & \text{if~$\kappa_2$ is odd}.
    \end{array}\right.\]
    If~$\kappa_2$ is odd, then we can also find two~$p$-cycles~$\sigma_{d-1}, \sigma_{d}$ with~$\pi_{\kappa_2} = \sigma_{d}\sigma_{d-1}$ by \Cref{lem:one permutation two cycles}. Now because the~$\sigma_1, \ldots, \sigma_{\kappa_1 - \kappa_2}, \pi_1, \ldots, \pi_{\kappa_2}$ have disjoint supports, we can replace each~$\pi_{2i}\pi_{2i-1}$ with the product~$\sigma_{\kappa_1-\kappa_2 + 3i}\sigma_{\kappa_1-\kappa_2 + 3i-1}\sigma_{\kappa_1-\kappa_2 + 3i-2}$ ($1 \leq i \leq \left \lfloor \kappa_2/2 \right \rfloor)$ in \eqref{eq:initial decomposition}, and if~$\kappa_2$ is odd, then we can additionally replace~$\pi_{\kappa_2}$ with~$\sigma_{d}\sigma_{d-1}$ to obtain
    \[p' = p\sigma_d \cdots \sigma_1,\]
    as desired.
\end{proof}

\subsection{Lower bound}\label{sec:lowerboundpp}
Next, we generalize the lower bound from~\cite{borgwardt2013diameter} on the diameter of partition polytopes. The following result holds for all values $\kappa_1,\dots,\kappa_n$ and $n\geq 4$; exact diameters are known for $n\leq 3$~\cite{borgwardt2013diameter}. 

\begin{lemma}\label{lem:disjoint 2-cycles lower bound}
	For any integers~$n \geq 4$ and~$\kappa_1 \geq \cdots \geq \kappa_n \geq 0$, the diameter of the partition polytope~$\PP(\kappa_1, \ldots, \kappa_n)$ is at least
	\[\left\{\begin{array}{c l}
	\left \lceil \frac{4(2\kappa_2 + 2\kappa_4 + \cdots + 2\kappa_n)}{3n} \right \rceil & \text{if~$n$ is even}, \\
	\left \lceil\frac{4(2\kappa_2 + 2\kappa_4 + \cdots + 2\kappa_{n-3} + 3\kappa_n)}{3n+1} \right \rceil & \text{if~$n$ is odd}.
	\end{array}\right.\]
\end{lemma}
\begin{proof}
	Consider, at first, an arbitrary CDG~$G = ([n], [m], p, p')$ and a resolution~$(p_0, \ldots, p_d)$ of~$(p,p')$ of length~$d$, encoded by~$(\tau_1,\dots,\tau_d)$. Let~$S$ denote the set of items~$x\in [m]$ such that~$p(x) \neq p'(x)$ for all~$x\in S$ (that is,~$S$ is the set of non-loop edges of~$G$), and for~$0\leq i\leq d$, define
	\[s_i := |\{x \in S : p_i(x) \neq p(x)\}| + |\{x \in S : p_i(x) = p'(x)\}|.\]
    Informally, the values~$s_i$ measure the progress towards~$p'$ made by the first~$i$ cyclic exchanges:  each item~$x$ contributes 1 to the count if~$x$ has ``departed'' its original cluster (i.e.,~$p_i(x)\neq p(x)$), and contributes an additional count if~$x$ has ``arrived'' at its desired cluster (i.e.,~$p_i(x)=p'(x)$).
    Clearly, we have~$s_0=0$, and after the final cyclic exchange, when every item~$x\in S$ has left~$p(x)$ and arrived at~$p'(x)$, we have~$s_d=2|S|$.
	
	If~$n$ is even, take~$G$ to be the digraph consisting of~$n/2$ vertex-disjoint cycles of length 2 with edge multiplicities~$\kappa_2, \kappa_4, \ldots, \kappa_n$, and possibly some loops. Since~$G$ has~$2\kappa_2 + 2\kappa_4 + \cdots + 2\kappa_n$ non-loop edges, we have~$s_d = 4\kappa_2 + 4\kappa_4 + \cdots + 4\kappa_n$. Note that for each~$i\in[d]$,
\begin{align}\label{eq:progress from ith step}
s_i - s_{i-1} 
&\leq |\{x \in S : p(x) = p_{i-1}(x) \neq p_{i}(x)\}| + |\{x \in S : p_{i-1}(x) \neq p_{i}(x) = p'(x)\}|.
\end{align}
Informally, the former term on the right-hand side counts the number of items~$x$ that leave~$p(x)$ precisely with the~$i$-th cyclic exchange; the latter counts the number of objects~$x$ that arrive at~$p'(x)$ precisely with the~$i$-th cyclic exchange. We estimate the expression on the right-hand side of \eqref{eq:progress from ith step} as follows. Since 
$|\supp(\tau_i)|\leq n$ (because~$\tau_i$ is a~$p_{i-1}$-cycle), at most~$n$ objects contribute at least~$1$.
Moreover, since~$x$ can only contribute~$2$ if the~$i$-th cyclic exchange moves~$x$ from~$p(x)$ to~$p'(x)$, at most~$n/2$ objects can contribute~$2$. Indeed, if we had at least~$n/2 + 1$ objects each contributing~$2$, then, by the Pigeonhole Principle and since~$n\geq 4$, the support of~$\tau_i$ must contain both edges of some cycle of length two and at least one other edge in~$G$, which is impossible because~$\tau_i$ is a~$p_{i-1}$-cycle. This means that~$s_i - s_{i-1} \leq 3n/2$. Altogether, this gives
	\[4\kappa_2 + 4\kappa_4 + \cdots + 4\kappa_n = \sum_{i=1}^d (s_i - s_{i-1}) \leq \frac{3n d}{2},\]
	so at least~$\displaystyle{\left\lceil \frac{4(2\kappa_2 + 2\kappa_4 + \cdots + 2\kappa_n)}{3n} \right \rceil}$ cyclic exchanges are required to resolve~$G$.
	
	If~$n$ is odd, take~$G$ to be the digraph consisting of~$(n-3)/2$ vertex-disjoint cycles of length 2 with edge multiplicities~$\kappa_2, \kappa_4, \ldots, \kappa_{n-3}$, one cycle of length 3 with edge multiplicity~$\kappa_n$ that is vertex-disjoint from the cycles of length 2, and possibly some loops. See Figure \ref{fig:lowerbound} for an example with~$n=11$,~$\kappa_2=\kappa_4=3$, and~$\kappa_6=\kappa_8=\kappa_{11}=2$. Following the same argument as in the even case, note that~$G$ has~$2\kappa_2 + 2\kappa_4 + \cdots + 2\kappa_{n-3} + 3\kappa_n$ non-loop edges. Again, \eqref{eq:progress from ith step} holds for each~$i \in [d]$. In this case, at most~$n$ objects contribute at least~$1$ to the right-hand side, and at most~$(n+1)/2$ objects contribute~$2$, so~$s_i - s_{i-1} \leq \frac{3n+1}{2}$. This gives
	\[4\kappa_2 + 4\kappa_4 + \cdots + 4\kappa_{n-3} + 6\kappa_n \leq \frac{(3n+1)d}{2},\]
    so at least~$\left \lceil\frac{4(2\kappa_2 + 2\kappa_4 + \cdots + 2\kappa_{n-3} + 3\kappa_n)}{3n+1} \right \rceil$ cyclic exchanges are required to resolve~$G$.
\end{proof}

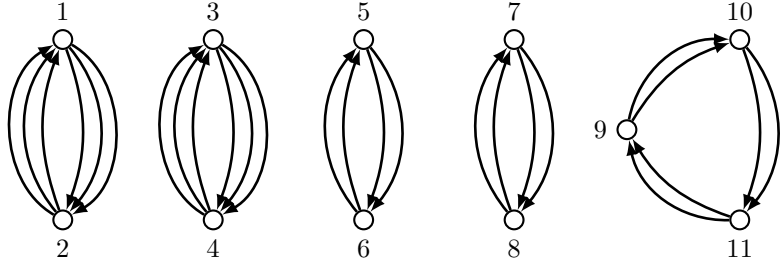
\begin{figure}
\centering
\begin{tikzpicture}[vertices/.style={draw, fill=black, circle, inner sep=0pt, minimum size = 4pt, outer sep=0pt}]
\node[knode,minimum height=0.25cm,label=above:{1}] (c_1) at (0, 2.4)   {};
\node[knode,minimum height=0.25cm,label=below:{2}] (c_2) at (0, 0)     {};
\node[knode,minimum height=0.25cm,label=above:{3}] (c_3) at (2, 2.4)   {};
\node[knode,minimum height=0.25cm,label=below:{4}] (c_4) at (2, 0)     {};
\node[knode,minimum height=0.25cm,label=above:{5}] (c_5) at (4, 2.4)   {};
\node[knode,minimum height=0.25cm,label=below:{6}] (c_6) at (4, 0)     {};
\node[knode,minimum height=0.25cm,label=above:{7}] (c_7) at (6, 2.4)   {};
\node[knode,minimum height=0.25cm,label=below:{8}] (c_8) at (6, 0)     {};
\node[knode,minimum height=0.25cm, label=left:{9}]  (c_9) at (7.5, 1.2){};
\node[knode,minimum height=0.25cm,label=above:{10}] (c_10) at (9, 2.4) {};
\node[knode,minimum height=0.25cm,label=below:{11}] (c_11) at (9, 0)   {};

\foreach \to/\from in {
	c_2/c_1, c_4/c_3, c_6/c_5, c_8/c_7}
\path[draw=black, line width= 1,  ->, >=latex]  (\to) edge[bend left=20] (\from);

\foreach \to/\from in {
	c_2/c_1, c_4/c_3, c_6/c_5, c_8/c_7}
\path[draw=black, line width= 1,  ->, >=latex]  (\to) edge[bend left=40] (\from);

\foreach \to/\from in {
	c_2/c_1, c_4/c_3}
\path[draw=black, line width= 1,  ->, >=latex]  (\to) edge[bend left=60] (\from);

\foreach \to/\from in {
	c_1/c_2, c_3/c_4, c_5/c_6, c_7/c_8,c_9/c_10,c_10/c_11,c_11/c_9}
\path[draw=black, line width= 1,  ->, >=latex]  (\to) edge[bend left=20] (\from);

\foreach \to/\from in {
	c_1/c_2, c_3/c_4, c_5/c_6, c_7/c_8,c_9/c_10,c_10/c_11,c_11/c_9}
\path[draw=black, line width= 1,  ->, >=latex]  (\to) edge[bend left=40] (\from);

\foreach \to/\from in {
	c_1/c_2, c_3/c_4}
\path[draw=black, line width= 1,  ->, >=latex]  (\to) edge[bend left=60] (\from);

\end{tikzpicture}
\caption{A CDG for $n$ odd, as constructed in the proof of Lemma \ref{lem:disjoint 2-cycles lower bound}.}\label{fig:lowerbound}
\end{figure}

Next, we exhibit that the bound from Lemma \ref{lem:disjoint 2-cycles lower bound} is not tight in general. The construction we consider is a special instance of the one given in \Cref{lem:disjoint 2-cycles lower bound} (with $n=6$ and $\kappa_1 = \cdots = \kappa_6 = 3$), but we tighten our analysis in order to give exact diameter $5$ in this case.

\begin{lemma}\label{lem:diamof33pp}
	The diameter of the partition polytope $\PP(3, 3, 3, 3, 3, 3)$ is $5$.
\end{lemma}
\begin{proof}
    Theorem~\ref{thm:improved upper bound} and Lemma~\ref{lem:disjoint 2-cycles lower bound} imply that the diameter of $\PP(3, 3, 3, 3, 3, 3)$ is at most 5 and at least 4 respectively.
    Suppose to the contrary that the diameter of $\PP(3, 3, 3, 3, 3, 3)$ is $4$.

    Similarly as in the proof of Lemma~\ref{lem:disjoint 2-cycles lower bound}, we will consider a CDG consisting of three vertex-disjoint cycles of length 2 each with edge multiplicity $3$. 
	Let $n = 6$ and $m = 18$. Write $[m] = \{a_i^j : 1 \leq i,j \leq 3\} \cup \{b_i^j : 1 \leq i,j \leq 3\}$, and write $[n] = \{A_i : 1 \leq i \leq 3\} \cup \{B_i : 1 \leq i \leq 3\}$. Let $p : [m] \to [n]$ be given by $p(a_i^j) = B_i$ and $p(b_i^j) = A_i$ for all $1 \leq i,j \leq 3$, and let $p' : [m] \to [n]$ be given by $p'(a_i^j) = A_i$ and $p(b_i^j) = B_i$ for all $1 \leq i,j \leq 3$. Define $G=([n],[m],p,p')$. 
    
    By the assumption that $\PP(3, 3, 3, 3, 3, 3)$ has diameter 4, there exists a resolution $(p_0, \ldots, p_4)$ of $(p,p')$, encoded by $(\tau_1,\tau_2,\tau_3,\tau_4)$. Note that $S := [n]$ is the set of non-loop edges of $G$. For $0 \leq k \leq 4$, define $s_k$ as in the proof of Lemma~\ref{lem:disjoint 2-cycles lower bound}. Then $s_0 = 0$, $s_4 = 36$, and for each $1 \leq k \leq 4$, we have 
    \begin{align}\label{eqn:moves}
s_k - s_{k-1} 
&\leq |\{x \in S : p(x) = p_{k-1}(x) \neq p_{k}(x)\}| + |\{x \in S : p_{k-1}(x) \neq p_{k}(x) = p'(x)\}|.
\end{align}
    As before, there are at most $|\supp(\tau_k)|\leq 6$ items $x\in\supp(\tau_k)$ that contribute at least 1 to the right-hand side, and at most 3 of those items contribute 2 because $G$ consists of three vertex-disjoint cycles of length 2.
    Hence
    $s_k - s_{k-1} \leq 9$ for each $1 \leq k \leq 4$, and since $s_0=0$ and $s_4=36$, we have $s_k - s_{k-1} = 9$ for each $1 \leq k \leq 4$. 

    To distinguish these contributions, let us say for $1 \leq k \leq 4$ that the cyclic exchange $\tau_k$ \emph{moves} an item $x\in\supp(\tau_k)$ from the cluster $p_{k-1}(x)$ to the cluster $p_k(x)$. A \emph{whole move} is a move of an item $x\in\supp(\tau_k)$ from $p(x)$ to $p'(x)$ (i.e.,~$p_{k-1}(x)=p(x)$ and $p_k(x)=p'(x)$). Note that whole moves are exactly the moves that contribute 2 to the right-hand side of \eqref{eqn:moves}. A \emph{half move} is a move of an item $x\in\supp(\tau_k)$ such that either $p_{k-1}(x)=p(x)$ or $p_k(x)=p'(x)$, but not both. Half moves are exactly the moves that contribute 1 to the right-hand side of \eqref{eqn:moves}. In general, there could be other kinds of moves (e.g.~$p_{k-1}(x)\neq p(x)$ and $p_k(x)\neq p'(x)$) but they do not contribute to the right-hand side of \eqref{eqn:moves}. 
    
    Since $\tau_k$ makes at most 6 moves and at most 3 of those are whole moves, the tightness of the inequality $s_k-s_{k-1}\leq 9$ implies that
	\begin{enumerate}[label={(\roman*)}]
		\item \label{item:move1}
        for each $1 \leq k \leq 4$, $\tau_k$ makes exactly six moves: three whole moves and the three half moves. 
        \end{enumerate}
    For $1 \leq k \leq 4$ and $1 \leq i \leq 3$, $\tau_k$ does not make two moves between $A_i$ and $B_i$ because this would form a cycle of length 2, whereas $\tau_k$ is a cycle of length 6 by \ref{item:move1}. Since the only whole moves in $G$ are between $A_i$ and $B_i$ for some $1 \leq i \leq 3$ and each $\tau_k$ makes three whole moves by \ref{item:move1}, we have that 
	\begin{enumerate}[label={(\roman*)}]
	\setcounter{enumi}{1}
        \item \label{item:move2}
        for each $1 \leq k \leq 4$ and $1 \leq i \leq 3$, $\tau_k$ makes one whole move between $A_i$ and $B_i$ (either from $A_i$ to $B_i$ or from $B_i$ to $A_i$). 
        \end{enumerate}
    Furthermore, since each $\tau_k$ makes only whole moves and half moves by \ref{item:move1}, we have the following:
	\begin{enumerate}[label={(\roman*)}]
	\setcounter{enumi}{2}
		\item If an item $x$ is ever moved to $p'(x)$, then it may never be moved again.\label{item:move3}
		\item If an item $x$ is ever moved to a cluster that is not $p'(x)$, then the next cyclic exchange that moves $x$ must take it to $p'(x)$. \label{item:move4}
	\end{enumerate}
	
    By \ref{item:move1} and \ref{item:move2}, we may assume by symmetry that $\tau_1 = \left(a_1^1 \: b_1^1 \: a_2^1 \: b_2^1 \: a_3^1 \: b_3^1\right)$. We display $\tau_1$ and the resulting $p_1$ below. 
	
	\begin{center}
    \resizebox{.45\linewidth}{!}
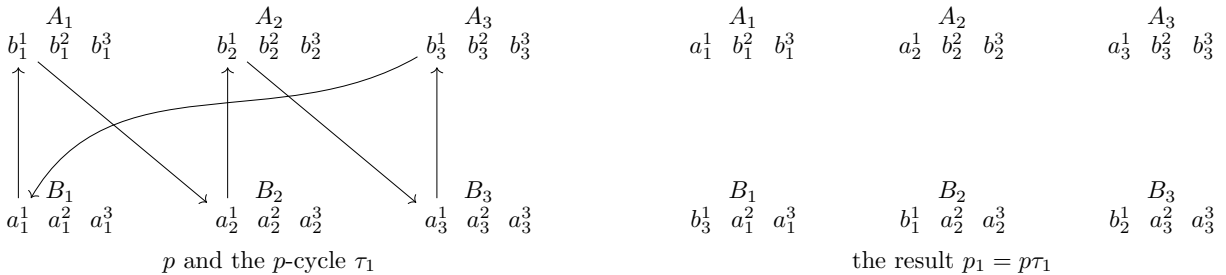
{
		\begin{tikzpicture}
			\node (A1) at (-2,2.5) {$A_1$};
			\node (A2) at (1,2.5) {$A_2$};
			\node (A3) at (4,2.5) {$A_3$};
			\node (B1) at (-2,0) {$B_1$};
			\node (B2) at (1,0) {$B_2$};
			\node (B3) at (4,0) {$B_3$};
			
			\node (a11) at (-2 - 0.6,2.5 - 0.4) {$b_1^1$};
			\node (a12) at (-2 - 0.0,2.5 - 0.4) {$b_1^2$};
			\node (a13) at (-2 + 0.6,2.5 - 0.4) {$b_1^3$};
			
			\node (a21) at (1 - 0.6,2.5 - 0.4) {$b_2^1$};
			\node (a22) at (1 - 0.0,2.5 - 0.4) {$b_2^2$};
			\node (a23) at (1 + 0.6,2.5 - 0.4) {$b_2^3$};
			
			\node (a31) at (4 - 0.6,2.5 - 0.4) {$b_3^1$};
			\node (a32) at (4 - 0.0,2.5 - 0.4) {$b_3^2$};
			\node (a33) at (4 + 0.6,2.5 - 0.4) {$b_3^3$};
			
			\node (b11) at (-2 - 0.6,0.0 - 0.4) {$a_1^1$};
			\node (b12) at (-2 - 0.0,0.0 - 0.4) {$a_1^2$};
			\node (b13) at (-2 + 0.6,0.0 - 0.4) {$a_1^3$};
			
			\node (b21) at (1 - 0.6,0.0 - 0.4) {$a_2^1$};
			\node (b22) at (1 - 0.0,0.0 - 0.4) {$a_2^2$};
			\node (b23) at (1 + 0.6,0.0 - 0.4) {$a_2^3$};
			
			\node (b31) at (4 - 0.6,0.0 - 0.4) {$a_3^1$};
			\node (b32) at (4 - 0.0,0.0 - 0.4) {$a_3^2$};
			\node (b33) at (4 + 0.6,0.0 - 0.4) {$a_3^3$};
			
			\draw[->] (b11)--(a11);
			\draw[->] (a11)--(b21);
			\draw[->] (b21)--(a21);
			\draw[->] (a21)--(b31);
			\draw[->] (b31)--(a31);
			\draw[->] (a31) to [out = -150, in = 60] (b11);
			\node at (1,-1) {$p$ and the $p$-cycle $\tau_1$};
		\end{tikzpicture}
        }
        \hfill
        \resizebox{.45\linewidth}{!}{
		\begin{tikzpicture}
			\node (A1) at (-2,2.5) {$A_1$};
			\node (A2) at (1,2.5) {$A_2$};
			\node (A3) at (4,2.5) {$A_3$};
			\node (B1) at (-2,0) {$B_1$};
			\node (B2) at (1,0) {$B_2$};
			\node (B3) at (4,0) {$B_3$};
			
			\node (a11) at (-2 - 0.6,2.5 - 0.4) {$a_1^1$};
			\node (a12) at (-2 - 0.0,2.5 - 0.4) {$b_1^2$};
			\node (a13) at (-2 + 0.6,2.5 - 0.4) {$b_1^3$};
			
			\node (a21) at (1 - 0.6,2.5 - 0.4) {$a_2^1$};
			\node (a22) at (1 - 0.0,2.5 - 0.4) {$b_2^2$};
			\node (a23) at (1 + 0.6,2.5 - 0.4) {$b_2^3$};
			
			\node (a31) at (4 - 0.6,2.5 - 0.4) {$a_3^1$};
			\node (a32) at (4 - 0.0,2.5 - 0.4) {$b_3^2$};
			\node (a33) at (4 + 0.6,2.5 - 0.4) {$b_3^3$};
			
			\node (b11) at (-2 - 0.6,0.0 - 0.4) {$b_3^1$};
			\node (b12) at (-2 - 0.0,0.0 - 0.4) {$a_1^2$};
			\node (b13) at (-2 + 0.6,0.0 - 0.4) {$a_1^3$};
			
			\node (b21) at (1 - 0.6,0.0 - 0.4) {$b_1^1$};
			\node (b22) at (1 - 0.0,0.0 - 0.4) {$a_2^2$};
			\node (b23) at (1 + 0.6,0.0 - 0.4) {$a_2^3$};
			
			\node (b31) at (4 - 0.6,0.0 - 0.4) {$b_2^1$};
			\node (b32) at (4 - 0.0,0.0 - 0.4) {$a_3^2$};
			\node (b33) at (4 + 0.6,0.0 - 0.4) {$a_3^3$};
			\node at (1,-1) {the result $p_1 = p\tau_1$};
		\end{tikzpicture}
        }
	\end{center}
	So far, for each $1 \leq i \leq 3$, none of $b_i^1,b_i^2,b_i^3$ are in $B_i$. Since there are three cyclic exchanges remaining, $\tau_2,\tau_3$, and $\tau_4$ must each move one of $b_i^1,b_i^2,b_i^3$ to $B_i$ for all $1 \leq i \leq 3$.
    Since $b_i^1\in B_{i+1}$ (indexing modulo $3$) and $b_i^2,b_i^3 \in A_i$, the second cyclic exchange $\tau_2$ either makes a half move on $b_i^1$ from $B_{i+1}$ to $B_i$ for some $1 \leq i \leq 3$, or makes three whole moves from $A_i$ to $B_i$ for all $1 \leq i \leq 3$.     
    We consider these two cases separately.

    \begin{description}
	\item[Case 1:] {\bf $\tau_2$ moves $b_i^1$ from $B_{i+1}$ to $B_i$ for some $1 \leq i \leq 3$ (indexing is modulo $3$).}

    Assume without loss of generality that $\tau_2$ moves $b_1^1$ from $B_2$ to $B_1$. This is a half move from $B_2$ to $B_1$, so by \ref{item:move2}, $\tau_2$ makes a whole move from $B_1$ to $A_1$ and a whole move from $A_2$ to $B_2$. Let us assume without loss of generality that $\tau_2$ moves $a_1^2$ from $B_1$ to $A_1$ and moves $b_2^2$ from $A_2$ to $B_2$. 
    
    Since $b_3^1$ is in $B_1$ and $\tau_2$ makes a move from $B_1$ to $A_1$, it does not move $b_3^1$ to $B_3$. Since $\tau_2$ must move one of $b_3^1,b_3^2,b_3^3$ to $B_3$ and $b_3^2,b_3^3\in A_3$, we may assume without loss of generality that $\tau_2$ moves $b_3^2$ from $A_3$ to $B_3$. 

    To complete a cycle of length 6, $\tau_2$ must make a move from $A_1$ to $A_3$ and a move from $B_3$ to $A_2$. By \ref{item:move3}, the item moved from $A_1$ to $A_3$ cannot be $a_1^1$ because $p'(a_1^1)=A_1$, so we may assume by symmetry that $b_1^2$ is moved from $A_1$ to $A_3$. By \ref{item:move4}, the next move that moves $b_2^1$ must take it to $p'(b_2^1)=B_2$, so we may assume by symmetry that $a_3^2$ is moved from $B_3$ to $A_2$.
    Altogether, this gives $\tau_2 = \left(b_1^1 \: a_1^2 \: b_1^2 \: b_3^2 \: a_3^2 \: b_2^2\right)$.
	
	\begin{center}
    \resizebox{.45\linewidth}{!}{
		\begin{tikzpicture}
			\node (A1) at (-2,2.5) {$A_1$};
			\node (A2) at (1,2.5) {$A_2$};
			\node (A3) at (4,2.5) {$A_3$};
			\node (B1) at (-2,0) {$B_1$};
			\node (B2) at (1,0) {$B_2$};
			\node (B3) at (4,0) {$B_3$};
			
			\node (a11) at (-2 - 0.6,2.5 - 0.4) {$a_1^1$};
			\node (a12) at (-2 - 0.0,2.5 - 0.4) {$b_1^2$};
			\node (a13) at (-2 + 0.6,2.5 - 0.4) {$b_1^3$};
			
			\node (a21) at (1 - 0.6,2.5 - 0.4) {$a_2^1$};
			\node (a22) at (1 - 0.0,2.5 - 0.4) {$b_2^2$};
			\node (a23) at (1 + 0.6,2.5 - 0.4) {$b_2^3$};
			
			\node (a31) at (4 - 0.6,2.5 - 0.4) {$a_3^1$};
			\node (a32) at (4 - 0.0,2.5 - 0.4) {$b_3^2$};
			\node (a33) at (4 + 0.6,2.5 - 0.4) {$b_3^3$};
			
			\node (b11) at (-2 - 0.6,0.0 - 0.4) {$b_3^1$};
			\node (b12) at (-2 - 0.0,0.0 - 0.4) {$a_1^2$};
			\node (b13) at (-2 + 0.6,0.0 - 0.4) {$a_1^3$};
			
			\node (b21) at (1 - 0.6,0.0 - 0.4) {$b_1^1$};
			\node (b22) at (1 - 0.0,0.0 - 0.4) {$a_2^2$};
			\node (b23) at (1 + 0.6,0.0 - 0.4) {$a_2^3$};
			
			\node (b31) at (4 - 0.6,0.0 - 0.4) {$b_2^1$};
			\node (b32) at (4 - 0.0,0.0 - 0.4) {$a_3^2$};
			\node (b33) at (4 + 0.6,0.0 - 0.4) {$a_3^3$};
			
			\draw[->] (b21) to [out=150, in = 30] (b12);
			\draw[->] (a22) to [out=-135, in = 90] (b21);
			\draw[->] (b12) to [out=135, in = -90] (a12);
			\draw[->] (a32) to [out=-45, in = 45] (b32);
			\draw[->] (a12) to [out=-30, in = -150] (a32);
			\draw[->] (b32) to [out=135, in=-90] (a22);
			
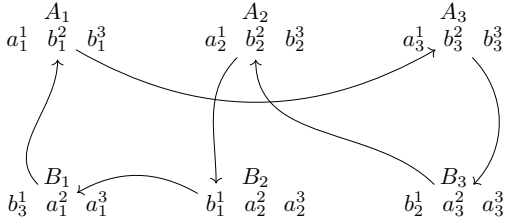
			\node at (1,-1) {Case 1: $p_1$ and the $p_1$-cycle $\tau_2$};
		\end{tikzpicture}
        } 
        \hfill
    \resizebox{.45\linewidth}{!}{		
		\begin{tikzpicture}
			\node (A1) at (-2,2.5) {$A_1$};
			\node (A2) at (1,2.5) {$A_2$};
			\node (A3) at (4,2.5) {$A_3$};
			\node (B1) at (-2,0) {$B_1$};
			\node (B2) at (1,0) {$B_2$};
			\node (B3) at (4,0) {$B_3$};
			
			\node (a11) at (-2 - 0.6,2.5 - 0.4) {$a_1^1$};
			\node (a12) at (-2 - 0.0,2.5 - 0.4) {$a_1^2$};
			\node (a13) at (-2 + 0.6,2.5 - 0.4) {$b_1^3$};
			
			\node (a21) at (1 - 0.6,2.5 - 0.4) {$a_2^1$};
			\node (a22) at (1 - 0.0,2.5 - 0.4) {$a_3^2$};
			\node (a23) at (1 + 0.6,2.5 - 0.4) {$b_2^3$};
			
			\node (a31) at (4 - 0.6,2.5 - 0.4) {$a_3^1$};
			\node (a32) at (4 - 0.0,2.5 - 0.4) {$b_1^2$};
			\node (a33) at (4 + 0.6,2.5 - 0.4) {$b_3^3$};
			
			\node (b11) at (-2 - 0.6,0.0 - 0.4) {$b_3^1$};
			\node (b12) at (-2 - 0.0,0.0 - 0.4) {$b_1^1$};
			\node (b13) at (-2 + 0.6,0.0 - 0.4) {$a_1^3$};
			
			\node (b21) at (1 - 0.6,0.0 - 0.4) {$b_2^2$};
			\node (b22) at (1 - 0.0,0.0 - 0.4) {$a_2^2$};
			\node (b23) at (1 + 0.6,0.0 - 0.4) {$a_2^3$};
			
			\node (b31) at (4 - 0.6,0.0 - 0.4) {$b_2^1$};
			\node (b32) at (4 - 0.0,0.0 - 0.4) {$b_3^2$};
			\node (b33) at (4 + 0.6,0.0 - 0.4) {$a_3^3$};
			
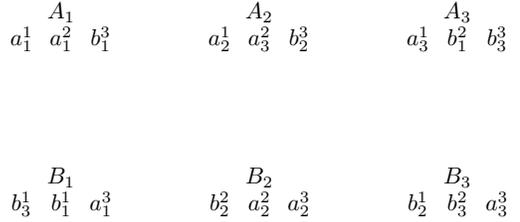
			\node at (1,-1) {Case 1: the resulting $p_2$};
		\end{tikzpicture}
        }
	\end{center}

	Now we consider the third cyclic exchange $\tau_3$. Note that $\tau_3$ is the second last cyclic exchange. Since both $a_2^2$ and $a_2^3$ are in $B_2$ and need to be moved to $A_2$, $\tau_3$ must make a whole move on one of them; we may assume that $\tau_3$ moves $a_2^2$ from $B_2$ to $A_2$. 
    Since both $b_2^1$ and $b_2^3$ need to be moved to $B_2$, one of these must be moved to $B_2$ by $\tau_3$. 
    This cannot be $b_2^3$ from $A_2$, or we would have a 2-cycle, so $\tau_3$ moves $b_2^1$ from $B_3$ to $B_2$. By \ref{item:move2}, $\tau_3$ makes a whole move from $A_3$ to $B_3$; the only such move is $b_3^3$ from $A_3$ to $B_3$. 

    Since both $a_3^2$ and $a_3^3$ need to be moved to $A_3$, $\tau_3$ must move one of these to $A_3$. But $a_3^3$ is in $B_3$ and $\tau_3$ makes a move from $B_3$ to $B_2$, so it cannot move $a_3^3$. Moreover, $a_3^2$ is in $A_2$ and moving it to $A_3$ would create a cycle of length 4, contradicting \ref{item:move1}. This concludes Case 1.
	
		\begin{center}
		\begin{tikzpicture}
			\node (A1) at (-2,2.5) {$A_1$};
			\node (A2) at (1,2.5) {$A_2$};
			\node (A3) at (4,2.5) {$\mathbf{A_3}$};
			\node (B1) at (-2,0) {$B_1$};
			\node (B2) at (1,0) {$B_2$};
			\node (B3) at (4,0) {$B_3$};
			
			\node (a11) at (-2 - 0.6,2.5 - 0.4) {$a_1^1$};
			\node (a12) at (-2 - 0.0,2.5 - 0.4) {$a_1^2$};
			\node (a13) at (-2 + 0.6,2.5 - 0.4) {$b_1^3$};
			
			\node (a21) at (1 - 0.6,2.5 - 0.4) {$a_2^1$};
			\node (a22) at (1 - 0.0,2.5 - 0.4) {$\mathbf{a_3^2}$};
			\node (a23) at (1 + 0.6,2.5 - 0.4) {$b_2^3$};
			
			\node (a31) at (4 - 0.6,2.5 - 0.4) {$a_3^1$};
			\node (a32) at (4 - 0.0,2.5 - 0.4) {$b_1^2$};
			\node (a33) at (4 + 0.6,2.5 - 0.4) {$b_3^3$};
			
			\node (b11) at (-2 - 0.6,0.0 - 0.4) {$b_3^1$};
			\node (b12) at (-2 - 0.0,0.0 - 0.4) {$b_1^1$};
			\node (b13) at (-2 + 0.6,0.0 - 0.4) {$a_1^3$};
			
			\node (b21) at (1 - 0.6,0.0 - 0.4) {$b_2^2$};
			\node (b22) at (1 - 0.0,0.0 - 0.4) {$a_2^2$};
			\node (b23) at (1 + 0.6,0.0 - 0.4) {$a_2^3$};
			
			\node (b31) at (4 - 0.6,0.0 - 0.4) {$b_2^1$};
			\node (b32) at (4 - 0.0,0.0 - 0.4) {$b_3^2$};
			\node (b33) at (4 + 0.6,0.0 - 0.4) {$\mathbf{a_3^3}$};
			
			\draw[->] (b22) to [out=135, in = -90] (a22);
			\draw[->] (b31) to [out=-150, in = -30] (b22);
			\draw[->] (a33) to [out=-135, in = 90] (b31);
			\draw[->][dotted] (a22) to [out=-30, in=-150] (a33);
			\draw[->][dotted] (b33) to (a33);
			
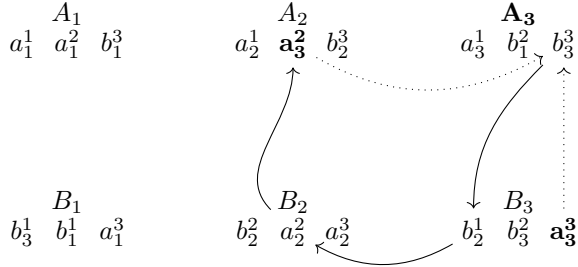
			\node at (1,-1.5) {Case 1: a contradiction trying to determine the $p_2$-cycle $\tau_3$};
		\end{tikzpicture}
	\end{center}
	
	\item[Case 2:] {\bf $\tau_2$ makes three whole moves from $A_i$ to $B_i$ for each $1 \leq i \leq 3$.}
    
    Assume without loss of generality that $\tau_2$ moves $b_i^2$ from $A_i$ to $B_i$ for each $1 \leq i \leq 3$. Up to symmetry and by \ref{item:move4}, there are exactly two ways that this can be completed to a $6$-cycle. Either we have $\tau_2 = \left(b_1^2 \: a_1^2 \: b_2^2 \: a_2^2 \: b_3^2 \: a_3^2\right)$ or $\tau_2 = \left(b_1^2 \: a_1^2 \: b_3^2 \: a_3^2 \: b_2^2 \: a_2^2\right)$, displayed below as Case 2a and Case 2b respectively. 

    	\begin{center}
        \resizebox{.45\linewidth}{!}{
		\begin{tikzpicture}
			\node (A1) at (-2,2.5) {$A_1$};
			\node (A2) at (1,2.5) {$A_2$};
			\node (A3) at (4,2.5) {$A_3$};
			\node (B1) at (-2,0) {$B_1$};
			\node (B2) at (1,0) {$B_2$};
			\node (B3) at (4,0) {$B_3$};
			
			\node (a11) at (-2 - 0.6,2.5 - 0.4) {$a_1^1$};
			\node (a12) at (-2 - 0.0,2.5 - 0.4) {$b_1^2$};
			\node (a13) at (-2 + 0.6,2.5 - 0.4) {$b_1^3$};
			
			\node (a21) at (1 - 0.6,2.5 - 0.4) {$a_2^1$};
			\node (a22) at (1 - 0.0,2.5 - 0.4) {$b_2^2$};
			\node (a23) at (1 + 0.6,2.5 - 0.4) {$b_2^3$};
			
			\node (a31) at (4 - 0.6,2.5 - 0.4) {$a_3^1$};
			\node (a32) at (4 - 0.0,2.5 - 0.4) {$b_3^2$};
			\node (a33) at (4 + 0.6,2.5 - 0.4) {$b_3^3$};
			
			\node (b11) at (-2 - 0.6,0.0 - 0.4) {$b_3^1$};
			\node (b12) at (-2 - 0.0,0.0 - 0.4) {$a_1^2$};
			\node (b13) at (-2 + 0.6,0.0 - 0.4) {$a_1^3$};
			
			\node (b21) at (1 - 0.6,0.0 - 0.4) {$b_1^1$};
			\node (b22) at (1 - 0.0,0.0 - 0.4) {$a_2^2$};
			\node (b23) at (1 + 0.6,0.0 - 0.4) {$a_2^3$};
			
			\node (b31) at (4 - 0.6,0.0 - 0.4) {$b_2^1$};
			\node (b32) at (4 - 0.0,0.0 - 0.4) {$a_3^2$};
			\node (b33) at (4 + 0.6,0.0 - 0.4) {$a_3^3$};

			\draw[->] (a12) to [out=-90, in = 135] (b12);
			\draw[->] (b12) to (a22);
			\draw[->] (a22) to [out=-90, in = 135] (b22);
			\draw[->] (b22) to (a32);
			\draw[->] (a32) to [out=-90, in = 135] (b32);
			\draw[->] (b32) to [out=150, in=-60, looseness=0.4] (a12);
			
			\node at (1,-1.5) {Case 2a: $p_1$ and $p_1$-cycle $\tau_2=(b_1^2 \: a_1^2 \: b_2^2 \: a_2^2 \: b_3^2 \: a_3^2)$};

		\end{tikzpicture}
}\hfill 
\resizebox{.45\linewidth}{!}{
        		\begin{tikzpicture}
			\node (A1) at (-2,2.5) {$A_1$};
			\node (A2) at (1,2.5) {$A_2$};
			\node (A3) at (4,2.5) {$A_3$};
			\node (B1) at (-2,0) {$B_1$};
			\node (B2) at (1,0) {$B_2$};
			\node (B3) at (4,0) {$B_3$};
			
			\node (a11) at (-2 - 0.6,2.5 - 0.4) {$a_1^1$};
			\node (a12) at (-2 - 0.0,2.5 - 0.4) {$b_1^2$};
			\node (a13) at (-2 + 0.6,2.5 - 0.4) {$b_1^3$};
			
			\node (a21) at (1 - 0.6,2.5 - 0.4) {$a_2^1$};
			\node (a22) at (1 - 0.0,2.5 - 0.4) {$b_2^2$};
			\node (a23) at (1 + 0.6,2.5 - 0.4) {$b_2^3$};
			
			\node (a31) at (4 - 0.6,2.5 - 0.4) {$a_3^1$};
			\node (a32) at (4 - 0.0,2.5 - 0.4) {$b_3^2$};
			\node (a33) at (4 + 0.6,2.5 - 0.4) {$b_3^3$};
			
			\node (b11) at (-2 - 0.6,0.0 - 0.4) {$b_3^1$};
			\node (b12) at (-2 - 0.0,0.0 - 0.4) {$a_1^2$};
			\node (b13) at (-2 + 0.6,0.0 - 0.4) {$a_1^3$};
			
			\node (b21) at (1 - 0.6,0.0 - 0.4) {$b_1^1$};
			\node (b22) at (1 - 0.0,0.0 - 0.4) {$a_2^2$};
			\node (b23) at (1 + 0.6,0.0 - 0.4) {$a_2^3$};
			
			\node (b31) at (4 - 0.6,0.0 - 0.4) {$b_2^1$};
			\node (b32) at (4 - 0.0,0.0 - 0.4) {$a_3^2$};
			\node (b33) at (4 + 0.6,0.0 - 0.4) {$a_3^3$};

			\draw[->] (a12) to [out=-90, in = 135] (b12);
			\draw[->] (b12) to [out=30, in=-120, looseness=0.4](a32);
			\draw[->] (a32) to [out=-90, in = 135] (b32);
			\draw[->] (b32) to (a22);
			\draw[->] (a22) to [out=-90, in = 135] (b22);
			\draw[->] (b22) to [out=150, in=-45, looseness=0.2] (a12);
			
			\node at (1,-1.5) {Case 2b: $p_1$ and $p_1$-cycle $\tau_2=(b_1^2 \: a_1^2 \: b_3^2 \: a_3^2 \: b_2^2 \: a_2^2)$};

		\end{tikzpicture}
}
	\end{center}

    	\begin{center}
        \resizebox{.45\linewidth}{!}{
		\begin{tikzpicture}
			\node (A1) at (-2,2.5) {$A_1$};
			\node (A2) at (1,2.5) {$A_2$};
			\node (A3) at (4,2.5) {$A_3$};
			\node (B1) at (-2,0) {$B_1$};
			\node (B2) at (1,0) {$B_2$};
			\node (B3) at (4,0) {$B_3$};
			
			\node (a11) at (-2 - 0.6,2.5 - 0.4) {$a_1^1$};
			\node (a12) at (-2 - 0.0,2.5 - 0.4) {$a_3^2$};
			\node (a13) at (-2 + 0.6,2.5 - 0.4) {$b_1^3$};
			
			\node (a21) at (1 - 0.6,2.5 - 0.4) {$a_2^1$};
			\node (a22) at (1 - 0.0,2.5 - 0.4) {$a_1^2$};
			\node (a23) at (1 + 0.6,2.5 - 0.4) {$b_2^3$};
			
			\node (a31) at (4 - 0.6,2.5 - 0.4) {$a_3^1$};
			\node (a32) at (4 - 0.0,2.5 - 0.4) {$a_2^2$};
			\node (a33) at (4 + 0.6,2.5 - 0.4) {$b_3^3$};
			
			\node (b11) at (-2 - 0.6,0.0 - 0.4) {$b_3^1$};
			\node (b12) at (-2 - 0.0,0.0 - 0.4) {$b_1^2$};
			\node (b13) at (-2 + 0.6,0.0 - 0.4) {$a_1^3$};
			
			\node (b21) at (1 - 0.6,0.0 - 0.4) {$b_1^1$};
			\node (b22) at (1 - 0.0,0.0 - 0.4) {$b_2^2$};
			\node (b23) at (1 + 0.6,0.0 - 0.4) {$a_2^3$};
			
			\node (b31) at (4 - 0.6,0.0 - 0.4) {$b_2^1$};
			\node (b32) at (4 - 0.0,0.0 - 0.4) {$b_3^2$};
			\node (b33) at (4 + 0.6,0.0 - 0.4) {$a_3^3$};

			\node at (1,-1.5) {Case 2a: the result $p_2=p_1\tau_2$};

		\end{tikzpicture}

}\hfill \resizebox{.45\linewidth}{!}{
        		\begin{tikzpicture}
			\node (A1) at (-2,2.5) {$A_1$};
			\node (A2) at (1,2.5) {$A_2$};
			\node (A3) at (4,2.5) {$A_3$};
			\node (B1) at (-2,0) {$B_1$};
			\node (B2) at (1,0) {$B_2$};
			\node (B3) at (4,0) {$B_3$};
			
			\node (a11) at (-2 - 0.6,2.5 - 0.4) {$a_1^1$};
			\node (a12) at (-2 - 0.0,2.5 - 0.4) {$a_2^2$};
			\node (a13) at (-2 + 0.6,2.5 - 0.4) {$b_1^3$};
			
			\node (a21) at (1 - 0.6,2.5 - 0.4) {$a_2^1$};
			\node (a22) at (1 - 0.0,2.5 - 0.4) {$a_3^2$};
			\node (a23) at (1 + 0.6,2.5 - 0.4) {$b_2^3$};
			
			\node (a31) at (4 - 0.6,2.5 - 0.4) {$a_3^1$};
			\node (a32) at (4 - 0.0,2.5 - 0.4) {$a_1^2$};
			\node (a33) at (4 + 0.6,2.5 - 0.4) {$b_3^3$};
			
			\node (b11) at (-2 - 0.6,0.0 - 0.4) {$b_3^1$};
			\node (b12) at (-2 - 0.0,0.0 - 0.4) {$b_1^2$};
			\node (b13) at (-2 + 0.6,0.0 - 0.4) {$a_1^3$};
			
			\node (b21) at (1 - 0.6,0.0 - 0.4) {$b_1^1$};
			\node (b22) at (1 - 0.0,0.0 - 0.4) {$b_2^2$};
			\node (b23) at (1 + 0.6,0.0 - 0.4) {$a_2^3$};
			
			\node (b31) at (4 - 0.6,0.0 - 0.4) {$b_2^1$};
			\node (b32) at (4 - 0.0,0.0 - 0.4) {$b_3^2$};
			\node (b33) at (4 + 0.6,0.0 - 0.4) {$a_3^3$};

			\node at (1,-1.5) {Case 2b: the result $p_2=p_1\tau_2$};

		\end{tikzpicture}
}
	\end{center}
    
    The resulting clustering difference graphs CDG$(p_2,p')$ and their simplified depictions omitting item and cluster labels are displayed below.
    
    	\begin{center}
        \resizebox{.45\linewidth}{!}{
		\begin{tikzpicture}
			\node (A1) at (-2,2.5) {$A_1$};
			\node (A2) at (1,2.5) {$A_2$};
			\node (A3) at (4,2.5) {$A_3$};
			\node (B1) at (-2,0) {$B_1$};
			\node (B2) at (1,0) {$B_2$};
			\node (B3) at (4,0) {$B_3$};
			
			\node (a11) at (-2 - 0.6,2.5 - 0.4) {$a_1^1$};
			\node (a12) at (-2 - 0.0,2.5 - 0.4) {$a_3^2$};
			\node (a13) at (-2 + 0.6,2.5 - 0.4) {$b_1^3$};
			
			\node (a21) at (1 - 0.6,2.5 - 0.4) {$a_2^1$};
			\node (a22) at (1 - 0.0,2.5 - 0.4) {$a_1^2$};
			\node (a23) at (1 + 0.6,2.5 - 0.4) {$b_2^3$};
			
			\node (a31) at (4 - 0.6,2.5 - 0.4) {$a_3^1$};
			\node (a32) at (4 - 0.0,2.5 - 0.4) {$a_2^2$};
			\node (a33) at (4 + 0.6,2.5 - 0.4) {$b_3^3$};
			
			\node (b11) at (-2 - 0.6,0.0 - 0.4) {$b_3^1$};
			\node (b12) at (-2 - 0.0,0.0 - 0.4) {$b_1^2$};
			\node (b13) at (-2 + 0.6,0.0 - 0.4) {$a_1^3$};
			
			\node (b21) at (1 - 0.6,0.0 - 0.4) {$b_1^1$};
			\node (b22) at (1 - 0.0,0.0 - 0.4) {$b_2^2$};
			\node (b23) at (1 + 0.6,0.0 - 0.4) {$a_2^3$};
			
			\node (b31) at (4 - 0.6,0.0 - 0.4) {$b_2^1$};
			\node (b32) at (4 - 0.0,0.0 - 0.4) {$b_3^2$};
			\node (b33) at (4 + 0.6,0.0 - 0.4) {$a_3^3$};

            \draw[->] (a12) to [out=-70, in = -110, looseness=0.6] (a32);
            \draw[->] (a13) to [bend right=10] (b13);
            \draw[->] (a22) to [out=-135, in = -45] (a12);
            \draw[->] (a23) to [bend right=10] (b23);
            \draw[->] (a32) to [out=-135, in = -45] (a22);
            \draw[->] (a33) to [bend right=10] (b33);

            \draw[->] (b11) to [out=-50, in = -130, looseness=0.6] (b31);
            \draw[->] (b13) to [bend right=10] (a13);
            \draw[->] (b21) to [out=-150, in = -30] (b11);
            \draw[->] (b23) to [bend right=10] (a23);
            \draw[->] (b31) to [out=-150, in = -30] (b21);
            \draw[->] (b33) to [bend right=10] (a33);

			\node at (1,-2) {Case 2a: CDG$(p_2,p')$};

		\end{tikzpicture}

}\hfill \resizebox{.45\linewidth}{!}{
        		\begin{tikzpicture}
			\node (A1) at (-2,2.5) {$A_1$};
			\node (A2) at (1,2.5) {$A_2$};
			\node (A3) at (4,2.5) {$A_3$};
			\node (B1) at (-2,0) {$B_1$};
			\node (B2) at (1,0) {$B_2$};
			\node (B3) at (4,0) {$B_3$};
			
			\node (a11) at (-2 - 0.6,2.5 - 0.4) {$a_1^1$};
			\node (a12) at (-2 - 0.0,2.5 - 0.4) {$a_2^2$};
			\node (a13) at (-2 + 0.6,2.5 - 0.4) {$b_1^3$};
			
			\node (a21) at (1 - 0.6,2.5 - 0.4) {$a_2^1$};
			\node (a22) at (1 - 0.0,2.5 - 0.4) {$a_3^2$};
			\node (a23) at (1 + 0.6,2.5 - 0.4) {$b_2^3$};
			
			\node (a31) at (4 - 0.6,2.5 - 0.4) {$a_3^1$};
			\node (a32) at (4 - 0.0,2.5 - 0.4) {$a_1^2$};
			\node (a33) at (4 + 0.6,2.5 - 0.4) {$b_3^3$};
			
			\node (b11) at (-2 - 0.6,0.0 - 0.4) {$b_3^1$};
			\node (b12) at (-2 - 0.0,0.0 - 0.4) {$b_1^2$};
			\node (b13) at (-2 + 0.6,0.0 - 0.4) {$a_1^3$};
			
			\node (b21) at (1 - 0.6,0.0 - 0.4) {$b_1^1$};
			\node (b22) at (1 - 0.0,0.0 - 0.4) {$b_2^2$};
			\node (b23) at (1 + 0.6,0.0 - 0.4) {$a_2^3$};
			
			\node (b31) at (4 - 0.6,0.0 - 0.4) {$b_2^1$};
			\node (b32) at (4 - 0.0,0.0 - 0.4) {$b_3^2$};
			\node (b33) at (4 + 0.6,0.0 - 0.4) {$a_3^3$};

            \draw[<-] (a12) to [out=-70, in = -110, looseness=0.6] (a32);
            \draw[->] (a13) to [bend right=10] (b13);
            \draw[<-] (a22) to [out=-135, in = -45] (a12);
            \draw[->] (a23) to [bend right=10] (b23);
            \draw[<-] (a32) to [out=-135, in = -45] (a22);
            \draw[->] (a33) to [bend right=10] (b33);

            \draw[->] (b11) to [out=-50, in = -130, looseness=0.6] (b31);
            \draw[->] (b13) to [bend right=10] (a13);
            \draw[->] (b21) to [out=-150, in = -30] (b11);
            \draw[->] (b23) to [bend right=10] (a23);
            \draw[->] (b31) to [out=-150, in = -30] (b21);
            \draw[->] (b33) to [bend right=10] (a33);
			
			\node at (1,-2) {Case 2b: CDG$(p_2,p')$};

		\end{tikzpicture}
}
	\end{center}

    	\begin{center}
        \resizebox{.45\linewidth}{!}{
		\begin{tikzpicture}
			\node[knode,minimum height=0.25cm] (A1) at (-2,2.5) {};
			\node[knode,minimum height=0.25cm] (A2) at (1,2.5)  {};
			\node[knode,minimum height=0.25cm] (A3) at (4,2.5)  {};
			\node[knode,minimum height=0.25cm] (B1) at (-2,0)   {};
			\node[knode,minimum height=0.25cm] (B2) at (1,0)    {};
			\node[knode,minimum height=0.25cm] (B3) at (4,0)    {};
			
			\node[outer sep=0.05cm] (a11) at (A1) {};
			\node[outer sep=0.05cm] (a12) at (A1) {};
			\node[outer sep=0.05cm] (a13) at (A1) {};
			\node[outer sep=0.05cm] (a21) at (A2) {};
			\node[outer sep=0.05cm] (a22) at (A2) {};
			\node[outer sep=0.05cm] (a23) at (A2) {};			
			\node[outer sep=0.05cm] (a31) at (A3) {};
			\node[outer sep=0.05cm] (a32) at (A3) {};
			\node[outer sep=0.05cm] (a33) at (A3) {};			
			\node[outer sep=0.05cm] (b11) at (B1) {};
			\node[outer sep=0.05cm] (b12) at (B1) {};
			\node[outer sep=0.05cm] (b13) at (B1) {};			
			\node[outer sep=0.05cm] (b21) at (B2) {};
			\node[outer sep=0.05cm] (b22) at (B2) {};
			\node[outer sep=0.05cm] (b23) at (B2) {};			
			\node[outer sep=0.05cm] (b31) at (B3) {};
			\node[outer sep=0.05cm] (b32) at (B3) {};
			\node[outer sep=0.05cm] (b33) at (B3) {};

            \draw[->] (a12) to [bend left, looseness=0.6] (a32);
            \draw[->] (a13) to [bend right=15] (b13);
            \draw[->] (a22) to (a12);
            \draw[->] (a23) to [bend right=15] (b23);
            \draw[->] (a32) to (a22);
            \draw[->] (a33) to [bend right=15] (b33);

            \draw[->] (b11) to [bend right, looseness=0.6] (b31);
            \draw[->] (b13) to [bend right=15] (a13);
            \draw[->] (b21) to (b11);
            \draw[->] (b23) to [bend right=15] (a23);
            \draw[->] (b31) to (b21);
            \draw[->] (b33) to [bend right=15] (a33);

			\node at (1,-1) {\small Case 2a: CDG$(p_2,p')$ without labels};

			\node at (-2 - 0.6,0.0 - 0.4) {};
			\node at (4 + 0.6,0.0 - 0.4) {};
            
		\end{tikzpicture}

}\hfill \resizebox{.45\linewidth}{!}{
        		\begin{tikzpicture}
			\node[knode,minimum height=0.25cm] (A1) at (-2,2.5) {};
			\node[knode,minimum height=0.25cm] (A2) at (1,2.5)  {};
			\node[knode,minimum height=0.25cm] (A3) at (4,2.5)  {};
			\node[knode,minimum height=0.25cm] (B1) at (-2,0)   {};
			\node[knode,minimum height=0.25cm] (B2) at (1,0)    {};
			\node[knode,minimum height=0.25cm] (B3) at (4,0)    {};
			
			\node[outer sep=0.05cm] (a11) at (A1) {};
			\node[outer sep=0.05cm] (a12) at (A1) {};
			\node[outer sep=0.05cm] (a13) at (A1) {};
			\node[outer sep=0.05cm] (a21) at (A2) {};
			\node[outer sep=0.05cm] (a22) at (A2) {};
			\node[outer sep=0.05cm] (a23) at (A2) {};			
			\node[outer sep=0.05cm] (a31) at (A3) {};
			\node[outer sep=0.05cm] (a32) at (A3) {};
			\node[outer sep=0.05cm] (a33) at (A3) {};			
			\node[outer sep=0.05cm] (b11) at (B1) {};
			\node[outer sep=0.05cm] (b12) at (B1) {};
			\node[outer sep=0.05cm] (b13) at (B1) {};			
			\node[outer sep=0.05cm] (b21) at (B2) {};
			\node[outer sep=0.05cm] (b22) at (B2) {};
			\node[outer sep=0.05cm] (b23) at (B2) {};			
			\node[outer sep=0.05cm] (b31) at (B3) {};
			\node[outer sep=0.05cm] (b32) at (B3) {};
			\node[outer sep=0.05cm] (b33) at (B3) {};

            \draw[<-] (a12) to [bend left, looseness=0.6] (a32);
            \draw[->] (a13) to [bend right=15] (b13);
            \draw[<-] (a22) to (a12);
            \draw[->] (a23) to [bend right=15] (b23);
            \draw[<-] (a32) to (a22);
            \draw[->] (a33) to [bend right=15] (b33);

            \draw[->] (b11) to [bend right, looseness=0.6] (b31);
            \draw[->] (b13) to [bend right=15] (a13);
            \draw[->] (b21) to (b11);
            \draw[->] (b23) to [bend right=15] (a23);
            \draw[->] (b31) to (b21);
            \draw[->] (b33) to [bend right=15] (a33);
			
			\node at (1,-1) {\small Case 2b: CDG$(p_2,p')$ without labels};

            \node at (-2 - 0.6,0.0 - 0.4) {};
			\node at (4 + 0.6,0.0 - 0.4) {};

		\end{tikzpicture}
}
	\end{center}

In either case, it is readily verified that CDG$(p_2,p')$ is a digraph on 6 vertices with 12 non-loop edges  that cannot be decomposed into two directed Hamiltonian cycles, hence it does not admit a resolution of length 2. This concludes Case 2, and the proof is complete. 
    \end{description}
\end{proof}

\section{Improved bounds on path and cycle odd-covers}\label{sec:odd}

The graphs in this section are simple and undirected.
Our main goal in this section is to prove the following two statements which make up \Cref{thm:oddcover3}. 
\begin{theorem}\label{thm:sum of three paths}
	Every Eulerian graph~$G$ of maximum degree~$4$ admits a path odd-cover of size at most~$3$.
\end{theorem}
\begin{theorem}\label{thm:sum of three cycles}
	Every Eulerian graph~$G$ of maximum degree~$4$ admits a cycle odd-cover of size at most~$3$.
\end{theorem}

We will then deduce \Cref{thm:odd cover general} as an immediate consequence, which we also restate in two parts.

\begin{theorem}\label{thm:path odd-cover general}
    Every graph $G$ admits a path odd-cover of size at most $\max \left\{\frac{\vodd(G)}{2}, \left \lceil\frac{\vodd(G)/2 + 3\Delta_e(G)}{4} \right \rceil\right\}$. In particular, if $G$ is Eulerian, then $G$ admits a path-odd cover of size at most $\left\lceil \frac{3}{4} \Delta(G) \right\rceil$. 
\end{theorem}

\begin{theorem}\label{thm:cycle odd-cover general}
    Every $n$-vertex Eulerian graph $G$ admits a cycle odd-cover of size at most $\left \lceil \frac{3}{4}\Delta(G) \right \rceil$. In fact, if the vertex degrees of $G$ are $d_1 \geq \cdots \geq d_n$, then $G$ admits a cycle odd-cover of size at most $d_1/2 + \lceil d_2/4 \rceil$.
\end{theorem}

We briefly discuss our strategy for proving the above theorems in Section \ref{sec:refinelinearforest}.
We then prove Theorems \ref{thm:sum of three paths} and \ref{thm:path odd-cover general} in Section \ref{sec:path odd-covers}, and prove Theorems \ref{thm:sum of three cycles} and \ref{thm:cycle odd-cover general} in Section \ref{sec:oddcycle}.

\subsection{Refining linear forest decompositions} \label{sec:refinelinearforest}

The \emph{endpoints} of a linear forest~$F$ are the vertices of degree 1 in~$F$. We denote the set of endpoints of~$F$ by~$\mathrm{end}(F)$. Given an Eulerian graph~$G$ of maximum degree~$4$ and a linear forest decomposition (or, more generally, a linear forest odd-cover)~$\mathcal F = \{F_1, F_2, F_3\}$ of~$G$, we use the following notation. For distinct~$i,j \in \{1,2,3\}$, we denote by~$R_{ij}(\mathcal F)$ the set~$\mathrm{end}(F_i) \cap \mathrm{end}(F_j)$ of common endpoints of~$F_i$ and~$F_j$, and we write~$r_{ij}(\mathcal F) := |R_{ij}(\mathcal F)|$. Note that~$R_{12}(\mathcal F)$,~$R_{13}(\mathcal F)$, and~$R_{23}(\mathcal F)$ are disjoint and partition the set~$\mathrm{end}(F_1) \cup \mathrm{end}(F_2) \cup \mathrm{end}(F_3)$ because~$G$ is Eulerian. For distinct~$i,j,k\in\{1,2,3\}$, we denote by~$\mathcal T_i(\mathcal F)$ the set of components of~$F_i$ with one endpoint in~$R_{ij}(\mathcal F)$ and the other in~$R_{ik}(\mathcal F)$, and we write~$t_i(\mathcal F) := |\mathcal T_i(\mathcal F)|$. As we will soon see (see \Cref{lem:linear forests parity control}), all of these parameters~$r_{ij}(\mathcal F)$ ($\{i,j\} \subseteq \{1,2,3\}$) and~$t_i(\mathcal F)$ ($1 \leq i \leq 3$) have the same parity, which we call the \emph{parity} of~$\mathcal F$ and denote by~$p(\mathcal F) \in \{0,1\}$. For each of these parameters, we may omit the argument~$\mathcal F$ if there is no danger of ambiguity. Figure \ref{fig:parity} depicts an example with even parity.

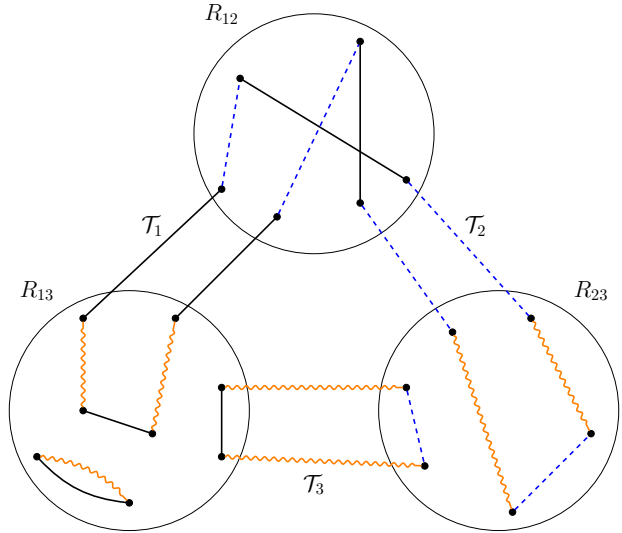
\begin{figure}[h]
\centering
\resizebox{0.5\textwidth}{!}{
\begin{tikzpicture}[vertices/.style={draw, fill=black, circle, inner sep=0pt, minimum size = 4pt, outer sep=0pt}]

\tikzset{decoration={snake,amplitude=.4mm,segment length=2mm, post length=1mm, pre length=0mm}}

\draw (0,0) circle (2.6cm);
\node at (-2, 2.6) {\Large $R_{13}$};
\draw (4,6) circle (2.6cm);
\node at (2, 8.6) {\Large $R_{12}$};
\draw (8,0) circle (2.6cm);
\node at (10, 2.6) {\Large $R_{23}$};

\node at (0.5, 4) {\Large $\mathcal{T}_1$};
\node at (7.5, 4) {\Large $\mathcal{T}_2$};
\node at (4, -1.6) {\Large $\mathcal{T}_3$};

\node[vertices] (c_1) at (-2, -1) {};
\node[vertices] (c_2) at (0, -2) {};
\node[vertices] (c_3) at (-1, 0) {};
\node[vertices] (c_4) at (0.5, -0.5) {};
\node[vertices] (c_5) at (-1, 2) {};
\node[vertices] (c_6) at (1, 2) {};
\node[vertices] (c_7) at (2, 0.5) {};
\node[vertices] (c_8) at (2, -1) {};

\node[vertices] (e_1) at (2, 4.8) {};
\node[vertices] (e_2) at (3.2, 4.2) {};
\node[vertices] (e_3) at (5, 8) {};
\node[vertices] (e_4) at (2.4, 7.2) {};
\node[vertices] (e_5) at (6, 5) {};
\node[vertices] (e_6) at (5, 4.5) {};

\node[vertices] (f_1) at (6, 0.5) {};
\node[vertices] (f_2) at (6.4, -1.2) {};
\node[vertices] (f_3) at (7, 1.7) {};
\node[vertices] (f_4) at (8.7, 2) {};
\node[vertices] (f_5) at (8.3, -2.2) {};
\node[vertices] (f_6) at (10, -0.5) {};

\path [draw=black, line width= 1] (c_1) edge[bend right=20] (c_2);
\path [draw=black, line width= 1] (c_3) edge (c_4);
\path [draw=black, line width= 1] (c_7) edge (c_8);
\path [draw=black, line width= 1] (c_5) edge (e_1);
\path [draw=black, line width= 1] (c_6) edge (e_2);
\path [draw=black, line width= 1] (e_3) edge (e_6);
\path [draw=black, line width= 1] (e_4) edge (e_5);

\path [draw=blue, dashed, line width= 1] (e_1) edge (e_4);
\path [draw=blue, dashed, line width= 1] (e_2) edge (e_3);
\path [draw=blue, dashed, line width= 1] (e_5) edge (f_4);
\path [draw=blue, dashed, line width= 1] (e_6) edge (f_3);
\path [draw=blue, dashed, line width= 1] (f_1) edge (f_2);
\path [draw=blue, dashed, line width= 1] (f_5) edge (f_6);

\path [draw=orange, line width= 1] (c_1) edge[bend left=20,decorate] (c_2);
\path [draw=orange, line width= 1] (c_3) edge[decorate] (c_5);
\path [draw=orange, line width= 1] (c_4) edge[decorate] (c_6);
\path [draw=orange, line width= 1] (c_7) edge[decorate] (f_1);
\path [draw=orange, line width= 1] (c_8) edge[decorate] (f_2);
\path [draw=orange, line width= 1] (f_3) edge[decorate] (f_5);
\path [draw=orange, line width= 1] (f_4) edge[decorate] (f_6);

\end{tikzpicture}
}

\caption{A linear forest decomposition or odd-cover $\mathcal F = \{F_1, F_2, F_3\}$ of a graph with even parity. Each line corresponds to a path with designated endpoints (internal vertices are not depicted and may coincide with other endpoints). Black, solid lines represent $F_1$, blue, dashed lines represent $F_2$, and orange, wavy lines represent $F_3$.}\label{fig:parity}
\end{figure}

Our strategy for proving \Cref{thm:sum of three paths} and \Cref{thm:sum of three cycles} is to convert a linear forest decomposition $\mathcal F = \{F_1, F_2, F_3\}$ into a path (or cycle) odd-cover by finding a matching~$M_{ij}$ with~$V(M_{ij})\subseteq R_{ij}$ so that 
\[\mathcal F' := \{F_1 \cup M_{12} \cup M_{13}, F_2 \cup M_{12} \cup M_{23}, F_3 \cup M_{13} \cup M_{23}\}\]
is a path (or cycle) odd-cover. For this strategy to succeed, it is necessary that~$p(\mathcal F) = 1$ in the path case and~$p(\mathcal F) = 0$ in the cycle case. Indeed, if~$\mathcal F'$ is a path odd-cover, then it is also a linear forest odd-cover with odd parity, and it is easily seen that~$p(\mathcal F) = p(\mathcal F')$. Similarly, if~$\mathcal F'$ is a cycle odd-cover, then because it has no endpoints, each~$M_{ij}$ must be a perfect matching in~$R_{ij}(\mathcal F)$, and so we need~$r_{ij}$ to be even. For the path case, the condition~$p(\mathcal F) = 1$ also turns out to be sufficient, so our proof of \Cref{thm:sum of three paths} in \Cref{sec:path odd-covers} is rather short. For the cycle case, however, we will not only need that~$p(\mathcal F) = 0$, but also that each~$t_i(\mathcal F) > 0$; otherwise, in some~$F_i$, we will have no way to join the components with endpoints in~$R_{ij}$ with those with endpoints in~$R_{ik}$.
Finding such a decomposition~$\mathcal F$ will require a bit more effort.

We begin by verifying that our parity parameter~$p(\mathcal F)$ is well-defined.

\begin{prop}\label{lem:linear forests parity control}
	Let~$G$ be an Eulerian graph of maximum degree~$4$, and let~$\mathcal F = \{F_1, F_2, F_3\}$ be a linear forest odd-cover of~$G$. Then
    \begin{equation}\label{eq:same parity}
        r_{12} \equiv r_{13} \equiv r_{23} \equiv t_1 \equiv t_2 \equiv t_3 \pmod 2.
    \end{equation}
    That is, the parity parameter~$p = p(\mathcal F) \in \{0,1\}$ is well-defined.
\end{prop}
\begin{proof}
	Let~$i,j,k\in \{1,2,3\}$ be pairwise distinct. 
    Note that~$R_{ij}$ and~$R_{ik}$ are disjoint and $R_{ij}\cup R_{ik} = \mathrm{end}(F_i)$. Since every component of~$F_i$ has two endpoints, we have
    \[r_{ij} + r_{ik} = |\mathrm{end}(F_i)| \equiv 0 \pmod 2.\]
    Hence $r_{ij}\equiv r_{ik}\pmod 2$ and, by symmetry, we have $r_{12}\equiv r_{13}\equiv r_{23} \pmod 2$.
    Moreover, $\mathcal{T}_i$ is precisely the set of components of~$F_i$ with an odd number of endpoints in~$R_{ij}$, so
	\[r_{ij} \equiv t_{i} \pmod 2.\]
    This establishes~\eqref{eq:same parity}.
\end{proof}

To find a linear forest decomposition $\mathcal F$ of our graph $G$ with a particular parity $p(\mathcal F)$, we begin by considering a decomposition of $G$ into two polycycles $H_1$ and $H_2$ as in \Cref{cor:Petersen generalization undirected}. We then carefully choose a transversal pair $(M_1, M_2)$ of matchings and apply \Cref{lem:3 linear forests from transversal pair} to obtain $\mathcal F$.  The following lemma shows that $p(\mathcal F)$ is determined by the parity of $|V(M_1) \cap V(M_2)|$. 

\begin{lemma} \label{lem:linearforestendsparity}
    Let~$H_1$ and~$H_2$ be edge-disjoint polycycles, let~$(M_1, M_2)$ be a transversal pair of matchings for~$(H_1, H_2)$, and let~$\mathcal F = \{F_1, F_2, F_3\}$ be a linear forest decomposition of~$H_1 \cup H_2$ such that~$H_1 \setminus M_1 \subseteq F_1 \subseteq (H_1\setminus M_1) \cup M_2$,~$M_1 \subseteq F_2 \subseteq M_1\cup M_2$, and~$H_2 \setminus M_2 = F_3$. 
    Then 
    \begin{align*}
        \mathrm{end}(F_1) &= V(M_1) \oplus V(M_2 \cap F_1),\\
        \mathrm{end}(F_2) &= V(M_1) \oplus V(M_2 \cap F_2),\\
        \mathrm{end}(F_3) &= V(M_2), \text{ and} \\
        p(\mathcal{F}) &\equiv |V(M_1)\cap V(M_2)| \pmod 2.
    \end{align*}
\end{lemma}
\begin{proof}

    Note that for any two edge-disjoint linear forests $L_1$ and $L_2$, if $L_1\cup L_2$ is a linear forest, then $\ennd(L_1\cup L_2) = \ennd(L_1)\oplus \ennd(L_2)$ and, for any matching $M$, we have $\ennd(M)=V(M)$.

    Since $H_1\setminus M_1 \subseteq F_1 \subseteq (H_1\setminus M_1)\cup M_2$, we have $F_1 = (H_1\setminus M_1)\cup (M_2\cap F_1)$.
    Since $H_1$ and $H_2$ are edge-disjoint and $M_1$ is transversal to $H_1$, we have 
    \[\ennd(F_1) = \ennd(H_1\setminus M_1)\oplus \ennd(M_2\cap F_1) = \ennd(M_1)\oplus \ennd(M_2\cap F_1) = V(M_1)\oplus V(M_2\cap F_1).\]
    Since $M_1\subseteq F_2\subseteq M_1\cup M_2$, we have $F_2 = M_1 \cup (M_2 \cap F_2)$ and hence 
    \[\ennd(F_2) = \ennd(M_1)\oplus \ennd(M_2 \cap F_2) = V(M_1)\oplus V(M_2\cap F_2).\]
    Since $F_3 = H_2\setminus M_2$ and $M_2$ is transversal to $H_2$, we clearly have \[\mathrm{end}(F_3) = \ennd(M_2) = V(M_2).\]

    To establish the last statement of this lemma, it suffices to show that~$|V(M_1) \cap V(M_2)| \equiv r_{23} \pmod 2$ (recall that $r_{23}=|\ennd(F_2)\cap \ennd(F_3)|$).
    Consider an edge~$e = \{v_1,v_2\} \in M_2$. 
    
    If $e\notin F_2$, then $v_1,v_2\not\in V(M_2\cap F_2)$ because $M_2$ is a matching. Since $\ennd(F_2) = V(M_1)\oplus V(M_2\cap F_2)$, we have $v_i\in \ennd(F_2)$ if and only if $v_i\in V(M_1)$ for each $i\in \{1,2\}$. This implies that $|\ennd(F_2)\cap e| = |V(M_1)\cap e|$.
    
    Similarly, if $e\in F_2$, then $v_1,v_2\in V(M_2\cap F_2)$, so we have $v_i\in \ennd(F_2)$ if and only if $v_i\not\in V(M_1)$ for each $i\in\{1,2\}$. This implies that~$|\mathrm{end}(F_2) \cap e| = 2 - |V(M_1) \cap e|$.

    In either case, we have
	\[|\mathrm{end}(F_2) \cap e| \equiv |V(M_1) \cap e| \pmod 2.\]
	As this holds for every edge~$e$ in the matching~$M_2$, it follows that
	\begin{align*}
	r_{23} &= |\mathrm{end}(F_2) \cap \mathrm{end}(F_3)| = |\mathrm{end}(F_2) \cap V(M_2)| =\sum_{e\in M_2} |\mathrm{end}(F_2) \cap e|\\
    &\equiv \sum_{e\in M_2} |V(M_1) \cap e|=|V(M_1) \cap V(M_2)| \pmod 2.
    \end{align*}
\end{proof}

\subsection{Path odd-covers}\label{sec:path odd-covers}
In this section, we prove \Cref{thm:sum of three paths}, which states that any Eulerian graph of maximum degree~$4$ is a symmetric difference of at most three paths. From this, we deduce \Cref{thm:path odd-cover general}, following the framework of~\cite{BBCFRY23}. 

Our first goal is to find a linear forest decomposition of size~$3$ with odd parity.
Recall that, given two edge-disjoint polycycles $H_1$ and $H_2$ and a transversal pair $(M_1,M_2)$ of matchings for $(H_1,H_2)$, Lemma \ref{lem:3 linear forests from transversal pair} finds a linear forest decomposition of $H_1\cup H_2$ of size 3, and its parity is equal to the parity of $|V(M_1)\cap V(M_2)|$ by Lemma \ref{lem:linearforestendsparity}. Hence, given $H_1$ and $H_2$, it suffices to find a transversal pair $(M_1,M_2)$ of matchings for $(H_1,H_2)$ with $|V(M_1)\cap V(M_2)|$ odd.

To control this parity, we make use of the following definition. Let $n$ be a positive integer. We say that a cycle $C \subseteq K_n$ is \emph{flexible} with respect to a vertex set $V \subseteq V(K_n)$ if $C$ contains two edges $e_0, e_1$ such that for each $i\in\{0,1\}$, we have
\[|e_i \cap V| \equiv i \pmod 2.\]
We first show that a cycle that intersects a given set $V$ of vertices can be made flexible with respect to $V$ by replacing at most one vertex of $V$.
\begin{lemma}\label{lem:exchange}
    Let $n$ be a positive integer, and let $C \subseteq K_n$ be a cycle. Let $x,z \in V(K_n)$ be two distinct vertices with $x \in V(C)$, and let $V \subseteq V(K_n)$ be a set of vertices containing $x$ but not $z$. Then $C$ is flexible with respect to at least one of $V$ or $V \cup \{z\} \setminus \{x\}$.
\end{lemma}
\begin{proof}
        Suppose that~$C$ is not flexible with respect to $V$. We consider two cases.
        \begin{description}
            \item[Case 1: Every edge of $C$ has even intersection with $V$.] 
            
            Since~$x \in V \cap V(C)$, we have~$V(C) \subseteq V$. Therefore, 
                \[\left(V \cup \{z\} \setminus \{x\}\right) \cap V(C) = V(C) \setminus \{x\}.\]
            Since~$C$ is a cycle (of length at least $3$),~$C$ has an edge~$e_0$ contained in~$V \cup \{z\} \setminus \{x\}$ and this edge satisfies $|e_0 \cap (V\cup \{z\} \setminus \{x\})|\equiv 0\pmod 2$. Moreover, any edge~$e_1$ of $C$ incident to $x$ satisfies $|e_1 \cap (V\cup \{z\} \setminus \{x\})|\equiv 1\pmod 2$.
            Thus~$C$ is flexible with respect to $V \cup \{z\} \setminus \{x\}$.
            \item[Case 2: Every edge of $C$ has odd intersection with $V$.]
            
            In this case,~$C$ is an even cycle and we can write $C=(x_1, y_1, \ldots, x_c, y_c)$ with $c \geq 2$ so that $x=x_1$ and
                \[V \cap V(C) = \{x_i : i \in [c]\}.\]
            Let~$y \in \{y_c, y_1\} \setminus \{z\}$ and~$y' \in \{y_1, y_2\} \setminus \{z\}$, so that $y,y'\not\in V\cup\{z\}\setminus\{x\}$. Then~$e_0 := x_1y$ is disjoint from~$V \cup \{z\} \setminus \{x\}$, and~$e_1 := x_2y'$ intersects~$V \cup \{z\} \setminus \{x\}$ precisely in the vertex~$x_2$. Hence,~$C$ is flexible with respect to $V \cup \{z\} \setminus \{x\}$.
        \end{description}
    \end{proof}

Given edge-disjoint polycycles $H_1$ and $H_2$, we first choose a matching $M_1$ transversal to $H_1$ so that a cycle $D$ of $H_2$ is flexible with respect to $V(M_1)$. We then use this flexibility to choose a matching $M_2$ transversal to $H_2$ so that $|V(M_1)\cap V(M_2)|$ is odd. For a cycle $C$ whose vertices occur in the cyclic order $v_1,v_2,\dots,v_c$, we identify $C$ with the sequence $(v_1,v_2,\dots,v_c)$.

\begin{lemma}\label{lem:matchings with odd intersection}
	For any two polycycles~$H_1$ and~$H_2$ intersecting in at least one vertex, there exists a transversal pair~$(M_1, M_2)$ of matchings for~$(H_1, H_2)$ such that~$|V(M_1) \cap V(M_2)|$ is odd.
\end{lemma}
\begin{proof}
	Since~$V(H_1) \cap V(H_2) \neq \emptyset$, there exist intersecting components~$C = (x_1, \ldots, x_c)$ of~$H_1$ and $D = (y_1, \ldots, y_d)$ of~$H_2$ with~$x_1 = y_1$. Let~$M_1'$ be any matching that is transversal to~$H_1 \setminus C$. Applying \Cref{lem:exchange} with $V := V(M_1') \cup \{x_1, x_2\}$, $x := x_1$, and $z := x_3$, we obtain an edge $e \in \{x_1x_2, x_2x_3\}$ such that, for $M_1 := M_1' \cup \{e\}$, $D$ is flexible with respect to $V(M_1)$.
    Let~$e_0, e_1$ be edges of $D$ such that for each $i\in\{0,1\}$, we have $|e_i \cap V(M_1)|\equiv i\pmod 2$, and let~$M_2'$ be any matching that is transversal to~$H_2 \setminus D$. We can now define
	   \[M_2 := \left\{\begin{array}{l l}
            M_2' \cup \{e_0\} & \text{if } |V(M_1) \cap V(M_2')| \text{ is odd}, \\
            M_2' \cup \{e_1\} & \text{if } |V(M_1) \cap V(M_2')| \text{ is even},
            \end{array}\right.\]
and we have that~$|V(M_1) \cap V(M_2)|$ is odd.
\end{proof}

Combining Lemmas \ref{lem:exchange} and \ref{lem:matchings with odd intersection} yields our first goal of finding a linear forest decomposition of size 3 with odd parity.

\begin{lemma}\label{lem:linear forests good for paths}
    Let $G$ be an Eulerian graph of maximum degree $4$. Then $G$ admits a linear forest decomposition $\mathcal F = \{F_1, F_2, F_3\}$ with $p(\mathcal F) = 1$.
\end{lemma}
\begin{proof}
    By \Cref{cor:Petersen generalization undirected}, there exists a decomposition of~$G$ into two polycycles~$H_1$ and~$H_2$. As~$G$ has maximum degree~$4$,~$V(H_1) \cap V(H_2) \neq \emptyset$, so by \Cref{lem:matchings with odd intersection}, there exists a transversal pair~$(M_1, M_2)$ of matchings for~$(H_1, H_2)$ with~$|V(M_1) \cap V(M_2)|$ odd. Let~$\mathcal F = \{F_1, F_2, F_3\}$ be a linear forest decomposition of~$G = H_1 \cup H_2$ as in \Cref{lem:3 linear forests from transversal pair}. By \Cref{lem:linearforestendsparity}, we have~$p(\mathcal F) = 1$.
\end{proof}

Recall that every linear forest decomposition of a graph $G$ is also a linear forest odd-cover of $G$.
We now show that the linear forest odd-cover obtained in \Cref{lem:linear forests good for paths} can be extended to a path odd-cover.

\begin{proof}[Proof of \Cref{thm:sum of three paths}]
Choose a linear forest odd-cover~$\mathcal F = \{F_1, F_2, F_3\}$ of~$G$ so that~$p=p(\mathcal F) = 1$ and, subject to this condition,~$r_{12} + r_{13} + r_{23}$ is as small as possible. Such an odd-cover exists by \Cref{lem:linear forests good for paths}. We claim that~$\mathcal F$ is in fact a path odd-cover of~$G$. That is, we show that $r_{12}=r_{13}=r_{23}=1$.

Suppose that~$r_{ij} > 1$ for some $i,j\in\{1,2,3\}$, say $r_{12}>1$ without loss of generality. Since $r_{12}$ is odd by \Cref{lem:linear forests parity control}, we have $r_{12}\geq 3$. As~$t_1$ is odd, there exists a vertex~$u \in R_{12}$ such that the other endpoint of the component of~$F_1$ containing~$u$ is in~$R_{13}$. Now since~$r_{12} \geq 3$, there exists~$v \in R_{12}$ that is not in the component of $F_2$ containing $u$. By our choice of~$u$, we have that~$u$ and~$v$ are also in different components of~$F_1$. But now 
\[\mathcal F' := \left\{F_1 \cup \{uv\}, F_2 \cup \{uv\}, F_3\right\}\]
is a linear forest odd-cover of~$G$, with 
\[r_{12}(\mathcal F') + r_{13}(\mathcal F') + r_{23}(\mathcal F') = r_{12}(\mathcal F) + r_{13}(\mathcal F) + r_{23}(\mathcal F) - 2\]
and~$p(\mathcal F') = p(\mathcal F) = 1$, contradicting our choice of~$\mathcal F$. We conclude that~$r_{12} = r_{13} = r_{23} = 1$, so each~$F_i$ is a path, and the proof is complete.
\end{proof}

We will now deduce \Cref{thm:path odd-cover general}. For this, we need the following technical tools of~\cite{BBCFRY23}, which we rephrase slightly for clarity.
\begin{lemma}[Corollary 18 of~\cite{BBCFRY23}]\label{lem:two paths exceptional case}
    Let~$M \subseteq K_n$ be a matching of size~$2$, and let~$H \subseteq K_n$ be a polycycle. If~$H \oplus M$ does not admit a path odd-cover of size at most~$2$, then~$H \oplus M$ contains all but at most one edge of $K_n$ spanned by~$V(M)$.
\end{lemma}
\begin{lemma}[Lemma 19 of~\cite{BBCFRY23}]\label{lem:two good edges}
    Let~$M$ be a matching of size at least~$3$, and let~$H \subseteq K_n$ be a polycycle. Then there exists a submatching~$M' \subseteq M$ of size~$2$ such that~$H\oplus M'$ admits a path odd-cover of size at most~$2$.
\end{lemma}
\begin{lemma}[Lemma 21 of~\cite{BBCFRY23}]\label{lem:technical path prelim}
    Let~$M \subseteq K_n$ be a matching, and define
    \[t := \left\{\begin{array}{c l}
    \left \lceil \frac{|M|}{2} \right \rceil & \text{if~$|M| \neq 2$}, \\
    2 & \text{if~$|M| = 2$.}
    \end{array} \right.\]
    Then for any Eulerian subgraph~$G' \subseteq K_n$ of maximum degree at most~$2t$, the graph~$G := G' \oplus M$ admits a path odd-cover of size at most~$2t$.
\end{lemma}
\begin{proof}[Proof of \Cref{thm:path odd-cover general}]
    We prove the Eulerian case first. By \Cref{cor:Petersen generalization undirected}, there exists a polycycle decomposition~$\{H_1, \ldots, H_k\}$ of~$G$ with~$k := \frac{\Delta(G)}{2}$. By \Cref{thm:sum of three paths}, we have for each~$i \leq \left \lfloor\frac{k}{2} \right \rfloor$ that~$H_{2i-1} \cup H_{2i}$ admits a path odd-cover of size at most 3. If~$k$ is even, this gives a path odd-cover of~$G$ of size at most~$3\frac{k}{2} = \frac{3}{4}\Delta$. If~$k$ is odd, then we also need to take care of the last polycycle~$H_k$, which admits a path odd-cover of size at most 2 by \Cref{lem:polycycle two paths or cycles}; altogether, this gives a path odd-cover of~$G$ of size at most~$3\frac{k-1}{2} + 2 = \left \lceil \frac{3}{4}\Delta(G) \right \rceil.$

    Now we consider the general case.
    Let~$M$ be a matching such that~$V(M)$ is the set of odd-degree vertices of~$G$, so we have~$|M| = \frac{\vodd(G)}{2}$.
    Then~$G' := G \oplus M$ is Eulerian with (even) maximum degree~$\Delta' \leq \Delta_e(G)$. 
    By \Cref{cor:Petersen generalization undirected}, there exists a decomposition~$\{H_1, \ldots, H_k\}$ of~$G'$ into polycycles, where~$k := \frac{\Delta'}{2}$. Let~$M' \subseteq M$ be a submatching of size~$|M'| = \min \{|M|, 2k\}$. Define 
    \[t := \left\{\begin{array}{c l}
    \left \lceil \frac{|M'|}{2} \right \rceil & \text{if~$|M'| \neq 2$}, \\
    2 & \text{if~$|M'| = 2$.}
    \end{array} \right.\]
    We will find our desired odd-cover in slightly different ways depending on the relationship between the parameters $|M'|$, $t$, and $k$. We consider each case separately.

    \begin{description}
        \item[Case 1:~$t \leq k$.] Then by \Cref{lem:technical path prelim},~$H_1 \oplus \cdots \oplus H_t \oplus M'$ admits a path odd-cover of size at most~$2t$, and by the Eulerian case,~$H_{t+1} \oplus \cdots \oplus H_k$ admits a path odd-cover of size at most~$\left \lceil \frac{3}{2}(k-t)\right \rceil$. Clearly,~$M \setminus M'$ admits a path odd-cover of size~$|M| -  |M'|$ by taking each edge as a path. Altogether, this gives a path odd-cover~$\mathcal P$ of~$G = G' \oplus M = H_1 \oplus \cdots \oplus H_k \oplus M' \oplus (M \setminus M')$ of size
    \[|\mathcal P| \leq |M| - |M'| + 2t + \left \lceil \frac{3}{2}(k-t)\right \rceil.\]
    \begin{description}
        \item[Subcase 1a:~$|M'| \neq 2$ is even.] Then~$|M'| = 2t$, and we have~$t = \min\{k, |M|/2\}$, hence~$\mathcal P$ has size
    \[|\mathcal P| \leq \max\left\{|M|, |M| + \left \lceil \frac{3}{2}\left(k-\frac{|M|}{2}\right)\right \rceil\right\} = \max\left\{|M|, \left \lceil \frac{|M| + 3\Delta'}{4}\right \rceil\right\}.\]
    \item[Subcase 1b:~$|M'| \neq 2$ is odd.] Then~$|M'| = 2t - 1$, and we must have~$|M'| = |M|$, hence~$\mathcal P$ has size
    \[|\mathcal P| \leq |M| + 1 + \left \lceil \frac{3}{2}\left(k-\frac{|M| + 1}{2}\right)\right \rceil = \left \lceil \frac{|M| + 3\Delta' + 1}{4}\right \rceil.\]
    Since~$|M|$ is odd and~$\Delta'$ is even, this is equal to~$\left \lceil \frac{|M| + 3\Delta'}{4}\right \rceil$.
    \item[Subcase 1c:~$|M'| = 2$.] Then~$k \geq t = 2$, hence~$|M| = |M'| = 2$ by the definition of~$M'$. As~$\mathcal P$ may be larger than we require, we consider a separate odd-cover~$\mathcal P'$ of~$G$, as follows. Since~$H_1$ and~$H_2$ are edge-disjoint, at most one of~$H_1 \oplus M$ or~$H_2 \oplus M$ contains all but one edge spanned by~$V(M)$, so by \Cref{lem:two paths exceptional case}, at least one of~$H_1 \oplus M$ or~$H_2 \oplus M$ admits a path odd-cover of size at most~$2$. Without loss of generality, assume~$H_1 \oplus M$ is a symmetric difference of at most two paths. By the Eulerian case,~$H_2 \oplus \cdots \oplus H_k$ admits a path odd-cover of size at most~$\left\lceil \frac{3}{2}(k-1) \right\rceil$. Altogether, this gives a path odd-cover~$\mathcal P'$ of~$G$ of size
        \[|\mathcal P'| \leq 2 + \left \lceil \frac{3}{2}(k-1)\right \rceil = \left \lceil \frac{|M| + 3\Delta'}{4}\right \rceil.\]
    \end{description}
    \item[Case 2:~$t > k$.] Then we necessarily have~$t=2$,~$k=1$, and~$|M'| = 2$. 
    \begin{description}
        \item[Subcase 2a:~$M' = M$.] Then~$G = H_1 \oplus M$ is the symmetric difference of a polycycle and two disjoint edges. If~$\Delta(G) \leq 2$, then~$G$ cannot contain five edges spanned by~$V(M)$, so by \Cref{lem:two paths exceptional case},~$G$ admits a path odd-cover of size at most~$2 = \frac{\vodd(G)}{2}$. If~$\Delta(G) \geq 3$, then~$G$ admits a path odd-cover of size at most~$4 \leq \left \lceil \frac{\vodd(G)/2 + 3\Delta_e(G)}{4}\right \rceil$ by \Cref{lem:technical path prelim}.
        \item[Subcase 2b:~$M' \neq M$.] Then~$|M| \geq 3$, and by \Cref{lem:two good edges}, there exists~$M'' \subseteq M$ of size~$|M''| = 2$ such that~$H_1 + M''$ admits a path odd-cover of size at most~$2$. We now extend this to a path odd-cover of~$G = (H_1 \oplus M'') \oplus (M \setminus M'')$ of size at most~$|M| = \frac{\vodd(G)}{2}$ by taking each of the~$|M| - 2$ edges in~$M \setminus M''$ as a path in itself.
    \end{description}
    \end{description}
    We have now considered all cases. In each of these, we found a path odd-cover of~$G$ of size at most 
    \[\max \left\{\frac{\vodd(G)}{2}, \left \lceil\frac{\vodd(G)/2 + 3\Delta_e(G)}{4} \right \rceil\right\}.\]
\end{proof}

\subsection{Cycle odd-covers} \label{sec:oddcycle}

In this section, we prove \Cref{thm:sum of three cycles}, which says that every Eulerian graph of maximum degree~$4$ admits a cycle odd-cover of size 3. \Cref{thm:cycle odd-cover general} will quickly follow. As we will see, the proof of \Cref{thm:sum of three cycles} is quite a bit more technical than the path case, though our general approach is the same.
As noted in our strategy outline in Section \ref{sec:refinelinearforest}, the main difference is that when we obtain our linear forest decomposition~${\mathcal F = \{F_1, F_2, F_3\}}$, it is not enough to control the parity parameter~$p(\mathcal{F})$, but we also need that each~$\mathcal T_i$ is nonempty.
For this, we prove a version of \Cref{lem:matchings with odd intersection} for even parity which imposes an additional property on our transversal pair~$(M_1, M_2)$. 

\begin{lemma}\label{lem:matchings with even intersection}
	Let~$H_1$ and~$H_2$ be edge-disjoint polycycles such that there exist distinct components~$C_1, C_1'$ of~$H_1$ and~$C_2, C_2'$ of~$H_2$ such that~$V(C_1) \cap V(C_2)$ and~$V(C_1') \cap V(C_2')$ are both nonempty. Then there exists a transversal pair~$(M_1, M_2)$ of matchings for~$(H_1, H_2)$ such that
	\begin{enumerate}[label=(\roman*)]
		\item~$|V(M_1) \cap V(M_2)|$ is even;
		\item there exist edges~$uv_1 \in M_1$ and~$uv_2 \in M_2$ such that~$v_1 \notin V(M_2)$ and~$v_2 \notin V(M_1)$.
	\end{enumerate}
\end{lemma}
\begin{proof}
    We consider a number of cases based on the way in which the polycycles overlap.
    \begin{description}
        \item[Case 1: There exists an edge~$uv_1 \in E(C_1)$ with~$u \in V(C_2)$ and~$v_1 \notin V(C_2')$.]
        It suffices to produce a matching~$M_1$ transversal to~$H_1$ such that~$uv_1 \in M_1$,~$V(M_1)$ excludes some neighbor~$v_2$ of~$u$ in~$C_2$, and~$C_2'$ is flexible with respect to~$V(M_1)$. Indeed, let~$e_0, e_1 \in E(C_2')$ be as in the definition of flexible. Let~$M_2'$ be any matching that is transversal to~$H_2 \setminus C_2'$ with~$uv_2 \in M_2'$ and~$v_1 \notin V(M_2')$ (recall that~$H_1$ and~$H_2$ are edge-disjoint, so~$v_1 \neq v_2$). We can now define
	           \[M_2 := \left\{\begin{array}{l l}
                    M_2' \cup \{e_0\} & \text{if } |V(M_1) \cap V(M_2')| \text{ is even}, \\
                    M_2' \cup \{e_1\} & \text{if } |V(M_1) \cap V(M_2')| \text{ is odd},
                    \end{array}\right.\]
            and we have that~$|V(M_1) \cap V(M_2)|$ is even. We also have that~${v_1 \notin V(M_2)}$ since we assumed that~${v_1 \notin V(C_2')}$.
        
        To obtain such a matching~$M_1$, we further distinguish subcases as follows.
        \begin{description}
            \item[Subcase 1a: There exists a neighbor~$v_2$ of~$u$ in~$C_2$ such that~$v_2 \notin V(C_1')$.]
            Since~$C_1'$ and~$C_2'$ intersect, there exists a path~$(x,y,z)$ in~$C_1'$ with~$x \in V(C_2')$. Let~$M_1'$ be a matching transversal to~$H_1 \setminus C_1'$ such that~$uv_1 \in M_1'$ and~$v_2 \notin V(M_1')$ (recall that~$H_1$ and~$H_2$ are edge-disjoint, so~$v_1 \neq v_2$). By \Cref{lem:exchange}, there exists~$e \in \{xy, yz\}$ such that for~$M_1 := M_1' \cup \{e\}$,~$C_2'$ is flexible with respect to~$V(M_1)$. We also have that~$v_2 \notin V(M_1)$ since we assumed that~$v_2 \notin V(C_1')$.
            \item[Subcase 1b: Both neighbors~$v_2$ and~$v_2'$ of~$u$ in~$C_2$ are in~$V(C_1')$.]
            We may assume the additional property that every component~$C_1''$ of~$H_1$ besides~$C_1$ and~$C_1'$ is disjoint from~$C_2'$, as otherwise we could apply Case 1a with~$C_1''$ in place of~$C_1'$. Since~$v_2, v_2' \in V(C_1') \setminus V(C_2')$, and ~$V(C_1') \cap V(C_2')$ is nonempty, there exists an edge~$xy \in E(C_1')$ with~$V(C_2') \cap \{x,y\} = \{x\}$. Without loss of generality, assume~$y \neq v_2$. Let~$M_1$ be any matching transversal to~$H_1$ with~$uv_1, xy \in M_1$ and~$v_2 \notin V(M_1)$. Since~$u \in V(C_2)$ and~$v_1 \notin V(C_2')$, we have that~$V(C_2') \cap V(M_1) = \{x\}$, so~$C_2'$ is flexible with respect to~$V(M_1)$.
        \end{description}
        \item[Case 2: Every edge~$uv \in E(C_1)$ with~$u \in V(C_2)$ has~$v \in V(C_2')$.] We will show that we can always reduce our task to Case 1 unless $H_1$ and $H_2$ intersect in a very particular way. First, we may assume that every edge of the form $uw \in E(C_2)$ with $u \in V(C_1)$ has $w \in V(C_1')$; otherwise, we could apply Case 1 with the roles of $H_1$ and $H_2$ interchanged. In particular, since $V(C_1) \cap V(C_2) \neq \emptyset$, we see that $V(C_1) \cap V(C_2')$ and $V(C_1') \cap V(C_2)$ are both nonempty as well. Therefore, we may also assume that every edge of the form~${vu \in E(C_1)}$ with~${v \in V(C_2')}$ has~$u \in V(C_2)$; otherwise, we could apply Case 1 with the roles of~$C_2$ and~$C_2'$ interchanged.
        This implies that~$C_1$ is an even cycle of the form~$(u_1, v_1, \ldots, u_c, v_c)$ with each~$u_i \in V(C_2)$ and each~$v_i \in V(C_2')$. By the same argument applied with the roles of~$C_1$ and~$C_1'$ interchanged, we may assume that~$C_1'$ is also an even cycle of the form~$(w_1, x_1, \ldots, w_{c'}, x_{c'})$ with each~$w_i \in V(C_2)$ and each~$x_i \in V(C_2')$. Now, by applying the same argument again with the roles of~$H_1$ and~$H_2$ interchanged, we may further assume that~$C_2$ is an even cycle of length~$2c=2c'$ alternating between~$\{u_i : i \in [c]\}$ and~$\{w_i : i \in [c']\}$, and~$C_2'$ is an even cycle of the same length alternating between~$\{v_i : i \in [c]\}$ and~$\{x_i : i \in [c']\}$. Lastly, we may assume that
        \[V(H_1) \cap V(H_2) = V(C_1) \cup V(C_1') = V(C_2) \cup V(C_2');\]
        otherwise, there exist intersecting components~$C_1'' \notin \{C_1, C_1'\}$ of~$H_1$ and~$C_2'' \notin \{C_2, C_2'\}$ of~$H_2$, and we could apply Case 1 with~$(C_1'', C_2'')$ in place of~$(C_1', C_2')$.
        
        Let~$w$ be a neighbor of~$u_1$ in~$C_2$ distinct from~$w_1$, and let~$v$ be a neighbor of~$x_1$ in~$C_2'$ distinct from~$v_1$. We may now choose~$(M_1, M_2)$ to be any transveral pair of matchings for~$(H_1, H_2)$ with~$u_1v_1, w_1x_1 \in M_1$ and~$u_1w, vx_1 \in M_2$. Then~$V(M_1) \cap V(M_2) = \{u_1, x_1\}$ has even size, and (ii) is satisfied with~$(u_1, v_1, w)$ in place of~$(u, v_1, v_2)$.
    \end{description}
    We have now considered all cases, and the proof is complete.
\end{proof}

We now show that applying \Cref{lem:3 linear forests from transversal pair} to the transversal pair obtained in \Cref{lem:matchings with even intersection} yields a linear forest decomposition with the properties we require.

\begin{lemma}\label{lem:linear forests good for cycles}
	Let~$G$ be an Eulerian graph of maximum degree~$4$. Then~$G$ either
	\begin{enumerate}[label=(\roman*)]
		\item admits a linear forest decomposition~$\mathcal F = \{F_1, F_2, F_3\}$ with~$p(\mathcal F) = 0$ and~$t_i(\mathcal F) \geq 2$ for every~$1 \leq i \leq 3$, or
		\item admits a decomposition into a polycycle and a cycle.
	\end{enumerate}
\end{lemma}
\begin{proof}
	By \Cref{cor:Petersen generalization undirected}, there exists a decomposition of~$G$ into two polycycles~$H_1$ and~$H_2$.
	
	If there do \textit{not} exist components~$C_1, C_1'$ of~$H_1$ and~$C_2, C_2'$ of~$H_2$ such that~$V(C_1) \cap V(C_2)$ and~$V(C_1') \cap V(C_2')$ are both nonempty, then there exists~$i \in \{1,2\}$ and a component~$C$ of~$H_i$ such that each of the remaining components of~$H_i$ is disjoint from~$H_{3-i}$. Then~$H_{3-i} \cup (H_i \setminus C)$ is a polycycle whose union with the cycle~$C$ gives~$G$.
	
	We henceforth assume that such components~$C_1, C_1'$ of~$H_1$ and~$C_2, C_2'$ of~$H_2$ do exist. By \Cref{lem:matchings with even intersection}, there exists a transversal pair $(M_1,M_2)$ of matchings for $(H_1,H_2)$ such that $|V(M_1)\cap V(M_2)|$ is even and there exist edges $uv_1\in M_1$ and $uv_2\in M_2$ such that $v_1\not\in V(M_2)$ and $v_2\not\in V(M_1)$.
    By \Cref{lem:3 linear forests from transversal pair}, there exists a linear forest decomposition~$\mathcal F = \{F_1, F_2, F_3\}$ of $H_1\cup H_2$ such that~$H_1 \setminus M_1 \subseteq F_1 \subseteq (H_1\setminus M_1)\cup M_2$,~$M_1 \subseteq F_2 \subseteq M_1\cup M_2$, and~$H_2 \setminus M_2 = F_3$.
    Since~$|V(M_1) \cap V(M_2)|$ is even, we also have that~$p(\mathcal F) = 0$ by \Cref{lem:linearforestendsparity}. Thus it remains to check that each~$t_i > 0$.
	
    Let~$D_1$ be the component of~$H_1$ containing~$uv_1$, and let~$D_2$ be the component of~$H_2$ containing~$uv_2$. For each~$i \in \{1,2,3\}$, define~$P_i := F_i \cap (D_1 \cup D_2)$. To show that each $t_i>0$, we prove the following claim.
    \begin{claim*}
        For each~$i \in \{1,2,3\}$,~$P_i$ is a component of~$F_i$, and~$P_i \in \mathcal T_i$.
    \end{claim*}
    \begin{subproof}
        Define~$M_2' := M_2 \cap F_1$ and~$M_2'' := M_2 \cap F_2$, so that~$F_1 = (H_1 \setminus M_1) \cup M_2'$ and~$F_2 = M_1 \cup M_2''$. Then~$P_1$ consists of the path~$D_1 \setminus \{uv_1\}$ and possibly the edge~$uv_2$, and~$P_2$ consists of the edge~$uv_1$ and possibly the edge~$uv_2$. Moreover, since $F_3=H_2\setminus M_2$,~$P_3$ is precisely the path~$D_2 \setminus \{uv_2\}$. Clearly, for each $i\in\{1,2,3\}$,~$P_i$ is connected and, since~$F_i$ is a linear forest,~$P_i$ is a path with endpoints in~$\{u, v_1, v_2\}$. To show that~$P_i$ is a component of~$F_i$, it suffices to show that~$\mathrm{end}(P_i) \subseteq \mathrm{end}(F_i)$.

        First, we recall that~$\mathrm{end}(F_1) = V(M_1) \oplus V(M_2')$,~$\mathrm{end}(F_2) = V(M_1) \oplus V(M_2'')$, and~$\mathrm{end}(F_3) = V(M_2)$ by \Cref{lem:linearforestendsparity}. This immediately gives~$v_1 \in \mathrm{end}(F_1) \cap \mathrm{end}(F_2)$ since~$v_1 \in V(M_1) \setminus V(M_2)$, and also $\mathrm{end}(P_3) = \{u,v_2\} \subseteq V(M_2) = \mathrm{end}(F_3)$. The nature of the other endpoints of~$P_1$ and~$P_2$ depends on whether~$uv_2$ is in~$M_2'$ or~$M_2''$, so we consider each case separately.
        \begin{description}
        \item[Case 1:~$uv_2 \in M_2'$.] Then~$P_1 = (D_1 \setminus \{uv_1\}) \cup \{uv_2\}$, and~$P_2 = \{uv_1\}$. The endpoints of~$P_1$ are~$v_1$ and~$v_2$, and the endpoints of~$P_2$ are~$u$ and~$v_1$. Then~$v_2 \in V(M_2') \setminus V(M_1)$ is an endpoint of~$F_1$, and~$u \in V(M_1) \setminus V(M_2'')$ is an endpoint of~$F_2$. That is,~$\mathrm{end}(P_1) = \{v_1, v_2\} \subseteq \mathrm{end}(F_1)$, and~$\mathrm{end}(P_2) = \{u,v_1\} \subseteq \mathrm{end}(F_2)$. Now~$v_1 \in R_{12}$,~$v_2 \in R_{13}$, and~$u \in R_{23}$, so~$P_i \in \mathcal T_i$ for each~$i\in\{1,2,3\}$.
        \item[Case 2:~$uv_2 \in M_2''$.] Then~$P_1 = D_1 \setminus \{uv_1\}$, and~$P_2 = \{uv_1, uv_2\}$. The endpoints of~$P_1$ are~$v_1$ and~$u$, and the endpoints of~$P_2$ are~$v_1$ and~$v_2$. Then~$u \in V(M_1) \setminus V(M_2')$ is an endpoint of~$F_1$, and~$v_2 \in V(M_2'') \setminus V(M_1)$ is an endpoint of~$F_2$. That is,~$\mathrm{end}(P_1) = \{u, v_1\} \subseteq \mathrm{end}(F_1)$, and~$\mathrm{end}(P_2) = \{v_1, v_2\} \subseteq \mathrm{end}(F_2)$. Now~$v_1 \in R_{12}$,~$u \in R_{13}$, and~$v_2 \in R_{23}$, so~$P_i \in \mathcal T_i$ for each~$i\in\{1,2,3\}$.
    \end{description}
    This completes the proof of the claim. 
    \end{subproof}
    Since $P_i\in\mathcal T_i$, we have $t_i>0$, for each $i\in\{1,2,3\}$. This completes the proof of the lemma.
\end{proof}

We are now ready to prove that every Eulerian graph of maximum degree~$4$ is a symmetric difference of at most~$3$ cycles.

\begin{proof}[Proof of \Cref{thm:sum of three cycles}]
    As any polycycle is a symmetric difference of at most two cycles by \Cref{lem:polycycle two paths or cycles}, if~$G$ admits a decomposition into a polycycle and a cycle, then~$G$ admits a cycle odd-cover of size at most~$3$. Therefore, we may assume by \Cref{lem:linear forests good for cycles} that~$G$ admits a linear forest odd-cover~$\mathcal F = \{F_1, F_2, F_3\}$ with~$p(\mathcal F) = 0$ and~$t_i > 0$ for every~$i\in\{1,2,3\}$. Choose such a linear forest odd-cover~$\mathcal F$ so that~$r_{12} + r_{13} + r_{23}$ is as small as possible.
    
    \begin{claim*}
        We have~$r_{12} = r_{13} = r_{23} = 2$.
    \end{claim*}
    \begin{subproof}
        By \Cref{lem:linear forests parity control}, each~$r_{ij}$ is even, and~$r_{ij} \geq t_i > 0$, so~$r_{ij} \geq 2$. Suppose that some~$r_{ij} > 2$, say~$r_{12} \geq 4$ without loss of generality. We also know that~$t_1, t_2 \geq 2$ by \Cref{lem:linear forests parity control}. Let~$T_1, T_1'$ be distinct components of~$F_1$ in~$\mathcal T_1$, and let~$u, x_1$ be the respective endpoints of~$T_1, T_1'$ in~$R_{12}$. Choose~$T_2 \in \mathcal T_2$ so that its endpoint~$x_2$ in~$R_{12}$ is distinct from~$u$. Since~$u \in R_{12}$, there exists a component of~$F_2$ with one endpoint at~$u$ and one endpoint at some other vertex~$w$. Choose a vertex~$v \in R_{12}$ as follows. If~$w \in R_{12}$, choose~$v \in R_{12} \setminus \{u,x_1, w\}$; otherwise, choose~$v \in R_{12} \setminus \{u,x_1, x_2\}$. Note that~$u$ and~$v$ are both endpoints in distinct components of~$F_1$ and of~$F_2$, so 
    \[\mathcal F' := \{F_1 \cup \{uv\}, F_2 \cup \{uv\}, F_3\}\]
    is a linear forest odd-cover of~$G$ with
    \[r_{12}(\mathcal F') + r_{13}(\mathcal F') + r_{23}(\mathcal F') = r_{12}(\mathcal F) + r_{13}(\mathcal F) + r_{23}(\mathcal F) - 2\]
    and~$p(\mathcal F') = p(\mathcal F) = 0$. Also, we have~$T_1' \in \mathcal T_1(\mathcal F')$ and~$\mathcal T_3(\mathcal F') = \mathcal T_3(\mathcal F) \neq \emptyset$ since~$u$ and~$v$ are not endpoints of~$T_1'$ nor any of the paths in~$\mathcal T_3(\mathcal F)$. Less obviously, we additionally have~$\mathcal T_2(\mathcal F') \neq \emptyset$. Indeed, recall that~$x_2$ is the endpoint of~$T_2$ in~$R_{12}$, and that~$x_2 \neq u$. If~$x_2 \neq v$ as well, then~$T_2 \in \mathcal T_2(\mathcal F')$. On the other hand, if~$x_2 = v$, then by our choice of~$v$, we must have that~$w \in R_{12}$, so~$u$ is not the endpoint of any path in~$\mathcal T_2(\mathcal F)$. Therefore, since~$t_2(\mathcal F) \geq 2$, there exists~$T_2' \in \mathcal T_2(\mathcal F)$ whose endpoints are distinct from~$u$ and~$v$, so~$T_2' \in \mathcal T_2(\mathcal F')$.
    But now~$t_i(\mathcal F') > 0$ for all~$i\in\{1,2,3\}$, contradicting our choice of~$\mathcal F$. This proves the claim.
    \end{subproof}
    
    For each pair~$\{i,j\} \subseteq \{1,2,3\}$, write~$R_{ij} = \{x_{ij}, y_{ij}\}$ and let~$e_{ij} := x_{ij}y_{ij}$. Then
    \[\left\{F_1 \cup \{e_{12}, e_{13}\}, F_2 \cup \{e_{12}, e_{23}\}, F_3 \cup \{e_{13}, e_{23}\}\right\}\]
    is a cycle odd-cover of~$G$.
\end{proof}

\Cref{thm:cycle odd-cover general} now follows readily.
\begin{proof}[Proof of \Cref{thm:cycle odd-cover general}]
    By \Cref{cor:Petersen generalization undirected} with $\Delta := d_1$ and $t := d_2/2$, there exists a polycycle decomposition~$\{H_1, \ldots, H_{d_1/2}\}$ of~$G$ such that $H_{d_2/2 + 1}, \ldots, H_{d_1/2}$ are cycles already. By \Cref{thm:sum of three cycles}, we have for each~$i \leq \lfloor d_2/4 \rfloor$ that~$H_{2i-1} + H_{2i}$ is the symmetric difference of at most~$3$ cycles. If~$d_2/2$ is even, this gives a cycle odd-cover of~$G$ of size at most~$3d_2/4 + (d_1/2 - d_2/2) = d_1/2 + d_2/4$. If~$d_2/2$ is odd, then we also need to take care of the last polycycle~$H_{d_2/2}$, which is a symmetric difference of at most~$2$ cycles by \Cref{lem:polycycle two paths or cycles}. Altogether, this gives a cycle odd-cover of~$G$ of size at most
    \[3\frac{d_2/2 - 1}{2} + 2 + \left(\frac{d_1}{2} - \frac{d_2}{2}\right) = \frac{d_1}{2} + \left \lceil \frac{d_2}{4}\right \rceil.\]
\end{proof}

\section{Concluding remark}\label{sec:conclusion}

    We obtained improved bounds on the combinatorial diameters of partition polytopes (Theorem \ref{thm:improved upper bound}) and on minimum path and cycle odd-covers of  graphs (Theorem \ref{thm:odd cover general}).
    In both cases, our proofs use the decomposition of Eulerian (di)graphs into polycycles (Section \ref{sec:polycycles}) and our main technical contribution is the ``decomposition'' of the union of two edge-disjoint polycycles into three cycles; for partition polytopes, this was phrased as decomposing the product of two $p$-balanced permutations with disjoint supports as a product of three $p$-cycles (Lemma \ref{lem:two permutations three cycles}), and for odd-covers, we expressed the union of two edge-disjoint polycycles as a symmetric difference of three paths and a symmetric difference of three cycles (Theorem \ref{thm:sum of three paths} and Theorem \ref{thm:sum of three cycles}, respectively). 
    
    These latter proofs themselves share similarities in that they are based on the proof of Akiyama, Exoo, and Harary~\cite{AHE81} that every graph with maximum degree 4 has linear arboricity at most 3 (Lemma \ref{lem:3 linear forests from transversal pair}), though this connection is disguised in our proof of Lemma \ref{lem:two permutations three cycles}.
    This naturally leads to the question of whether there is a general framework that encompasses both the diameter problem for partition polytopes and the cycle (and path) odd-cover problem for Eulerian graphs with bounded maximum degree, and whether our proofs can be unified in such a framework. We leave this as an open question.
    
\vspace*{2cm}

\noindent\textbf{Acknowledgments.} The work of Steffen Borgwardt and Abigail Nix was supported by the Air Force Office of Scientific Research [grants FA9550-21-1-0233 and FA9550-24-1-0240] (Complex Networks).  The work of Zden\v{e}k Dvo\v{r}\'ak was supported by the ERC-CZ project LL2005 (Algorithms and complexity within and beyond bounded expansion) of the Ministry of Education of Czech Republic. The work of Bryce Frederickson was supported by the NSF Graduate Research Fellowship Program [grant 1937971] and the National Science Foundation [grant 2247013] (Forbidden and Colored Subgraphs).


\end{document}